\documentclass{article}
\usepackage{geometry}
\usepackage{mythmsub}
\usepackage{amsthm}
\usepackage{tikz-cd}
\usepackage[shortlabels]{enumitem}
\usepackage[bookmarks]{hyperref}
\usepackage{titlesec}
\numberwithin{equation}{thm}
\begin{document}
\newcommand{\eff}[0]{{e\!f\!f}}
\newcommand{\dmeff}[0]{{\rm DM}^\eff}
\newcommand{\dmeffloc}[0]{{\rm DM}^{\eff,loc}}
\title{Construction of triangulated categories of motives using the localization property}
\author{Doosung Park}
\date{}
\maketitle
\begin{abstract}
Using the localization property, we construct a triangulated category of motives over quasi-projective $T$-schemes for any coefficient where $T$ is a noetherian separated scheme, and we prove the Grothendieck six operations formalism. We also construct integral \'etale realization of motives.
\end{abstract}
\titlelabel{\thetitle.\;\;}
\titleformat*{\section}{\center \bf}
\section{Introduction}
\begin{none}\label{0.1}
  Throughout this paper, fix a ring $\Lambda$, and fix a noetherian separated scheme $T$. Let
  \begin{enumerate}[(1)]
    \item $\mathscr{S}$ (resp.\ $\mathscr{S}^{sm}$) denote the category of schemes quasi-projective (resp.\ quasi-projective and smooth) over $T$,
    \item ${\rm Tri}^\otimes$ denote the $2$-category of symmetric monoidal triangulated categories where $1$-morphisms are symmetric monoidal functors and $2$-morphisms are symmetric monoidal natural transformations,
    \item $Sm$ denote the class of smooth morphisms of schemes.
  \end{enumerate}
\end{none}
\begin{none}\label{0.2}
  According to \cite{Ayo07}, when $\Lambda$ is a $\mathbb{Q}$-algebra, we have triangulated categories of motives over schemes satisfying the Grothendieck six operations formalism in \cite[2.4.50]{CD12}. However, for general $\Lambda$, such a construction have not been available yet according to our knowledge.
\end{none}
\begin{none}\label{0.3}
  In this paper, we define ${\rm DM}^{loc}$, which is a successful construction of triangulated categories of motives over quasi-projective $T$-schemes with all coefficients satisfying the Grothendieck six operations formalism. Our idea is as follows. For any object $S$ of $\mathscr{S}^{sm}$, the construction ${\rm DM}(S,\Lambda)$ in \cite[11.1.1]{CD12} seems to enjoy many useful properties, so put
  \[{\rm DM}^{loc}(S,\Lambda)={\rm DM}(S,\Lambda)\]
  for such an $S$.

  Assume that we can extend the above ${\rm DM}^{loc}(-,\Lambda)$ to $\mathscr{S}$ such that the Grothendieck six operations formalism is satisfied. Let $i:Z\rightarrow S$ be a closed immersion in $\mathscr{S}$, and let $j:U\rightarrow S$ denote its complement. Then from the formalism, $i_*$ should be fully faithful, and we should have the distinguished triangle
  \[j_\sharp j^*\longrightarrow {\rm id}\longrightarrow i_*i^*\longrightarrow j_\sharp j^*[1].\]
  Thus we see that ${\rm DM}^{loc}(Z,\Lambda)$ should be the full subcategory of ${\rm DM}^{loc}(S,\Lambda)$ consisting of objects $K$ of ${\rm DM}^{loc}(S,\Lambda)$ such that $j^*K=0$.

  We can consider it as our definition of triangulated categories of motives. A problem is that the definition itself is dependent on the choice of the closed immersion $i:Z\rightarrow S$, so we need to show that the definition is independent of the choice of $i$. We also need to show that ${\rm DM}^{loc}(-,\Lambda)$ satisfies the Grothendieck six operations formalism. We solve these problems by proving the following theorems.
\end{none}
\begin{thm}\label{0.4}
  {\rm ((\ref{6.3}) in the text)} Consider a diagram
  \[\begin{tikzcd}
    \mathscr{S}^{sm}\arrow[rd,"V"']\arrow[rr,"W"]&&\mathscr{S}\\
    &{\rm Tri}^{\otimes}
  \end{tikzcd}\]
  of $2$-categories where
  \begin{enumerate}[{\rm (i)}]
    \item $V$ is a $Sm$-premotivic pseudofunctor satisfying {\rm (Loc)} {\rm (}see {\rm (\ref{3.2})} and {\rm (\ref{2.6})} for the definitions{\rm )},
    \item $W$ denotes the inclusion functor.
  \end{enumerate}
  Then there is a contravariant pseudofunctor $H:\mathscr{S}\rightarrow {\rm Tri}^{\otimes}$ satisfying {\rm (Loc)} and a pseudonatural equivalence $\kappa:V\rightarrow H\circ W$ making the above diagram commutes.

  The above construction is functorial in the following sense. Suppose that we have a diagram
  \[\begin{tikzcd}
    \mathscr{S}^{sm}\arrow[rd,"V'"]\arrow[rd,bend right,"V"']\arrow[rr,"W"]&&\mathscr{S}\arrow[ld,"H"']\arrow[ld,"H'",bend left]\\
    &{\rm Tri}^{\otimes}&\;
  \end{tikzcd}\]
  of $2$-categories and a diagram
  \begin{equation}\label{0.4.1}\begin{tikzcd}
    V\arrow[d,"\epsilon"]\arrow[r,"\kappa"]&H\circ W\\
    V'\arrow[r,"\kappa'"]&H'\circ W
  \end{tikzcd}\end{equation}
  of pseudofunctors where $V'$ and $H'$ are $Sm$-premotivic pseudofunctors satisfying {\rm (Loc)}, $\kappa'$ is a pseudonatural equivalence, and $\epsilon$ is a $Sm$-premotivic pseudonatural transformation {\rm (}see {\rm (\ref{3.5})} for the definition{\rm )}. Then there exists a $Sm$-premotivic pseudonatural transformation $H\rightarrow H'$ unique up to isomorphism such that the induced $Sm$-pseudonatural transformation $H\circ W\rightarrow H'\circ W$ makes the diagram {\rm (\ref{0.4.1})} commutative.
\end{thm}
\begin{none}\label{0.5}
  Note that in (\ref{0.4}), the functoriality implies that $H$ is unique up to pseudonatural equivalence.
\end{none}
\begin{thm}\label{0.7}
  {\rm ((\ref{7.2}) in the text)} In {\rm (\ref{0.4})}, if $V$ satisfies the axioms from {\rm (B--5)} to {\rm (B--8)} in {\rm (2.10)}, then $H$ satisfies the Grothendieck six operations formalism in  {\rm \cite[2.4.50]{CD12}}, which is as follows.
  \begin{enumerate}[{\rm (1)}]
    \item For any morphism $f:X\rightarrow S$ in $\mathscr{S}$, there exists $3$ pairs of adjoint functors as follows:
    \[f^*:H(S)\rightleftarrows H(X):f_*,\]
    \[f_!:H(X)\rightleftarrows H(S):f^!,\]
    \[(\otimes,Hom),\text{ symmetric closed monoidal structure on }H(S).\]
    \item There exists a structure of covariant {\rm (}resp.\ contravariant{\rm )} pseudofunctors on $f\mapsto f_*$, $f\mapsto f_!$ {\rm (}resp.\ $f\mapsto f^*$, $f\mapsto f^!${\rm )}.
    \item There exists a canonical natural transformation
    \[\alpha_f:f_!\rightarrow f_*\]
    which is an isomorphism when $f$ is proper.
    \item For any smooth separated morphism $f:X\rightarrow S$ in $\mathscr{S}$ of relative dimension $d$ with tangent bundle $T_f$, there exists a canonical natural isomorphism
    \[f^!\otimes MTh(-T_f)\longrightarrow f^*.\]
    See {\rm \cite[2.4.12]{CD12}} for the definition of $MTh$.

    If $H$ admits an orientation in the sense of {\rm \cite[2.4.38]{CD12}} and $f$ has dimension $d$, then there exists a canonical natural isomorphism
    \[f^!(-d)[-2d]\rightarrow f^*.\]
    \item For any Cartesian diagram
    \[\begin{tikzcd}
      X'\arrow[d,"f'"]\arrow[r,"g'"]&X\arrow[d,"f"]\\
      S'\arrow[r,"g"]&S
    \end{tikzcd}\]
    in $\mathscr{S}$, there exist natural isomorphisms
    \[g^*f_!\stackrel{\sim}\rightarrow f_!'g'^*,\]
    \[g_*'f'^!\stackrel{\sim}\rightarrow f^!g_*.\]
    \item For any morphism $f:X\rightarrow S$ in $\mathscr{S}$, there exist natural isomorphisms
    \[f_!K\otimes_S L\stackrel{\sim}\longrightarrow f_!(K\otimes_X L),\]
    \[Hom_S(f_!(L),K)\stackrel{\sim}\longrightarrow f_*Hom_X(L,f^!K),\]
    \[f^!Hom_S(L,M)\stackrel{\sim}\longrightarrow Hom_X(f^*L,f^!M).\]
    \item For any closed immersion $i:Z\rightarrow S$ with complement $j:U\rightarrow S$, there exists a distinguished triangle of natural transformations
    \[j_!j^!\stackrel{ad'}\longrightarrow {\rm id}\stackrel{ad}\longrightarrow i_*i^*\stackrel{\partial_i}\longrightarrow j_!j^![1]\]
    where $ad$ {\rm (}resp.\ $ad'${\rm )} denotes the unit {\rm (}resp.\ counit{\rm )} of the relevant adjunction.
  \end{enumerate}
\end{thm}
\begin{none}\label{0.6}
  Our primary application of the theorems is as follows. Consider the diagram
  \[\begin{tikzcd}
    \mathscr{S}^{sm}\arrow[rd,"{{\rm DM}(-,\Lambda)}"']\arrow[rr,"W"]&&\mathscr{S}\\
    &{\rm Tri}^{\otimes}
  \end{tikzcd}\]
  of $2$-categories. In (\ref{8.1}), we will check that ${\rm DM}(-,\Lambda):\mathscr{S}^{sm}\rightarrow {\rm Tri}^{\otimes}$ is a $Sm$-premotivic pseudofunctor satisfying (Loc) and the axioms from (B--5) to (B--8), so by (\ref{0.5}), there is a pseudofunctor ${\rm DM}^{loc}(-,\Lambda):\mathscr{S}\rightarrow {\rm Tri}^{\otimes}$ unique up to pseudonatural equivalence such that it is an extension of ${\rm DM}(-,\Lambda):\mathscr{S}^{sm}\rightarrow {\rm Tri}^{\otimes}$ and satisfies the Grothendieck six operations formalism.

  We can also apply the theorem to \'etale realization. Let $n$ be a positive integer, and assume that $T$ is a noetherian scheme separated over ${\rm Spec}\,\mathbb{Z}[1/n]$. In (\ref{8.4}), we will construct the \'etale realization of ${\rm DM}(-,\mathbb{Z}):\mathscr{S}^{sm}\rightarrow {\rm Tri}^{\otimes}$, which is a $Sm$-premotivic pseudonatural transformation
  \[R_{\et,n}:{\rm DM}(-,\mathbb{Z})\rightarrow {\rm D}_{\et}(-,\mathbb{Z}/n\mathbb{Z}).\]
  From this, we have the commutative diagram
  \[\begin{tikzcd}
    \mathscr{S}^{sm}\arrow[rrdd,"{{\rm DM}(-,\mathbb{Z})}"',bend right]\arrow[rrdd,"{{\rm D}_{\et}(-,\mathbb{Z}/n\mathbb{Z})}",near end]\arrow[rrrr,"W"]&&\,&\,&\mathscr{S}\\
    \\
    \,\arrow[rrruu,phantom,"{\;\;\;\;\;\;\;\;\;\rotatebox[origin=c]{45}{$\Rightarrow$}_{R_{\et}}}",near start]&&{\rm Tri}^{\otimes}
  \end{tikzcd}\]
  of $2$-categories. Then by (\ref{0.4}), there is a $Sm$-premotivic pseudonatural transformation
  \[R_{\et,n}^{loc}:{\rm DM}^{loc}(-,\mathbb{Z})\rightarrow {\rm D}_{\et}(-,\mathbb{Z}/n\mathbb{Z})\]
  unique up to isomorphism such that it is an extension of $R_{\et,n}:{\rm DM}(-,\mathbb{Z})\rightarrow {\rm D}_{\et}(-,\mathbb{Z}/n\mathbb{Z})$.
\end{none}
\begin{none}\label{0.8}
  We also prove that ${\rm DM}^{loc}(-,\Lambda)$ admits an orientation, so the second part of (\ref{0.7}(4)) holds for ${\rm DM}^{loc}(-,\Lambda)$.
\end{none}
\begin{none}
  {\bf Organization of the paper.} In Section 2, we review pseudofunctors, pseudonatural transformations, and modifications to help the reader to understand our notations. Then we review $\mathscr{P}$-premotivic pseudofunctors and $\mathscr{P}$-premotivic pseudonatural transformations. In Section 4, we define categories with immersions, and we review the localization property and motivic pseudofunctors. We also construct a category $\widetilde{\mathcal{C}}$ that will be used later.

  In Sections 3, 5, 6, and 7, we prove the main theorem (\ref{0.4}) by constructing the commutative diagram
  \[\begin{tikzcd}
    \mathcal{E}\arrow[rd,"V"']\arrow[r,"U"]&\widetilde{\mathcal{C}}\arrow[d,"G"]\arrow[ld,phantom,"{\rotatebox[origin=c]{0}{$\Leftrightarrow$}}_{\delta}",very near start]\arrow[r,"F"]\arrow[rd,phantom,"{\rotatebox[origin=c]{0}{$\Leftrightarrow$}}_\alpha",very near start]&\mathcal{C}\arrow[ld,"H"]\\
    \,&{\rm Tri}^\otimes&\,
  \end{tikzcd}\]
  of $2$-categories where $\mathcal{C}=\mathscr{S}$ and $\mathcal{E}=\mathscr{S}^{sm}$. In Section 3, we construct $H$ and $\alpha$ when $G$ is given and some conditions are satisfied. In Section 5, we construct $G$ and $\delta$. In Section 6, we show a condition for $F$ so that we can use the argument in Section 3. In Section 7, we show that $H$ is $Sm$-premotivic, and then we prove the main theorem (\ref{0.4}).

  In Section 8, we prove (\ref{0.7}), and in Section 9, we apply (\ref{0.4}) and (\ref{0.7}) to construct ${\rm DM}^{loc}(-,\Lambda)$ and its \'etale realization. We also construct a pseudonatural equivalence between ${\rm DM}^{loc}(-,\Lambda)$ and ${\rm DA}_{\et}(-,\Lambda)$ when $\Lambda$ is a ${\bf Q}$-algebra. In Section 10, we prove that ${\rm DM}^{loc}(-,\Lambda)$ admits an orientation.
\end{none}
\begin{none}
  {\bf Terminology and conventions.} Alongside (\ref{0.1}), we have the following notations. When we have an adjunction
  \[F:\mathcal{C}\leftrightarrows \mathcal{D}:G\]
  where $F$ and $G$ are contravariant functors of categories, we denote by $ad$ (resp.\ $ad'$) the unit ${\rm id}\longrightarrow GF$ (resp.\ the counit $FG\longrightarrow {\rm id}$).
\end{none}
\begin{none}
  {\bf Acknowledgements.} The author is grateful to Martin Olsson for helpful conversations. The author is grateful to Adeel Khan for indicating that orientation for ${\rm DM}^{loc}(-,\Lambda)$ is needed. Section 10 is obtained from a conversation with him.
\end{none}
\section{Premotivic pseudofunctors}
\begin{none}\label{3.6}
  To help the reader to understand our notations, we include the definitions of pseudofunctors, pseudonatural transformations, and modifications as follows.
\end{none}
\begin{df}\label{3.7}
  Let $\mathcal{C}$ be a category, and let $\mathcal{D}$ be a $2$-category. A contravariant {\it pseudofunctor} $F:\mathcal{C}\rightarrow \mathcal{D}$ is the data of
  \begin{enumerate}[(1)]
    \item an object $F(S_1)$ of $\mathcal{D}$ for each object $S_1$ of $\mathcal{C}$,
    \item a $1$-morphism
    \[F(f_1):F(S_1)\rightarrow F(S_2)\]
    in $\mathcal{D}$ for each morphism $f_1:S_2\rightarrow S_1$ in $\mathcal{C}$,
    \item an invertible $2$-morphism
    \[F_{{\rm id}_{S_1}}:F({\rm id}_{S_1})\rightarrow {\rm id}\]
    in $\mathcal{D}$ for each object $S_1$ of $\mathcal{C}$,
    \item an invertible $2$-morphism
    \[F_{f_1,f_2}:F(f_2)F(f_1)\rightarrow F(f_1f_2)\]
    in $\mathcal{D}$ for each pair of morphisms $f_1:S_2\rightarrow S_1$ and $f_2:S_3\rightarrow S_2$ in $\mathcal{C}$
  \end{enumerate}
  satisfying the following axioms.
  \begin{enumerate}[(i)]
    \item For any morphism $f_1:S_2\rightarrow S_1$ in $\mathcal{C}$, the diagram
    \[\begin{tikzcd}
      F(f_1)F({\rm id}_{S_1})\arrow[r,"F_{{\rm id}_{S_1},f_1}"]\arrow[d,"F_{{\rm id}_{S_1}}"]&F({\rm id}_{S_1}\circ f_1)\arrow[d,equal]\\
      F(f_1)\circ {\rm id}\arrow[r,equal]&F(f_1)
    \end{tikzcd}\]
    in $\mathcal{D}$ commutes.
    \item For any morphism $f_1:S_2\rightarrow S_1$ in $\mathcal{C}$, the diagram
    \[\begin{tikzcd}
      F({\rm id}_{S_2})F(f_1)\arrow[r,"F_{f_1,{\rm id}_{S_2}}"]\arrow[d,"F_{{\rm id}_{S_2}}"]&F(f_1\circ {\rm id}_{S_2})\arrow[d,equal]\\
      {\rm id}\circ F(f_1) \arrow[r,equal]&F(f_1)
    \end{tikzcd}\]
    in $\mathcal{D}$ commutes.
    \item For any morphisms $S_4\stackrel{f_3}\rightarrow S_3\stackrel{f_2}\rightarrow S_2\stackrel{f_1}\rightarrow S_1$ in $\mathcal{C}$, the diagram
    \[\begin{tikzcd}
      F(f_3)F(f_2)F(f_1)\arrow[r,"F_{f_1,f_2}"]\arrow[d,"F_{f_2,f_3}"]&F(f_3)F(f_1f_2)\arrow[d,"F_{f_1f_2,f_3}"]\\
      F(f_2f_3)F(f_1)\arrow[r,"F_{f_1,f_2f_3}"]&F(f_1f_2f_3)
    \end{tikzcd}\]
    commutes.
  \end{enumerate}
  We often write $f_1^*$ for $F(f_1)$ when no confusion seems likely to arise.
\end{df}
\begin{df}
  Let $\mathcal{C}$ be a category, let $\mathcal{D}$ be a $2$-category, and let $F,F':\mathcal{C}\rightarrow \mathcal{D}$ be contravariant pseudofunctors. A {\it pseudonatural transformation} $\alpha:F\rightarrow F'$ is the data of
  \begin{enumerate}[(1)]
    \item $1$-morphism
    \[\alpha(S_1):F(S_1)\rightarrow F'(S_1)\]
    in $\mathcal{D}$ for each object $S_1$ of $\mathcal{C}$,
    \item invertible $2$-morphism
    \[\alpha_{f_1}:\alpha(S_2)F(f_1)\rightarrow F'(f_1)\alpha(S_1)\]
    in $\mathcal{D}$ for each morphism $f_1:S_2\rightarrow S_1$ in $\mathcal{C}$
  \end{enumerate}
  satisfying the following axioms.
  \begin{enumerate}[(i)]
    \item For any object $S_1$ of $\mathcal{C}$, the diagram
    \[\begin{tikzcd}
      \alpha(S_1)F({\rm id}_{S_1})\arrow[rr,"\alpha_{{\rm id}_{S_1}}"]\arrow[rd,"F_{{\rm id}_{S_1}}"']&&F'({{\rm id}_{S_1}})\alpha(S_1)\arrow[ld,"F_{{\rm id}_{S_1}}'"]\\
      &\alpha(S_1)
    \end{tikzcd}\]
    in $\mathcal{D}$ commutes.
    \item For any morphisms $S_3\stackrel{f_2}\rightarrow S_2\stackrel{f_1}\rightarrow S_1$ in $\mathcal{C}$, the diagram
    \[\begin{tikzcd}
      \alpha(S_3)F(f_2)F(f_1)\arrow[d,"F_{f_1,f_2}"]\arrow[r,"\alpha_{f_2}"]&F'(f_2)\alpha(S_2)F(f_1)\arrow[r,"\alpha_{f_1}"]&F'(f_2)F'(f_1)\alpha(S_1)\arrow[d,"F_{f_1,f_2}'"]\\
      \alpha(S_3)F(f_1f_2)\arrow[rr,"\alpha_{f_1f_2}"]&&F'(f_1f_2)\alpha(S_1)
    \end{tikzcd}\]
    in $\mathcal{D}$ commutes.
  \end{enumerate}
  A pseudonatural transformation $\alpha:F\rightarrow F'$ is called a {\it pseudonatural equivalence} if $\alpha(S_1)$ is invertible for any object $S_1$ of $\mathcal{C}$.
\end{df}
\begin{df}\label{3.11}
  Let $\mathcal{C}$ be a category, let $\mathcal{D}$ be a $2$-category, let $F,F':\mathcal{C}\rightarrow \mathcal{D}$ be contravariant pseudofunctors, and let $\alpha,\alpha':F\rightarrow F'$ be pseudonatural transformations. A {\it modification} $\Phi:\alpha\rightarrow \alpha'$ is the data of $2$-morphism
  \[\Phi_{S_1}:\alpha(S_1)\rightarrow \alpha'(S_1)\]
  for each object $S_1$ of $\mathcal{C}$ satisfying the axiom that for any morphism $f_1:S_2\rightarrow S_1$ in $\mathcal{C}$, the diagram
  \[\begin{tikzcd}
    \alpha(S_2)F(f_1)\arrow[d,"\alpha_{f_1}"]\arrow[r,"\Phi_{S_2}"]&\alpha'(S_2)F(f_1)\arrow[d,"\alpha_{f_1}'"]\\
    F'(f_1)\alpha(S_1)\arrow[r,"\Phi_{S_1}"]&F'(f_1)\alpha'(S_1)
  \end{tikzcd}\]
  in $\mathcal{D}$ commutes.

  A modification $\Phi:\alpha\rightarrow \alpha'$ is called an {\it isomorphism} if $\Phi_{S_1}$ is invertible for any object $S_1$ of $\mathcal{C}$.
\end{df}
\begin{none}\label{3.4}
  Now, following \cite[1.4.2, 1.4.6]{CD12}, we will define $\mathscr{P}$-premotivic pseudofunctors and $\mathscr{P}$-premotivic pseudonatural transformations as follows.
\end{none}
\begin{df}\label{3.2}
  Let $\mathcal{C}$ be a category with fiber products, and let $\mathscr{P}$ be a class of morphisms of $\mathcal{C}$ containing all isomorphisms and stable by compositions and pullbacks. We say that a pseudofunctor
  \[H:\mathcal{C}\longrightarrow {\rm Tri}^\otimes\]
  is $\mathscr{P}$-{\it premotivic} if it satisfies the following axioms.
  \begin{enumerate}
    \item[(B--1)] For any morphism $f$ in $\mathcal{C}$, the functor $f^*:=H(f)$ has a right adjoint, denoted by $f_*$.
    \item[(B--2)] For any morphism $f$ in $\mathscr{P}$, the functor $f^*$ has a left adjoint, denoted by $f_\sharp$.
    \item[(B--3)] For any Cartesian diagram
    \[\begin{tikzcd}
      X'\arrow[d,"f'"]\arrow[r,"g'"]&X\arrow[d,"f"]\\
      S'\arrow[r,"g"]&S
    \end{tikzcd}\]
    in $\mathcal{C}$ with $f\in \mathscr{P}$, the natural transformation
    \begin{equation}\label{3.2.1}
      Ex:f_\sharp g^*\longrightarrow g'^*f_\sharp'
    \end{equation}
    given by the composition
    \[f_\sharp g^*\stackrel{ad}\longrightarrow f_\sharp g^*f'^*f_\sharp'\stackrel{\sim}\longrightarrow f_\sharp f^*g'^*f_\sharp'\stackrel{ad'}\longrightarrow g'^*f_\sharp'\]
    is an isomorphism. The natural transformation (\ref{3.2.1}) is called an {\it exchange transformation}.
    \item[(B--4)] For any morphism $f:X\rightarrow S$ in $\mathscr{P}$, and for any objects $K$ of $H(X)$ and $L$ of $H(S)$, the natural transformation
    \begin{equation}\label{3.2.2}
      Ex:f_\sharp(K\otimes_X f^*L)\longrightarrow f_\sharp K\otimes_S L
    \end{equation}
    given by the composition
    \[f_\sharp(K\otimes_X f^*L)\stackrel{ad}\longrightarrow f_\sharp (f^*f_\sharp K\otimes_X f^*L)\stackrel{\sim}\longrightarrow f_\sharp f^*(f_\sharp K\otimes_S L)\stackrel{ad'}\longrightarrow f_\sharp K\otimes_S L\]
    is an isomorphism. Here, the second arrow is obtained by the fact that $f^*$ is monoidal. The natural transformation (\ref{3.2.2}) is called an {\it exchange transformation}.
  \end{enumerate}
\end{df}
\begin{rmk}\label{3.10}
  In (\ref{3.2}), note that if $f$ is an isomorphism, then the axioms from (B--1) to (B--4) for $f$ is always satisfied.
\end{rmk}
\begin{df}\label{3.8}
  Under the notations and hypotheses of (\ref{3.2}), assume that $H$ is $\mathscr{P}$-premotivic. For any object $S$ of $\mathcal{C}$, we denote by $1_S$ the unit of $H(S)$ obtained by the monoidal structure of $H(S)$.
\end{df}
\begin{df}\label{3.5}
  Let $\mathcal{C}$ be a category with fiber products, and let $\mathscr{P}$ be a class of morphisms of $\mathcal{C}$ containing all isomorphisms and stable by compositions and pullbacks, and let $H,H':\mathcal{C}\rightarrow {\rm Tri}^\otimes$ be $\mathscr{P}$-premotivic pseudofunctors. We say that a pseudonatural transformation
  \[\alpha:H\rightarrow H'\]
  is $\mathscr{P}$-{\it premotivic} if it satisfies the following axioms.
  \begin{enumerate}
    \item[(C--1)] For any object $S_1$ of $\mathcal{C}$, the functor
    \[\alpha(S_1):H(S_1)\rightarrow H'(S_1)\]
    has a right adjoint.
    \item[(C--2)] For any morphism $f_1:S_2\rightarrow S_1$ in $\mathscr{P}$, the natural transformation
    \[Ex:f_{1\sharp}\alpha(S_2)\stackrel{ad}\longrightarrow f_{1\sharp}\alpha(S_2)f_1^*f_{1\sharp}\stackrel{\alpha_{f_1}}\longrightarrow f_{1\sharp}f_1^*\alpha(S_1)f_{1\sharp}\stackrel{ad'}\longrightarrow \alpha(S_1)f_{1\sharp}\]
    is an isomorphism.
  \end{enumerate}
\end{df}
\begin{rmk}\label{3.12}
  In (\ref{3.5}), note that if $f_1$ is an isomorphism, then the axiom (C--2) for $f_1$ is always satisfied.
\end{rmk}
\begin{prop}\label{3.9}
  Under the notations and hypotheses of {\rm (\ref{3.5})}, if $\alpha(S_1)$ is an isomorphism for any object $S_1$ of $\mathcal{C}$, then $H$ is $\mathscr{P}$-premotivic.
\end{prop}
\begin{proof}
  For any object $S_1$ of $\mathcal{C}$, let $\beta(S_1)$ denote the inverse of $\alpha(S_1)$. Then $\beta(S_1)$ is both the right adjoint and left adjoint of $\alpha(S_1)$, so (C--1) is satisfied.

  Let $f_1:S_2\rightarrow S_1$ be a morphism in $\mathcal{P}$. The left adjoint of the isomorphism
  \[\alpha_{f_1}:\alpha(S_2)f_1^*\longrightarrow f_1^*\alpha(S_1)\]
  is the composition
  \[\begin{split}
    \beta(S_1)f_{1\sharp}&\stackrel{ad}\longrightarrow \beta(S_1)f_{1\sharp}\alpha(S_2)\beta(S_2)\stackrel{ad}\longrightarrow \beta(S_1)f_{1\sharp}\alpha(S_2)f_1^*f_1{\sharp}\beta(S_2)\\
     &\stackrel{\alpha_{f_1}}\longrightarrow \beta(S_1)f_{1\sharp}f_1^*\alpha(S_1)f_{1\sharp}\beta(S_2)\stackrel{ad'}\longrightarrow \beta(S_1)\alpha(S_1)f_{1\sharp}\beta(S_2) \stackrel{ad'}\longrightarrow f_{1\sharp}\beta(S_2),
  \end{split}\]
  which is also an isomorphism. The first and fifth arrows are isomorphisms since $\beta(S_1)$ and $\beta(S_2)$ are the inverses of $\alpha(S_1)$ and $\alpha(S_2)$ respectively. Thus the composition of the second, third, and fourth arrows is an isomorphism, applying $\alpha(S_1)$ to the left and $\alpha(S_2)$ to the right of the composition, we get the natural transformation in the axiom (C--2). Thus $H$ satisfies (C--2).
\end{proof}
\begin{rmk}\label{3.14}
  For example, the pseudonatural equivalence $\kappa$ in (\ref{0.4}) is $Sm$-premotivic.
\end{rmk}
\begin{prop}\label{3.3}
  Let $\mathcal{C}$ and $\mathcal{D}$ be categories with fiber product, and let $\mathscr{P}$ be a class of morphisms in $\mathcal{C}$ containing all isomorphisms and stable by compositions and pullback. Consider a commutative diagram
  \[\begin{tikzcd}
    \mathcal{C}\arrow[rd,"G"']\arrow[rrd,phantom,"{\rotatebox[origin=c]{0}{$\Leftrightarrow$}}_{\alpha}"]\arrow[rr,"F"]&&\mathcal{D}\arrow[ld,"H"]\\
    &{\rm Tri}^{\otimes}&\;
  \end{tikzcd}\]
  of $2$-categories where
  \begin{enumerate}[{\rm (i)}]
    \item $F$ is a covariant functor preserving fiber products,
    \item $G$ and $H$ are contravariant pseudofunctors,
    \item $\alpha$ is a pseudonatural equivalence.
  \end{enumerate}
  Then $G$ is $\mathscr{P}$-premotivic if and only if $H$ is $F(\mathscr{P})$-premotivic.
\end{prop}
\begin{proof}
  It follows from (\ref{3.10}) since every morphism in $F(\mathscr{P})$ is a composition of isomorphisms and morphisms of the form $F(f)$ for $f\in \mathscr{P}$.
\end{proof}
\begin{prop}\label{3.13}
  Under the notations and hypotheses of {\rm (\ref{3.3})}, consider a commutative diagram
  \[\begin{tikzcd}
    \mathcal{C}\arrow[rd,"G'"']\arrow[rrd,phantom,"{\rotatebox[origin=c]{0}{$\Leftrightarrow$}}_{\alpha'}"]\arrow[rr,"F"]&&\mathcal{D}\arrow[ld,"H'"]\\
    &{\rm Tri}^{\otimes}&\;
  \end{tikzcd}\]
  of $2$-categories and a diagram
  \[\begin{tikzcd}
    G\arrow[r,"\alpha"]\arrow[d,"\beta"]&H\circ F\\
    G'\arrow[r,"\alpha'"]&H'\circ F
  \end{tikzcd}\]
  of pseudofunctors where
  \begin{enumerate}[(i)]
    \item $G'$ is a $\mathscr{P}$-premotivic pseudofunctor and $H'$ are contravariant pseudofunctors,
    \item $\alpha'$ is a pseudonatural equivalence.
    \item $\beta$ is a $\mathscr{P}$-premotivic pseudonatural transformation.
  \end{enumerate}
  If $\gamma:H\rightarrow H'$ is a pseudonatural transformation such that the induced pseudonatural transformation $H\circ F\rightarrow H'\circ F$ makes the above diagram commutative, then $\gamma$ is $F(\mathscr{P})$-premotivic.
\end{prop}
\begin{proof}
  It follows from (\ref{3.12}) since every morphism in $F(\mathscr{P})$ is a composition of isomorphisms and morphisms of the form $F(f)$ for $f\in \mathscr{P}$.
\end{proof}
\section{Proof of (\ref{0.4}), part I}
\begin{none}\label{1.1}
  Throughout this section, we fix a diagram
  \begin{equation}\label{1.1.1}\begin{tikzcd}
    \widetilde{\mathcal{C}}\arrow[d,"G"']\arrow[r,"F"]&\mathcal{C}\\
    \mathcal{D}
  \end{tikzcd}\end{equation}
  where
  \begin{enumerate}[(i)]
    \item $\mathcal{C}$ and $\widetilde{\mathcal{C}}$ are categories, and $\mathcal{D}$ is a $2$-category,
    \item $F$ is an essentially surjective covariant functor, and $G$ is a contravariant pseudofunctor.
  \end{enumerate}
  We denote by $\mathscr{W}$ the class of morphisms in $\widetilde{\mathcal{C}}$ such that $f\in \mathscr{W}$ if and only if $F(f)$ is an isomorphism. Assume that for any morphism $f$ in $\mathscr{W}$, $f^*$ is invertible.

  Consider the category $\mathcal{C}'$ where ${\rm ob}\,\mathcal{C}':={\rm ob}\,\widetilde{\mathcal{C}}$ and
  \[{\rm Hom}_{\mathcal{C}'}(X_1,X_2)={\rm Hom}_{\mathcal{C}}(F(X_1),F(X_2))\]
  for any $X_1,X_2\in {\rm ob}\,\mathcal{C}'$. Then $\mathcal{C}'$ is equivalent to $\mathcal{C}$ since $F$ is essentially surjective, so it is not harmful to consider $\mathcal{C}'$ instead of $\mathcal{C}$. Thus we may assume ${\rm ob}\,\mathcal{C}:={\rm ob}\,\widetilde{\mathcal{C}}$.

  Under this assumption, for any object $X$ of $\mathcal{C}$, we denote by $\widetilde{X}$ the corresponding object in $\widetilde{C}$ to avoid confusion.
\end{none}
\begin{none}\label{1.2}
  We will study when the diagram (\ref{1.1.1}) can be extended to a commutative diagram
  \begin{equation}\label{1.2.1}\begin{tikzcd}
    \widetilde{\mathcal{C}}\arrow[d,"G"']\arrow[r,"F"]\arrow[rd,phantom,"{\rotatebox[origin=c]{0}{$\Leftrightarrow$}}_{\alpha}",very near start]&\mathcal{C}\arrow[ld,"H"]\\
    \mathcal{D}&\,
  \end{tikzcd}\end{equation}
  of $2$-categories where $H$ is a contravariant pseudofunctor and $\alpha:G\rightarrow H\circ F$ is a pseudonatural equivalence. Consider the following conditions for $G$.
  \begin{enumerate}
    \item[(A--1)] For any morphism $X_2\stackrel{f_1}\rightarrow X_1$ in $\mathcal{C}$, we can choose a diagram
    \[\begin{tikzcd}
      &\widetilde{X_{12}}\arrow[ld,"p_2^{12}"']\arrow[rd,"p_1^{12}"]\\
      \widetilde{X_2}&&\widetilde{X_1}
    \end{tikzcd}\]
    in $\widetilde{\mathcal{C}}$ such that $p_2^{12}$ is in $\mathscr{W}$ and $f_1=F(p_1^{12})(F(p_2^{12}))^{-1}$. Moreover, when $f_1=F(\widetilde{f_1})$ for some morphism $\widetilde{f_1}:\widetilde{X_2}\rightarrow \widetilde{X_1}$ in $\widetilde{\mathcal{C}}$, we require that $p_2^{12}$ has a section $d_{12}^2$ such that $p_1^{12}d_{12}^2=\widetilde{f_1}$. Here, $\widetilde{X_1}=X_1$ and $\widetilde{X_2}=X_2$, but we distinguish the notations following (\ref{1.1}) to avoid confusion.
    \item[(A--2)] Assume (A--1). For any morphisms $X_3\stackrel{f_2}\rightarrow X_2\stackrel{f_1}\rightarrow X_1$ in $\mathcal{C}$, we can choose a commutative diagram
    
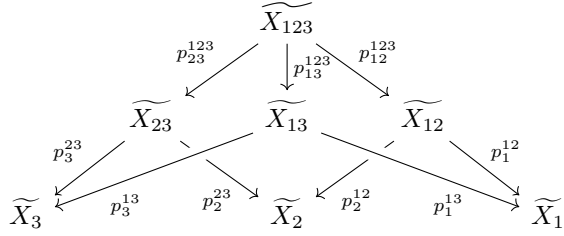
\[\begin{tikzcd}
      &&\widetilde{X_{123}}\arrow[ld,"p_{23}^{123}"']\arrow[d,"p_{13}^{123}"]\arrow[rd,"p_{12}^{123}"]\\
      &\widetilde{X_{23}}\arrow[ld,"p_3^{23}"']\arrow[rd,"p_2^{23}"',near end]&\widetilde{X_{13}}\arrow[lld,crossing over,"p_3^{13}",near end]&\widetilde{X_{12}}\arrow[ld,"p_2^{12}",near end]\arrow[rd,"p_1^{12}"]\\
      \widetilde{X_3}&&\widetilde{X_2}&&\widetilde{X_1}\arrow[llu,"p_1^{13}",near start,leftarrow,crossing over]
    \end{tikzcd}\]
    such that $p_{13}^{123}$ and $p_{23}^{123}$ are in $\mathscr{W}$. Here, $\widetilde{X_{12}},\widetilde{X_{13}},\widetilde{X_{23}},p_1^{12},\ldots,p_3^{23}$ are already chosen from (A--1). Moreover, when $f_1=F(\widetilde{f_1})$ (resp.\ $f_2=F(\widetilde{f_2})$) for some morphism $\widetilde{f_1}:\widetilde{X_2}\rightarrow \widetilde{X_1}$ (resp.\ $\widetilde{f_2}:\widetilde{X_3}\rightarrow \widetilde{X_2}$) in $\widetilde{\mathcal{C}}$, we require that $p_{23}^{123}$ (resp.\ $p_{13}^{123}$) has a section $d_{123}^{23}$ (resp.\ $d_{123}^{13}$ ) such that $p_{12}^{123}d_{123}^{23}=d_{12}^2p_2^{23}$ (resp.\ $p_{23}^{123}d_{123}^{13}=d_{23}^3p_3^{13}$). Here, $d_{12}^2$ and $d_{23}^3$ are defined in (A--1).
    \item[(A--3)] Assume (A--1) and (A--2). For any morphisms $X_4\stackrel{f_3}\rightarrow X_3\stackrel{f_2}\rightarrow X_2\stackrel{f_1}\rightarrow X_1$ in $\mathcal{C}$, we can choose a commutative diagram
    
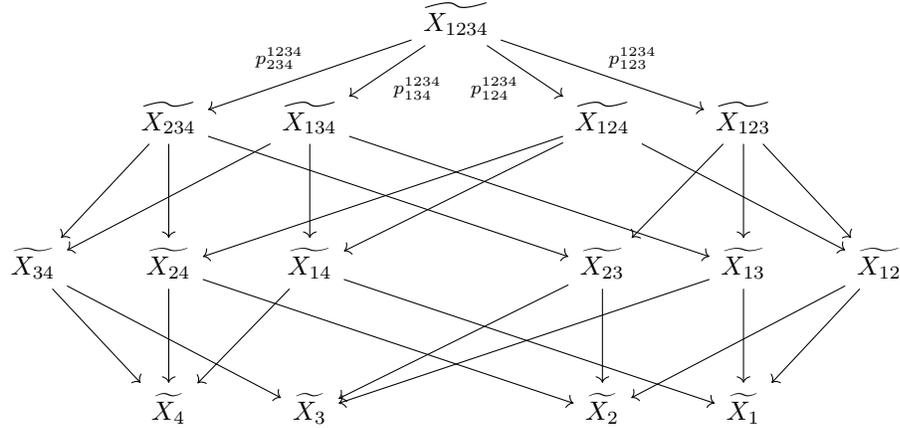
\[\begin{tikzcd}
      &&&\widetilde{X_{1234}}\arrow[lld,"p_{234}^{1234}"']\arrow[ld,"p_{134}^{1234}"]\arrow[rd,"p_{124}^{1234}"']\arrow[rrd,"p_{123}^{1234}"]\\
      &\widetilde{X_{234}}\arrow[ldd]\arrow[dd]\arrow[rrrdd]&\widetilde{X_{134}}\arrow[lldd]\arrow[dd]\arrow[rrrdd]&&\widetilde{X_{124}}\arrow[llldd]\arrow[lldd] \arrow[rrdd]&\widetilde{X_{123}} \arrow[ldd]\arrow[dd]\arrow[rdd]\\
      \\
      \widetilde{X_{34}}\arrow[rdd]\arrow[rrdd]&\widetilde{X_{24}}\arrow[dd]\arrow[rrrdd]&\widetilde{X_{14}}\arrow[ldd]\arrow[rrrdd] &&\widetilde{X_{23}}\arrow[lldd]\arrow[dd]&\widetilde{X_{13}}\arrow[llldd]\arrow[dd]&\widetilde{X_{12}}\arrow[lldd]\arrow[ldd]\\
      \\
      &\widetilde{X_4}&\widetilde{X_3}&&\widetilde{X_2}&\widetilde{X_1}
    \end{tikzcd}\]
    in $\widetilde{\mathcal{C}}$ such that $p_{124}^{1234}$, $p_{134}^{1234}$, and $p_{234}^{1234}$ are in $\mathscr{W}$. Here, $\widetilde{X_{12}},\ldots,\widetilde{X_{34}},\widetilde{X_{123}},\ldots,\widetilde{X_{234}}$ and not denoted arrows are already chosen from (A--1) and (A--2).
  \end{enumerate}
\end{none}
\begin{none}\label{1.3}
  Under the conditions from (A--1) to (A--3), we can construct a pseudofunctor $H$ as follows.
  \begin{enumerate}[(1)]
    \item For any object $X_1$ of $\mathcal{C}$, put
    \[H(X_1):=G(\widetilde{X_1}).\]
    \item For any morphism $f_1:X_2\rightarrow X_1$ in $\mathcal{C}$, we can fix a diagram
        
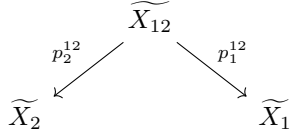
\[\begin{tikzcd}
      &\widetilde{X_{12}}\arrow[ld,"p_2^{12}"']\arrow[rd,"p_1^{12}"]\\
      \widetilde{X_2}&&\widetilde{X_1}
    \end{tikzcd}\]
    in $\widetilde{\mathcal{C}}$ by (A--1). Then put
    \[H(f_1):=(p_2^{12*})^{-1}p_1^{12*}.\]
    \item Under the notations and hypotheses of (2), if $f_1$ is the identity morphism, we have the invertible $2$-morphism
    \[H_{f_1}:H(f_1)\longrightarrow {\rm id}\]
    given by
    \[(p_2^{12*})^{-1}p_1^{12*}\stackrel{\sim}\longrightarrow d_{12}^{2*}p_2^{12*}(p_2^{12*})^{-1}p_1^{12*} \stackrel{\sim}\longrightarrow d_{12}^{2*}p_1^{12*}\stackrel{\sim}\longrightarrow {\rm id}^*\stackrel{\sim}\longrightarrow {\rm id}.\]
    \item For any morphisms $f_2:X_3\rightarrow X_2$ and $f_1:X_2\rightarrow X_1$ in $\mathcal{C}$, we have the invertible $2$-morphism
    \[H_{f_1,f_2}:H(f_2)H(f_1)\longrightarrow H(f_1f_2)\]
    given by
    \[\begin{split}
      (p_3^{23*})^{-1}p_2^{23*}(p_2^{12*})^{-1}p_1^{12*}&\stackrel{\sim}\longrightarrow (p_3^{23*})^{-1}(p_{23}^{123*})^{-1}p_{12}^{123*}p_1^{12*} \stackrel{\sim}\longrightarrow (p_3^{123*})^{-1}p_1^{123*}\\
      &\stackrel{\sim}\longrightarrow (p_3^{13*})^{-1}(p_{13}^{123*})^{-1}p_{13}^{123*}p_1^{13*} \stackrel{\sim}\longrightarrow (p_3^{13*})^{-1}p_1^{13*}.
    \end{split}\]
    where $p_1^{123}=p_1^{12}p_{12}^{123}$ and $p_3^{123}=p_3^{23}p_{23}^{123}$. Here, the notations are from (A--2), and the first arrow is induced by the $2$-morphism
    \[p_{23}^{123*}p_2^{23*}\stackrel{\sim}\longrightarrow p_{12}^{123*}p_2^{12*}.\]
  \end{enumerate}
  We will verify the axioms of pseudofunctors for $H$ in (\ref{1.5}).
\end{none}
\begin{none}\label{1.4}
  Let us first study the functoriality of the condition (A--2). Assume the conditions from (A--1) to (A--3), and consider a commutative diagram
  
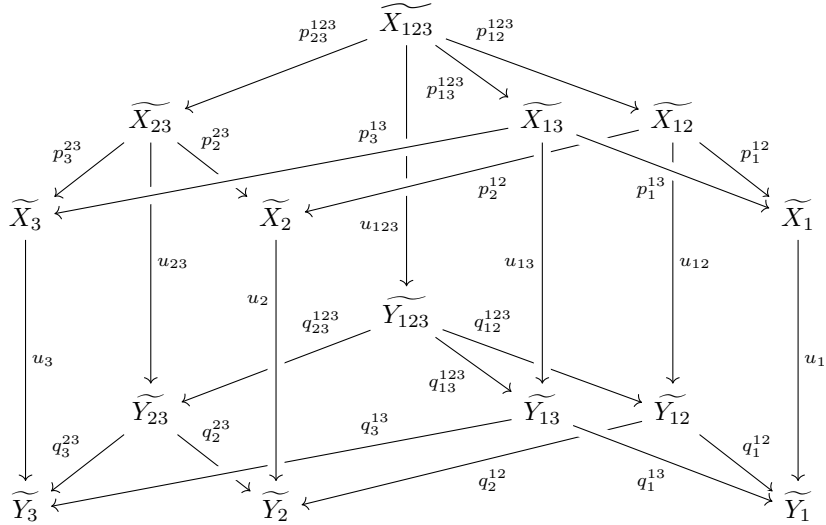
\[\begin{tikzcd}
    &&&\widetilde{X_{123}}\arrow[lld,"p_{23}^{123}"',very near start]\arrow[rrd,"p_{12}^{123}",very near start]\arrow[rd,"p_{13}^{123}"']\arrow[ddd,"u_{123}"',near end]\\
    &\widetilde{X_{23}}\arrow[ddd,"u_{23}"]\arrow[ld,"p_3^{23}"']\arrow[rd,"p_2^{23}",near start]&&&\widetilde{X_{13}}\arrow[lllld,"p_3^{13}"',crossing over,near start]&\widetilde{X_{12}}\arrow[ddd,"u_{12}"]\arrow[llld,"p_2^{12}",crossing over] \arrow[rd,"p_1^{12}"]\\
    \widetilde{X_3}\arrow[ddd,"u_3"]&&\widetilde{X_2}&&&&\widetilde{X_1}\arrow[ddd,"u_1"]\arrow[llu,"p_1^{13}",crossing over,leftarrow]\\
    &&&\widetilde{Y_{123}}\arrow[lld,"q_{23}^{123}"',very near start]\arrow[rrd,"q_{12}^{123}",very near start]\arrow[rd,"q_{13}^{123}"']\\
    &\widetilde{Y_{23}}\arrow[ld,"q_3^{23}"']\arrow[rd,"q_2^{23}",near start]&&&\widetilde{Y_{13}}\arrow[uuu,"u_{13}",crossing over,leftarrow]\arrow[lllld,"q_3^{13}"',crossing over,near start]&\widetilde{Y_{12}}\arrow[llld,"q_2^{12}",crossing over] \arrow[rd,"q_1^{12}"]\\
    \widetilde{Y_3}&&\widetilde{Y_2}\arrow[uuu,"u_2",near end,leftarrow,crossing over]&&&&\widetilde{Y_1}\arrow[llu,"q_1^{13}",crossing over,leftarrow]\\
  \end{tikzcd}\]
  in $\widetilde{\mathcal{C}}$ such that $q_{2}^{12},q_{3}^{13},q_{3}^{23},q_{13}^{123},q_{23}^{123}$ are also in $\mathscr{W}$. Temporary put
  \[P(f_1):=(q_2^{12*})^{-1}q_1^{12*},\]
  \[P(f_2):=(q_3^{23*})^{-1}q_2^{23*},\]
  \[P(f_1f_2):=(q_3^{13*})^{-1}q_1^{13*}.\]
  We have the invertible $2$-morphism
  \[P_{f_1,f_2}:P(f_2)P(f_1)\longrightarrow P(f_1f_2)\]
  in $\mathcal{D}$ as in (\ref{1.3}), and we have the invertible $2$-morphism
  \[Ex:H(f_1)u_1^*\longrightarrow u_2^*P(f_1)\]
  in $\mathcal{D}$ given by
  \[(p_2^{12*})^{-1}p_1^{12*}u_1^*\stackrel{\sim}\longrightarrow (p_2^{12*})^{-1}u_{12}^*q_1^{12*}\stackrel{\sim}\longrightarrow u_2^*(q_2^{12*})^{-1}q_1^{12*}.\]
  Similarly, we have the invertible $2$-morphisms
  \[H(f_2)u_2^*\stackrel{Ex}\longrightarrow u_3^*P(f_2),\]
  \[H(f_1f_2)u_1^*\stackrel{Ex}\longrightarrow u_3^*P(f_1f_2).\]
  in $\mathcal{D}$. Now, we have the diagram
  \begin{equation}\label{1.4.1}\begin{tikzcd}
    H(f_2)H(f_1)u_1^*\arrow[r,"Ex"]\arrow[d,"H_{f_1,f_2}"]&H(f_2)u_2^*P(f_1)\arrow[r,"Ex"]&u_3^*P(f_2)P(f_1)\arrow[d,"P_{f_1,f_2}"]\\
    H(f_1f_2)u_1^*\arrow[rr,"Ex"]&&u_3^*P(f_1f_2)
  \end{tikzcd}\end{equation}
  in $\mathcal{D}$. It is a big outside diagram of the diagram
  \[\begin{tikzpicture}[baseline= (a).base]
    \node[scale=.62] (a) at (0,0)
    {\begin{tikzcd}
    (p_3^{23*})^{-1}p_2^{23*}(p_2^{12*})^{-1}p_1^{12*}u_1^*\arrow[r,"\sim"]\arrow[d,"\sim"]&(p_3^{23*})^{-1}(p_{23}^{123*})^{-1}p_{12}^{123*}p_1^{12*}u_1^*\arrow[r,"\sim"]\arrow[d,"\sim"] &(p_3^{123*})^{-1}p_1^{123*}u_1^*\arrow[r,"\sim"]\arrow[dd,"\sim"]&(p_3^{13*})^{-1}(p_{13}^{123*})^{-1}p_{13}^{123*}p_1^{13*}u_1^*\arrow[r,"\sim"]\arrow[d,"\sim"]& (p_3^{13*})^{-1} p_1^{13*}u_1^*\arrow[d,"\sim"]\\
    (p_3^{23*})^{-1}p_2^{23*}(p_2^{12*})^{-1}u_{12}^*q_1^{12*}\arrow[r,"\sim"]\arrow[d,"\sim"]&(p_3^{23*})^{-1}(p_{23}^{123*})^{-1}p_{12}^{123*}u_{12}^*q_1^{12*}\arrow[d,"\sim"]&& (p_3^{13*})^{-1}(p_{13}^{123*})^{-1}p_{13}^{123*}u_{13}^*q_1^*\arrow[r,"\sim"]\arrow[d,"\sim"]&(p_3^{13*})^{-1}u_{13}^*q_1^{13*}\arrow[dd,"{\rm id}"]\\
    (p_3^{23*})^{-1}p_2^{23*}u_2^*(q_2^{12*})^{-1}q_1^{12*}\arrow[d,"\sim"]& (p_3^{23*})^{-1}(p_{23}^{123*})^{-1}u_{123}^*q_{12}^{123*}q_1^{12*}\arrow[r,"\sim"]\arrow[d,"\sim"]& (p_3^{123*})^{-1}u_{123}^*q_1^{123*}\arrow[dd,"\sim"]\arrow[r,"\sim"]&(p_3^{13*})^{-1}(p_{13}^{123*})^{-1}u_{123}^*q_{13}^{123*}q_1^{13*}\arrow[d,"\sim"]\\
    (p_3^{23*})^{-1}u_{23}^*q_2^{23*}(q_2^{12*})^{-1}q_1^{12*}\arrow[r,"\sim"]\arrow[d,"\sim"]& (p_3^{23*})^{-1}u_{23}^*(q_{23}^{123*})^{-1}q_{12}^{123*}q_1^{12*} \arrow[d,"\sim"] && (p_3^{13*})^{-1}u_{13}^*(q_{13}^{123*})^{-1}q_{13}^{123*}q_1^{13*}\arrow[d,"\sim"]\arrow[r,"\sim"]& (p_3^{13*})^{-1}u_{13}^*q_1^{13*}\arrow[d,"\sim"]\\
    u_3^*(q_3^{23*})^{-1}q_2^{23*}(q_2^{12*})^{-1}q_1^{12*}\arrow[r,"\sim"]&  u_3^*(q_3^{23*})^{-1}(q_{23}^{123*})^{-1}q_{12}^{123*}q_1^{12*} \arrow[r,"\sim"]& u_3^*(q_3^{123*})^{-1}q_1^{123*}\arrow[r,"\sim"]& u_3^*(q_3^{13*})^{-1}(q_{13}^{123*})^{-1}q_{13}^{123*}q_1^{13*}\arrow[r,"\sim"]& u_3^*(q_3^{13*})^{-1}q_1^{13*}
  \end{tikzcd}};
  \end{tikzpicture}\]
  in $\mathcal{D}$, and each small square commutes since $G$ is a pseudofunctor. Thus (\ref{1.4.1}) commutes.
\end{none}
\begin{none}\label{1.5}
  Now, under the conditions from (A--1) to (A--3), we will show that $H$ constructed in (\ref{1.3}) satisfies the axioms of pseudofunctors.
  \begin{enumerate}[(1)]
    \item Let $X_3\stackrel{f_2}\rightarrow X_2\stackrel{f_1}\rightarrow X_1$ be morphisms in $\mathcal{C}$. When $f_1$ is the identity morphism, we have to show that the diagram
    \[\begin{tikzcd}
      H(f_2)H(f_1)\arrow[d,"H_{f_1,f_2}"]\arrow[r,"H_{f_1}"]&H(f_2)\arrow[d,equal]\\
      H(f_1f_2)\arrow[r,equal]&H(f_2)
    \end{tikzcd}\]
    in $\mathcal{D}$ commutes. It follows from applying (\ref{1.4}) to the commutative diagram
    
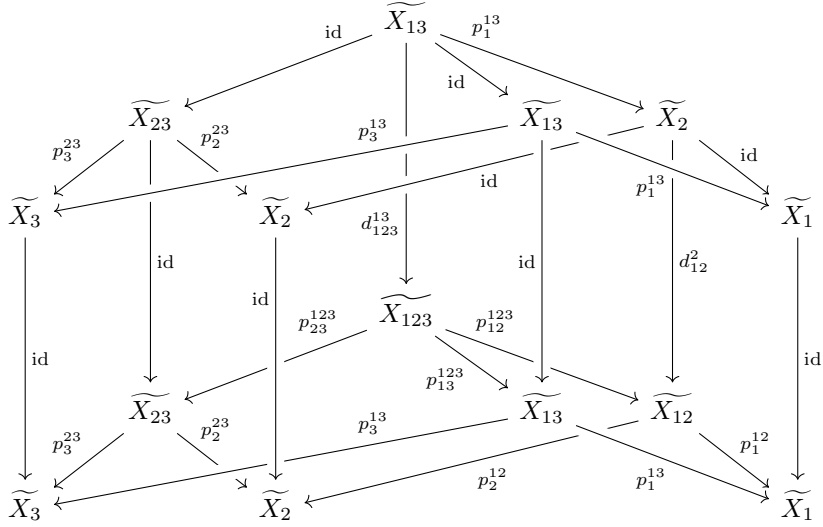
\[\begin{tikzcd}
    &&&\widetilde{X_{13}}\arrow[lld,"{\rm id}"',very near start]\arrow[rrd,"p_{1}^{13}",very near start]\arrow[rd,"{\rm id}"']\arrow[ddd,"d_{123}^{13}"',near end]\\
    &\widetilde{X_{23}}\arrow[ddd,"{\rm id}"]\arrow[ld,"p_3^{23}"']\arrow[rd,"p_2^{23}",near start]&&&\widetilde{X_{13}}\arrow[lllld,"p_3^{13}"',crossing over,near start]&\widetilde{X_2}\arrow[ddd,"d_{12}^2"]\arrow[llld,"{\rm id}",crossing over] \arrow[rd,"{\rm id}"]\\
    \widetilde{X_3}\arrow[ddd,"{\rm id}"]&&\widetilde{X_2}&&&&\widetilde{X_1}\arrow[ddd,"{\rm id}"]\arrow[llu,"p_1^{13}",crossing over,leftarrow]\\
    &&&\widetilde{X_{123}}\arrow[lld,"p_{23}^{123}"',very near start]\arrow[rrd,"p_{12}^{123}",very near start]\arrow[rd,"p_{13}^{123}"']\\
    &\widetilde{X_{23}}\arrow[ld,"p_3^{23}"']\arrow[rd,"p_2^{23}",near start]&&&\widetilde{X_{13}}\arrow[uuu,"{\rm id}",crossing over,leftarrow]\arrow[lllld,"p_3^{13}"',crossing over,near start]&\widetilde{X_{12}}\arrow[llld,"p_2^{12}",crossing over] \arrow[rd,"p_1^{12}"]\\
    \widetilde{X_3}&&\widetilde{X_2}\arrow[uuu,"{\rm id}",near end,leftarrow,crossing over]&&&&\widetilde{X_1}\arrow[llu,"p_1^{13}",crossing over,leftarrow]\\
    \end{tikzcd}\]
    in $\widetilde{\mathcal{C}}$.
    \item Let $X_3\stackrel{f_2}\rightarrow X_2\stackrel{f_1}\rightarrow X_1$ be morphisms in $\mathcal{C}$. When $f_2$ is the identity morphism, as in (1), the diagram
    \[\begin{tikzcd}
      H(f_2)H(f_1)\arrow[d,"H_{f_1,f_2}"]\arrow[r,"H_{f_2}"]&H(f_1)\arrow[d,equal]\\
      H(f_1f_2)\arrow[r,equal]&H(f_1)
    \end{tikzcd}\]
    in $\mathcal{D}$ commutes.
    \item Let $X_4\stackrel{f_3}\rightarrow X_3\stackrel{f_2}\rightarrow X_2\stackrel{f_1}\rightarrow X_1$ be morphisms in $\mathcal{C}$. Consider the commutative diagram
    
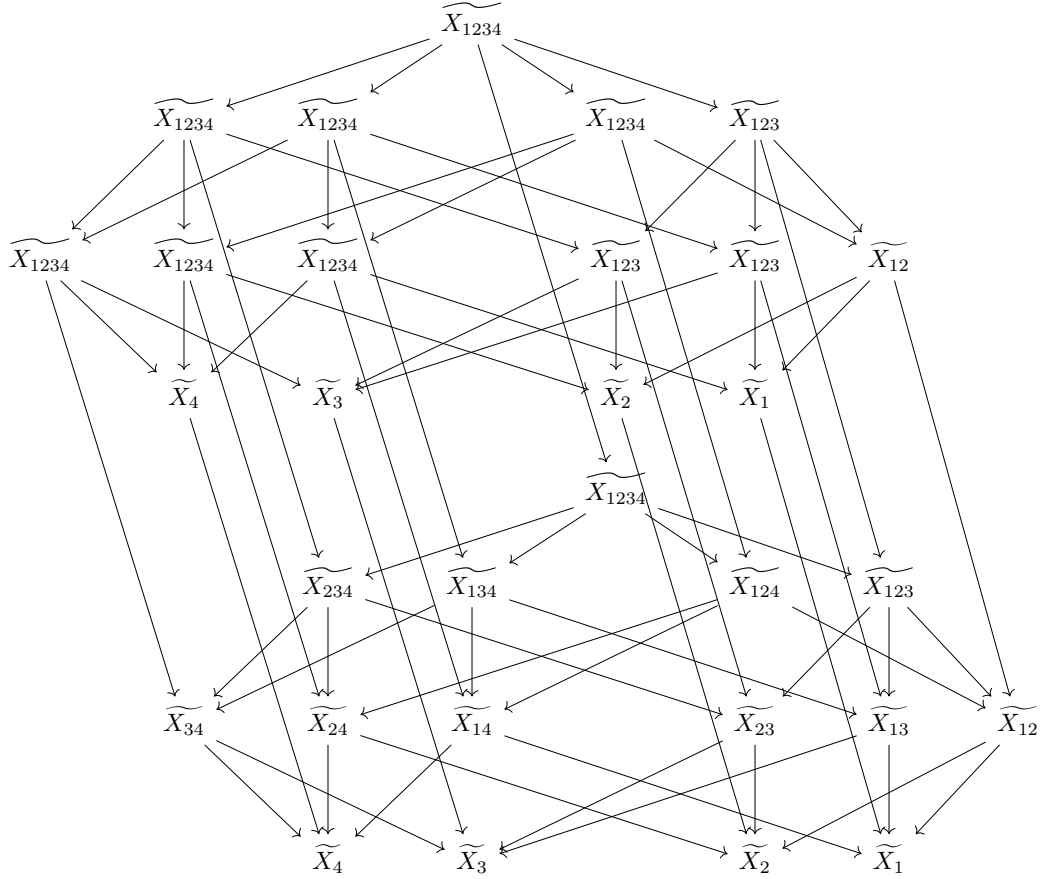
\[\begin{tikzpicture}[baseline= (a).base]
      \node[scale=.95] (a) at (0,0)
      {\begin{tikzcd}
      &&&\widetilde{X_{1234}}\arrow[lld]\arrow[ld]\arrow[rd]\arrow[rrd]\arrow[ddddddr]\\
      &\widetilde{X_{1234}}\arrow[ddddddr]\arrow[ldd]\arrow[dd]\arrow[rrrdd]&\widetilde{X_{1234}}\arrow[ddddddr] \arrow[lldd]\arrow[dd]\arrow[rrrdd]&&\widetilde{X_{1234}}\arrow[ddddddr]\arrow[llldd]\arrow[lldd] \arrow[rrdd]&\widetilde{X_{123}} \arrow[ddddddr]\arrow[ldd]\arrow[dd]\arrow[rdd]\\
      \\
      \widetilde{X_{1234}}\arrow[ddddddr]\arrow[rdd]\arrow[rrdd]&\widetilde{X_{1234}}\arrow[ddddddr]\arrow[dd]\arrow[rrrdd]&\widetilde{X_{1234}}\arrow[ddddddr]\arrow[ldd]\arrow[rrrdd] &&\widetilde{X_{123}}\arrow[ddddddr]\arrow[lldd]\arrow[dd]&\widetilde{X_{123}}\arrow[ddddddr]\arrow[llldd]\arrow[dd]&\widetilde{X_{12}}\arrow[ddddddr]\arrow[lldd]\arrow[ldd]\\
      \\
      &\widetilde{X_4}\arrow[ddddddr]&\widetilde{X_3}\arrow[ddddddr]&&\widetilde{X_2}\arrow[ddddddr]&\widetilde{X_1}\arrow[ddddddr]\\
      &&&&\widetilde{X_{1234}}\arrow[lld]\arrow[ld]\arrow[rd]\arrow[rrd]\\
      &&\widetilde{X_{234}}\arrow[ldd]\arrow[dd]\arrow[rrrdd]&\widetilde{X_{134}}\arrow[lldd]\arrow[dd]\arrow[rrrdd]&&\widetilde{X_{124}}\arrow[llldd]\arrow[lldd] \arrow[rrdd]&\widetilde{X_{123}} \arrow[ldd]\arrow[dd]\arrow[rdd]\\
      \\
      &\widetilde{X_{34}}\arrow[rdd]\arrow[rrdd]&\widetilde{X_{24}}\arrow[dd]\arrow[rrrdd]&\widetilde{X_{14}}\arrow[ldd]\arrow[rrrdd] &&\widetilde{X_{23}}\arrow[lldd]\arrow[dd]&\widetilde{X_{13}}\arrow[llldd]\arrow[dd]&\widetilde{X_{12}}\arrow[lldd]\arrow[ldd]\\
      \\
      &&\widetilde{X_4}&\widetilde{X_3}&&\widetilde{X_2}&\widetilde{X_1}
      \end{tikzcd}};
    \end{tikzpicture}\]
    in $\mathcal{C}$. Here, each arrow is the identity morphism or a composition of morphisms in the diagram in (A--3). Let
    \[Q(f_1),Q(f_2),Q(f_3),Q(f_1f_2),Q(f_1f_3),Q(f_1f_2f_3),Q_{f_1,f_2},Q_{f_2,f_3}.Q_{f_1f_2,f_3},Q_{f_1,f_2f_3}\]
    denote the $1$-morphisms and $2$-morphisms obtained as in (\ref{1.3}) using the upper part of the above diagram instead of the diagram in (A--3). As in (\ref{1.4}), we also have the invertible $2$-morphisms
    \[Q(f_1)\stackrel{Ex}\longrightarrow H(f_1),\quad Q(f_2)\stackrel{Ex}\longrightarrow H(f_2),\quad Q(f_3)\stackrel{Ex}\longrightarrow H(f_3),\]
    \[Q(f_1f_2)\stackrel{Ex}\longrightarrow H(f_1f_2),\quad Q(f_2f_3)\stackrel{Ex}\longrightarrow H(f_2f_3),\quad Q(f_1f_2f_3)\stackrel{Ex}\longrightarrow H(f_1f_2f_3)\]
    in $\mathcal{D}$. Now applying (loc.\ cit), we see that in the diagram
    \[\begin{tikzcd}
      Q(f_3)Q(f_2)Q(f_1)\arrow[rd,"Ex"]\arrow[rr,"Q_{f_1,f_2}"]\arrow[dd,"Q_{f_2,f_3}"]&&Q(f_3)Q(f_1f_2)\arrow[rd,"Ex"]\arrow[dd,"Q_{f_1f_2,f_3}",near end]\arrow[rd,"Ex"]\\
      &H(f_3)H(f_2)H(f_1)\arrow[rr,"H_{f_1,f_2}",near start,crossing over]&&H(f_3)H(f_1f_2)\arrow[dd,"H_{f_1f_2,f_3}"]\\
      Q(f_2f_3)Q(f_1)\arrow[rr,"Q_{f_1,f_2f_3}",near end]\arrow[rd,"Ex"]&&Q(f_1f_2f_3)\arrow[rd,"Ex"]\\
      &H(f_2f_3)H(f_1)\arrow[uu,"H_{f_2,f_3}"',leftarrow,crossing over,near start]\arrow[rr,"H_{f_1,f_2f_3}"]&&H(f_1f_2f_3)
    \end{tikzcd}\]
    in $\mathcal{D}$, the small squares other than the front side and back side squares commute. Since the purpose is to show that the front side square commutes, we only need to show that the back side square commutes. Put
    \[p_2^{123}=p_2^{12}p_{12}^{123},\quad p_3^{123}=p_3^{23}p_{23}^{123},\quad p_3^{1234}=p_3^{123}p_{123}^{1234},\quad p_4^{1234}=p_4^{34}p_{34}^{234}p_{234}^{1234},\]
    and consider the commutative diagram
    \[\begin{tikzcd}
      &\widetilde{X_{1234}}\arrow[ld,"p_4^{1234}"]\arrow[rd,"p_3^{1234}"]\arrow[rr,"p_{123}^{1234}"]&&\widetilde{X_{123}}\arrow[ld,"p_3^{123}"]\arrow[rr,"p_{12}^{123}"] \arrow[rd,"p_2^{123}"]&&\widetilde{X_{12}} \arrow[ld,"p_2^{12}"]\arrow[rd,"p_1^{12}"]\\
      \widetilde{X_4}&&\widetilde{X_3}&&\widetilde{X_2}&&\widetilde{X_1}
    \end{tikzcd}\]
    in $\widetilde{\mathcal{C}}$. The back side square is then the diagram
    \[\begin{tikzcd}
      (p_4^{1234*})^{-1}p_3^{1234*}(p_3^{123*})^{-1}p_2^{123*}(p_2^{12*})^{-1}p_1^{12*}\arrow[r,"\sim"]\arrow[d,"\sim"]&(p_4^{1234*})^{-1}p_3^{1234*}(p_3^{123*})^{-1}p_{12}^{123*}p_1^{12*} \arrow[d,"\sim"]\\
      (p_4^{1234*})^{-1}p_{123}^{1234*}p_2^{123*}(p_2^{12*})^{-1}p_1^{12*}\arrow[r,"\sim"]&(p_4^{1234*})^{-1}p_{123}^{1234*}p_{12}^{123*}p_1^{12^*}
    \end{tikzcd}\]
    in $\mathcal{D}$, and it commutes since $G$ is a pseudofunctor.
  \end{enumerate}
  We have proved all axioms of pseudofunctors for $H$, so $H$ is a pseudofunctor.
\end{none}
\begin{none}\label{1.6}
  The remaining part to construct (\ref{1.2.1}) is the construction of $\alpha:G\rightarrow H\circ F$. Under the conditions from (A--1) to (A--3), we can construct $\alpha$ as follows. For any object $\widetilde{X_1}$ of $\widetilde{\mathcal{C}}$, put $\alpha(\widetilde{X_1})$ as the identity $1$-morphism
  \[G(\widetilde{X_1})\rightarrow (H\circ F)(\widetilde{X_1}).\]
  For any morphism $\widetilde{X_2}\stackrel{\widetilde{f_1}}\rightarrow \widetilde{X_1}$ in $\widetilde{C}$, by (A--1), we have the diagram
    \[\begin{tikzcd}
      &\widetilde{X_{12}}\arrow[ld,"p_2^{12}"']\arrow[rd,"p_1^{12}"]\\
      \widetilde{X_2}&&\widetilde{X_1}
    \end{tikzcd}\]
  in $\widetilde{\mathcal{C}}$ such that $p_2^{12}$ has a section $d_{12}^2$ with $p_1^{12}d_{12}^2=\widetilde{f_1}$. Then put $\alpha_{\widetilde{f_1}}$ as the invertible $2$-morphism
  \[\widetilde{f_1}^*\stackrel{\sim}\longrightarrow d_{12}^{2*}p_1^{12*}\stackrel{\sim}\longrightarrow d_{12}^{2*}p_2^{12*}(p_2^{12*})^{-1}p_1^{12*}\stackrel{\sim}\longrightarrow (p_2^{12*})^{-1}p_1^{12*}=H(f_1)\]
  in $\mathcal{D}$ where $f_1=F(\widetilde{f_1})$.

  Now, we will verify the axioms of pseudonatural transformations for $\alpha$.
  \begin{enumerate}[(1)]
    \item Let $\widetilde{X_2}\stackrel{\widetilde{f_1}}\rightarrow \widetilde{X_1}$ be a morphism in $\widetilde{C}$, and put $f_1=F(\widetilde{f_1})$. When $\widetilde{f_1}$ is the identity morphism, we have to show that the diagram
    \[\begin{tikzcd}
      \widetilde{f_1}^*\arrow[r,"\alpha_{\widetilde{f_1}}"]\arrow[rd,"\sim"']&H(f_1)\arrow[d,"H_{f_1}"]\\
      &{\rm id}
    \end{tikzcd}\]
    in $\mathcal{D}$ commutes. It follows from construction.
    \item Let $\widetilde{X_3}\stackrel{\widetilde{f_2}}\rightarrow \widetilde{X_2}\stackrel{\widetilde{f_1}}\rightarrow \widetilde{X_1}$ be morphisms in $\widetilde{C}$, and put
    \[f_1=F(\widetilde{f_1}),\quad f_2=F(\widetilde{f_2}).\]
    We have to show that the diagram
    \begin{equation}\label{1.6.1}\begin{tikzcd}
      \widetilde{f_2}^*\widetilde{f_1}^*\arrow[d,"\sim"]\arrow[r,"\alpha_{\widetilde{f_2}}\alpha_{\widetilde{f_1}}"]&H(f_2)H(f_1)\arrow[d,"H_{f_1,f_2}"]\\
      (\widetilde{f_1}\widetilde{f_2})^*\arrow[r,"\alpha_{\widetilde{f_1}\widetilde{f_2}}"]&H(f_1f_2)
    \end{tikzcd}\end{equation}
    in $\mathcal{D}$ commutes. Since $p_2^{12*}$, $p_3^{23*}$, $p_{13}^{123*}$, and $p_{23}^{123*}$ are invertible, we have isomorphisms
    \[d_{12}^{2*}\stackrel{\sim}\longrightarrow (p_2^{12*})^{-1},\quad d_{23}^{3*}\stackrel{\sim}\longrightarrow (p_3^{23*})^{-1},\]
    \[d_{123}^{13*}\stackrel{\sim}\longrightarrow (p_{13}^{123*})^{-1},\quad d_{123}^{23*}\stackrel{\sim}\longrightarrow (p_{23}^{123*})^{-1}.\]
    Then we have the diagram
    \[\begin{tikzcd}
      \widetilde{f_2}^*\widetilde{f_1}^*\arrow[ddd,"\sim"]\arrow[r,"\sim"]&d_{23}^{3*}p_2^{23*}d_{12}^{2*}p_1^{12*}\arrow[d,"\sim"]\arrow[r,"\sim"]& (p_3^{23*})^{-1}p_2^{23*}(p_2^{12*})^{-1}p_1^{12*}\arrow[d,"\sim"]\\
      &d_{23}^{3*}d_{123}^{23*}p_{12}^{123*}p_1^{12*}\arrow[d,"\sim"]\arrow[r,"\sim"]&(p_3^{23*})^{-1}(p_{23}^{123*})^{-1}p_{12}^{123*}p_1^{12*}\arrow[d,"\sim"]\\
      &d_{13}^{3*}d_{123}^{13*}p_{13}^{123*}p_1^{13*}\arrow[d,"\sim"]\arrow[r,"\sim"]&(p_3^{13*})^{-1}(p_{13}^{123*})^{-1}p_{13}^{123*}p_1^{13*}\arrow[d,"\sim"]\\
      (\widetilde{f_1}\widetilde{f_2})^*\arrow[r,"\sim"]&d_{13}^{3*}p_1^{13*}\arrow[r,"\sim"]&(p_3^{13*})^{-1}p_1^{13*}
    \end{tikzcd}\]
    in $\mathcal{D}$, and its big outside diagram is (\ref{1.6.1}). Since each small squares of the above diagram commute, (\ref{1.6.1}) commutes.
  \end{enumerate}
  We have verified the axioms of pseudonatural transformations for $\alpha$, so $\alpha$ is a pseudonatural transformation. Since $\alpha(\widetilde{X_1})$ is an equivalence for any object $\widetilde{X_1}$ of $\widetilde{C}$, $\alpha$ is a pseudonatural equivalence. Thus we have proven the following theorem.
\end{none}
\begin{thm}\label{1.7}
  Assume the conditions from {\rm (A--1)} to {\rm (A--3)}, and assume that $g^*$ is an equivalence for any $g\in \mathscr{W}$. Then there is a pseudofunctor $H:\mathcal{C}\rightarrow \mathcal{D}$ and a pseudonatural equivalence $\alpha:G\rightarrow H\circ F$ making the diagram
  \[\begin{tikzcd}
    \widetilde{\mathcal{C}}\arrow[d,"G"']\arrow[r,"F"]\arrow[rd,phantom,"{\rotatebox[origin=c]{0}{$\Leftrightarrow$}}_\alpha",very near start]&\mathcal{C}\arrow[ld,"H"]\\
    \mathcal{D}&\,
  \end{tikzcd}\]
  commutes.
\end{thm}
\begin{none}\label{1.9}
  Now, we will prove that the above construction is functorial in the following sense.
\end{none}
\begin{thm}\label{1.8}
   Assume the conditions from {\rm (A--1)} to {\rm (A--3)}, and assume that $g^*$ is an equivalence for any $g\in \mathscr{W}$. Let
  \[\begin{tikzcd}
    \widetilde{\mathcal{C}}\arrow[dd,"G'"]\arrow[rr,"F"]\arrow[dd,bend right,"G"']&&\mathcal{C}\arrow[lldd,"H"']\arrow[lldd,bend left,"H'"]\\
    \\
    \mathcal{D}
  \end{tikzcd}\]
  be a diagram in $\mathcal{C}$ extending {\rm (\ref{1.1.1})} where $G'$ and $H'$ are contravariant pseudofunctors such that $G'(g)$ is invertible for any morphism $g\in\mathscr{W}$, and let
  \[\begin{tikzcd}
    G\arrow[r,"\alpha"]\arrow[d,"\beta"]&H\circ F\\
    G'\arrow[r,"\alpha'"]&H'\circ F
  \end{tikzcd}\]
  be a diagram of pseudofunctors {\rm (}i.e., $\alpha$, $\beta$, and $\alpha'$ are pseudonatural transformations{\rm )} such that $\alpha$ and $\alpha'$ are pseudonatural equivalences. Then there is a unique pseudonatural transformation $\gamma:H\rightarrow H'$ up to isomorphisms such that the induced pseudonatural transformation
  \[H\circ F\rightarrow H'\circ F\]
  makes the above diagram commutative.
\end{thm}
\begin{proof}
  We may assume that $\alpha$ and $\alpha'$ are the identity $2$-morphisms for simplicity. We will first construct $\gamma$. For any object $X_1$ of $\mathcal{C}$, put $\gamma(X_1)$ as the $1$-morphism $H(X_1)\rightarrow H'(X_1)$ given by
  \[H(X_1)= G(\widetilde{X_1})\stackrel{\beta(\widetilde{X_1})}\longrightarrow G'(\widetilde{X_1})=H'(X_1).\]  For any morphism $X_2\stackrel{f_1}\rightarrow X_1$ in $\mathcal{C}$, consider the diagram
  \[\begin{tikzcd}
      &X_{12}\arrow[ld,"r_2^{12}"']\arrow[rd,"r_1^{12}"]\\
      X_2&&X_1
  \end{tikzcd}\]
  in $\mathcal{C}$ that is the image of the diagram in (A--1) via $F$. Put
  \[H_1(f_1)=H(r_2^{12})^{-1}H(r_1^{12}),\quad H_1'(f_1)=H'(r_2^{12})^{-1}H'(r_1^{12}).\]
  Then we have the invertible $2$-morphism
  \[\mu_{f_1}:H_1(f_1)\rightarrow H(f_1)\]
  in $\mathcal{D}$ given by the composition
  \[H(r_2^{12})^{-1}H(r_1^{12})\stackrel{\sim}\longrightarrow H((r_2^{12})^{-1})H(r_1^{12})\stackrel{H_{r_1^{12},(r_2^{12})^{-1}}}\longrightarrow H(r_1^{12}(r_2^{12})^{-1})\stackrel{\sim}\longrightarrow H(f_1).\]
  We similarly have the invertible $2$-morphism
  \[\mu_{f_1}':H_1'(f_1)\rightarrow H'(f_1)\]
  in $\mathcal{D}$. We also have the invertible $2$-morphism
  \[\eta_{f_1}:\gamma(X_2)H_1(f_1)\longrightarrow H_1'(f_1)\gamma(X_1)\]
  in $\mathcal{D}$ given by the composition
  \[\begin{split}
    \gamma(X_2)H(r_2^{12})^{-1}H(r_1^{12})&=\beta(\widetilde{X_2})G(p_2^{12})^{-1}G(p_1^{12})\rightarrow G'(p_2^{12})^{-1}\beta(\widetilde{X_{12}})G(p_1^{12})\\
    &\stackrel{\beta_{p_1^{12}}}\rightarrow G'(p_2^{12})^{-1}G(p_1^{12})\beta(\widetilde{X_1})=H_1'(r_2^{12})^{-1}H_1'(r_1^{12})\gamma(X_1).
  \end{split}\]
  Here, the second arrow is obtained by applying $H'(r_2^{12})^{-1}$ to the left side and $H(r_1^{12})^{-1}$ to the right side of the $2$-morphism
  \[\beta_{p_2^{12}}^{-1}:H'(r_2^{12})\beta(\widetilde{X_2})\longrightarrow \beta(\widetilde{X_{12}})H(r_2^{12})\]
  in $\mathcal{D}$. Then put $\gamma_{f_1}$ as the $2$-morphism
  \[\gamma(X_2)H(f_1)\rightarrow H'(f_2)\gamma(X_1)\]
  in $\mathcal{D}$ given by the composition
  \[\gamma(X_2)H(f_1)\stackrel{\mu_{f_1}^{-1}}\longrightarrow \gamma(X_2)H_1(f_1)\stackrel{\eta_{f_1}}\longrightarrow H_1'(f_1)\gamma(X_1)\stackrel{\mu_{f_1}'}\longrightarrow H'(f_1)\gamma(X_1).\]
  Now, we will verify the axioms of pseudonatural transformations for $\gamma$.
  \begin{enumerate}[(1)]
    \item For any morphism $X_2\stackrel{f_1}\rightarrow X_1$, when $X_1=X_2$ and $f_1$ is the identity morphism, we have to show that the diagram
    \begin{equation}\label{1.8.1}\begin{tikzcd}
      \gamma(X_1)H(f_1)\arrow[rr,"\gamma_{f_1}"]\arrow[rd,"H_{f_1}"']&&H'(f_1)\gamma(X_1)\arrow[ld,"H_{f_1}'"]\\
      &\gamma(X_1)
    \end{tikzcd}\end{equation}
    in $\mathcal{D}$ commutes. By (A--1), $p_2^{12}$ has a section $d_{12}^2$ such that $\widetilde{f_1}=p_1^{12}d_{12}^2$, and the above diagram is the big outside diagram of the diagram
    \[\begin{tikzpicture}[baseline= (a).base]
    \node[scale=.85] (a) at (0,0)
    {\begin{tikzcd}
      \gamma(X_1)H(f_1)\arrow[r,"\mu_{f_1}^{-1}"]\arrow[rdd,equal]&\gamma(X_1)G(p_2^{12})^{-1}G(p_1^{12})\arrow[rr]\arrow[d,"\sim"]
      &&G'(p_2^{12})^{-1}G'(p_1^{12})\gamma(X_1)\arrow[d,"\sim"]\arrow[r,"\mu_{f_1}'"]&H'(f_1)\gamma(X_1) \arrow[ldd,equal]\\
      &\gamma(X_1)G(d_2^{12})G(p_1^{12})\arrow[rr,"{\beta_{d_2^{12}}\beta_{p_1^{12}}}"]\arrow[d,"{G_{p_1^{12},d_2^{12}}}"]&&G'(d_2^{12})G'(p_1^{12}) \gamma(X_1)\arrow[d,"{G'_{p_1^{12},d_2^{12}}}"']\\
      &\gamma(X_1)G(\widetilde{f_1})\arrow[rr,"\beta_{\widetilde{f_1}}"]\arrow[rd,"G_{\widetilde{f_1}}"']&&G'(\widetilde{f_1})\gamma(X_1)\arrow[ld,"G_{\widetilde{f_1}}"]\\
      &&\gamma(X_1)
    \end{tikzcd}};
    \end{tikzpicture}\]
    in $\mathcal{D}$, whose small diagrams commute. Here, the not denoted arrow in the first row is the $2$-morphism induced by $\gamma$. Thus (\ref{1.8.1}) commutes.
    \item For any morphisms $X_3\stackrel{f_2}\rightarrow X_2\stackrel{f_1}\rightarrow X_1$ in $\mathcal{C}$, consider the commutative diagram
      \[\begin{tikzcd}
      &&X_{123}\arrow[ld,"r_{23}^{123}"']\arrow[d,"r_{13}^{123}"]\arrow[rd,"r_{12}^{123}"]\\
      &X_{23}\arrow[ld,"r_3^{23}"']\arrow[rd,"r_2^{23}"',near end]&X_{13}\arrow[lld,crossing over,"r_3^{13}",near end]&X_{12}\arrow[ld,"r_2^{12}",near end]\arrow[rd,"r_1^{12}"]\\
      X_3&&X_2&&X_1\arrow[llu,"r_1^{13}",near start,leftarrow,crossing over]
    \end{tikzcd}\]
    in $\mathcal{C}$ that is the image of the diagram in (A--2) via $F:\widetilde{\mathcal{C}}\rightarrow \mathcal{C}$. Put
    \[ U(f_2)=H(r_2^{12})^{-1}H(r_1^{12}),\quad U(f_1f_2)=H(r_3^{13})^{-1}H(r_1^{13}),\]
    \[ U'(f_2)=H'(r_2^{12})^{-1}H'(r_1^{12}),\quad U'(f_1f_2)=H'(r_3^{13})^{-1}H'(r_1^{13}),\]
    and consider the $2$-morphisms
    \[U(f_2)U(f_1)\stackrel{U_{f_1,f_2}}\longrightarrow U(f_1f_2),\quad U'(f_2)U(f_1)\stackrel{U'_{f_1,f_2}}\longrightarrow U'(f_1f_2)\]
    in $\mathcal{D}$ defined as in (\ref{1.3}(4)). We also have the $2$-morphisms
    \[U(f_2)\stackrel{\mu_{f_2}}\longrightarrow H(f_2),\quad U(f_1f_2)\stackrel{\mu_{f_1f_2}}\longrightarrow H(f_1f_2),\]
    \[U'(f_2)\stackrel{\mu_{f_2}'}\longrightarrow H'(f_2),\quad U'(f_1f_2)\stackrel{\mu'_{f_1f_2}}\longrightarrow H'(f_1f_2),\]
    \[\gamma(X_3)U(f_2)\stackrel{\eta_{f_2}}\longrightarrow U'(f_2)\gamma(X_2),\quad \gamma(X_3)U(f_1f_2)\stackrel{\eta_{f_1f_2}}\longrightarrow U'(f_1f_2)\gamma(X_1)\]
    in $\mathcal{D}$ defined as $\mu(f_1)$, $\mu'(f_1)$, and $\eta(f_1)$. Then we have the diagram
    \[\begin{tikzcd}
      \gamma(X_3)H(f_2)H(f_1)\arrow[r,"{H_{f_1,f_2}}"]\arrow[d,"\mu_{f_2}^{-1}\mu_{f_1}^{-1}"]&H(f_1f_2)\arrow[d,"\mu_{f_1f_2}^{-1}"]\\
      \gamma(X_3)H_1(f_2)H_1(f_1)\arrow[d,"\eta_{f_2}"]\arrow[r,"{U_{f_1,f_2}}"]&U(f_1f_2)\arrow[dd,"\eta_{f_1f_2}"]\\
      U'(f_2)\gamma(X_2)U(f_1)\arrow[d,"\eta_{f_2}"]\\
      U'(f_2)U'(f_1)\gamma(X_1)\arrow[r,"{U_{f_1,f_2}'}"]\arrow[d,"\mu_{f_2}'\mu_{f_1}'"]&U'(f_1f_2)\arrow[d,"\mu_{f_1f_2}'"]\\
      H'(f_2)H'(f_1)\gamma(X_1)\arrow[r,"{H_{f_1,f_2}'}"]&H'(f_1f_2)
    \end{tikzcd}\]
    in $\mathcal{D}$. The middle diagram commutes as in (\ref{1.4}) since $\beta:G\rightarrow G'$ is a pseudonatural transformation. The top and bottom diagrams commute as in the proof that (\ref{1.6.1}) commutes. Thus the above diagram commutes, so we obtain the commutative diagram
    \[\begin{tikzcd}
      \gamma(X_3)H(f_2)H(f_1)\arrow[r,"\gamma_{f_2}"]\arrow[d,"H_{f_1,f_2}"]&H'(f_2)\gamma(X_2)H(f_1)\arrow[r,"\gamma_{f_1}"]&H'(f_2)H'(f_1)\gamma(X_1)\arrow[d,"H_{f_1,f_2}'"]\\
      \gamma(X_3)H(f_1f_2)\arrow[rr,"\gamma_{f_1f_2}"]&&H'(f_1f_2)\gamma(X_1)
    \end{tikzcd}\]
    in $\mathcal{D}$.
  \end{enumerate}
  We have verified the axioms of pseudonatural transformations for $\gamma$, so $\gamma$ is a pseudonatural transformation. The commutativity of the diagram
  \[\begin{tikzcd}
    G\arrow[r,"\alpha"]\arrow[d,"\beta"]&H\circ F\arrow[d,"\gamma"]\\
    G'\arrow[r,"\alpha'"]&H'\circ F
  \end{tikzcd}\]
  of pseudofunctors follows from construction. Thus the remaining part is the uniqueness of $\gamma$ up to isomorphisms. If $\gamma':H\rightarrow H'$ is another pseudonatural transformation such that $\alpha'\beta$ is isomorphic to $\gamma'\alpha$, then $\gamma\alpha$ and $\gamma'\alpha$ are isomorphic. From (\ref{3.11}), the data of the isomorphism is a collection of invertible $2$-morphisms
  \[\Theta_{\widetilde{Y_1}}:\gamma\alpha(\widetilde{Y_1})\rightarrow \gamma'\alpha(\widetilde{Y_1})\]
  in $\mathcal{D}$ for any object $\widetilde{Y_1}$ of $\widetilde{\mathcal{C}}$ such that for any morphism $\widetilde{g_1}:\widetilde{Y_2}\rightarrow \widetilde{Y_1}$ in $\widetilde{\mathcal{C}}$, the diagram
  \begin{equation}\label{1.8.3}\begin{tikzcd}
    \gamma\alpha(\widetilde{Y_1})\arrow[r,"\Theta_{\widetilde{Y_1}}"]\arrow[d,"(\gamma\alpha)_{\widetilde{g_1}}"]&\gamma'\alpha(\widetilde{Y_1}) \arrow[d,"(\gamma'\alpha)_{\widetilde{g_1}}"]\\
    \gamma\alpha(\widetilde{Y_2})\arrow[r,"\Theta_{\widetilde{Y_2}}"]&\gamma'\alpha(\widetilde{Y_2})
  \end{tikzcd}\end{equation}
  in $\mathcal{D}$ commutes. Then for any object $X_1$ of $\mathcal{C}$, the data gives the collection of invertible $2$-morphisms
  \[\Phi_{X_1}:\gamma(X_1)\rightarrow \gamma'(X_1)\]
  in $\mathcal{D}$. The condition that $\Gamma$ is a modification is the commutativity of the diagram
  \[\begin{tikzcd}
    \gamma(X_1)\arrow[r,"\Phi_{X_1}"]\arrow[d,"\gamma_{f_1}"]&\gamma'(X_2)\arrow[d,"\gamma_{f_1}'"]\\
    \gamma(X_2)\arrow[r,"\Phi_{X_2}"]&\gamma'(X_2)
  \end{tikzcd}\]
  in $\mathcal{D}$ for any morphism $f_1:X_2\rightarrow X_1$ in $\mathcal{C}$. It follows from showing that the diagram
  \[\begin{tikzcd}
    \gamma(X_1)\arrow[r,"\Phi_{X_1}"]\arrow[d,"\gamma_{r_1^{12}}"]&\gamma'(X_2)\arrow[d,"\gamma_{r_1^{12}}'"]\\
    \gamma(X_{12})\arrow[r,"\Phi_{X_{12}}"]\arrow[d,"\gamma_{r_2^{12}}^{-1}"]&\gamma'(X_{12})\arrow[d,"\gamma_{r_2^{12}}'^{-1}"]\\
    \gamma(X_2)\arrow[r,"\Phi_{X_2}"]&\gamma'(X_2)
  \end{tikzcd}\]
  in $\mathcal{D}$ commutes, which is true because of the commutativity of (\ref{1.8.3}). Thus $\Phi:\gamma\rightarrow \gamma'$ is a modification, which is an isomorphism since $\Phi_{X_1}$ is invertible for any object $X$ of $\mathcal{C}$.
\end{proof}
\section{Categories with immersions}
\begin{df}\label{2.1}
  A {\it category with immersions} is the data of $(\mathcal{C},c:\mathcal{I}\rightarrow \mathcal{J})$ with the following axioms.
  \begin{enumerate}
    \item[(D--1)] $\mathcal{C}$ is a category, and $\mathcal{I}$ and $\mathcal{J}$ are classes of monomorphisms in $\mathcal{C}$, and $c:\mathcal{I}\rightarrow \mathcal{J}$ is a function. Assume that $\mathcal{C}$ has fiber products, an initial object, and a final object. A morphism in $\mathcal{I}$ is called a {\it closed immersions}, and a morphism in $\mathcal{J}$ is called an {\it open immersions}. When $i\in \mathcal{I}$, $c(i)$ is called the {\it complement} of $i$.
    \item[(D--2)] For any closed immersion $Z\rightarrow S$ in $\mathcal{C}$, the morphism $c(Z\rightarrow S)$ is an open immersion whose target is $S$.
    \item[(D--3)] Isomorphisms are closed immersions and open immersions, and closed immersions and open immersions are stable by compositions and pullbacks.
    \item[(D--4)] For any commutative diagram
    \begin{equation}\label{2.1.1}\begin{tikzcd}
      Z_2\arrow[d,"v_1"]\arrow[r,"i_2"]&S_2\arrow[d,"u_1"]\\
      Z_1\arrow[r,"i_1"]&S_1
    \end{tikzcd}\end{equation}
    in $\mathcal{C}$ where $i_1$ and $i_2$ are closed immersions with complements $j_1:U_1\rightarrow S_1$ and $j_2:U_2\rightarrow S_2$ respectively, the projection $U_1\times_{S_1}S_2\rightarrow S_2$ has a factorization
    \[U_1\times_{S_1}S_2\stackrel{j_2'}\rightarrow U_2\stackrel{j_2}\rightarrow S_2\]
    for some open immersion $j_2'$. Moreover, $j_2'$ is an isomorphism when the diagram (\ref{2.1.1}) is Cartesian.
  \end{enumerate}
  If no confusion seems likely to arise, we simply write $\mathcal{C}$ for $(\mathcal{C},c:\mathcal{I}\rightarrow \mathcal{J})$.
\end{df}
\begin{df}\label{2.3}
  Let $\mathcal{C}$ be a category with immersions, and let $S$ be an object of $\mathcal{C}$. We say that $S$ is {\it separated} if the diagonal morphism $S\rightarrow S\times S$ is a closed immersion.
\end{df}
\begin{prop}\label{2.4}
  Let $\mathcal{C}$ be a category with immersions, and let $i_1:Z\rightarrow S_1$ and $i_2:Z\rightarrow S_2$ be closed immersions in $\mathcal{C}$. If $Z$ is separated, then the morphism
  \[i:Z\rightarrow S_1\times S_2\]
  induced by $i_1$ and $i_2$ is a closed immersion.
\end{prop}
\begin{proof}
  The morphism $i$ has the factorization
  \[Z\stackrel{\Delta}\longrightarrow Z\times Z\stackrel{{\rm id}\times i_2}\longrightarrow Z\times S_2\stackrel{i_1\times {\rm id}}\longrightarrow S_1\times S_2\]
  where $\Delta$ denotes the diagonal morphism. The first arrow is a closed immersion since $Z$ is separated, and the second and third arrows are closed immersions by (D--3). Thus the composition is also a closed immersion by (D--3).
\end{proof}
\begin{eg}\label{2.5}
  Let $\mathcal{C}$ be a category with immersions, and let $\mathcal{E}$ be a full subcategory of $\mathcal{C}$ such that
  \begin{enumerate}[(i)]
    \item for any object $Z_1$ of $\mathcal{C}$, there is a closed immersion $Z_1\rightarrow S_1$ with $S_1\in {\rm ob}\,\mathcal{E}$,
    \item for any open immersion $U_1\rightarrow S_1$ in $\mathcal{C}$, if $S_1$ is in $\mathcal{E}$, then $U_1$ is also in $\mathcal{E}$.
  \end{enumerate}
  We will construct some categories with immersions and functors that will be used frequently later.

  The category $\mathcal{E}$ has a structure of immersions as follows. A morphism $S_2\rightarrow S_1$ in $\mathcal{E}$ is a closed immersion if its image in $\mathcal{C}$ is a closed immersion. A morphism $S_2\rightarrow S_1$ in $\mathcal{E}$ is an open immersion if its image in $\mathcal{C}$ is an open immersion. The complement of a closed immersion is defined from the complement function in $\mathcal{C}$, which is possible because of (ii). The axioms (D--3) and (D--4) for $\mathcal{E}$ follows from those for $\mathcal{C}$.

  Consider the category $\widetilde{\mathcal{C}}$ where objects are morphisms $(Z_1\rightarrow S_1)$ in $\mathcal{C}$ and morphisms are commutative diagrams
  \[\begin{tikzcd}
    (Z_2\arrow[d,"v_1"]\arrow[r,"i_2"]&S_2)\arrow[d,"u_1"]\\
    (Z_1\arrow[r,"i_1"]&S_1)
  \end{tikzcd}\]
  in $\mathcal{C}$. Then $\widetilde{\mathcal{C}}$ has a structure of immersions as follows. A closed immersion in $\widetilde{\mathcal{C}}$ is a morphism that is isomorphic to a commutative diagram
  \begin{equation}\label{2.5.1}\begin{tikzcd}
    (Z_1\arrow[d,"i_1"]\arrow[r,"i_1"]&S_1)\arrow[d,"{\rm id}"]\\
    (S_1\arrow[r,"{\rm id}"]&S_1)
  \end{tikzcd}\end{equation}
  in $\mathcal{C}$ where $i_1$ is a closed immersion in $\mathcal{C}$. An open immersion in $\widetilde{\mathcal{C}}$ is a morphism that is isomorphic to a commutative diagram
  \begin{equation}\label{2.5.2}\begin{tikzcd}
    (U_1\arrow[d,"j_1"]\arrow[r,"{\rm id}"]&U_1)\arrow[d,"j_1"]\\
    (S_1\arrow[r,"{\rm id}"]&S_1)
  \end{tikzcd}\end{equation}
  in $\mathcal{C}$ where $j_1$ is an open immersion in $\mathcal{C}$. The complement of a closed immersion in $\widetilde{\mathcal{C}}$ given by (\ref{2.5.1}) is the open immersion in $\widetilde{\mathcal{C}}$ given by (\ref{2.5.2}) if $j_1$ is the complement of $i_1$. The axioms (D--3) and (D--4) for $\widetilde{\mathcal{C}}$ follows from those for $\mathcal{C}$.

  Let $S_1$ be an object of $\mathcal{C}$. We often write $S_1$ for the object $(S_1\rightarrow S_1)$ in $\widetilde{\mathcal{C}}$ when no confusion seems likely to arise. Let $S_2\rightarrow S_1$ be a morphism in $\mathcal{C}$. We often write $S_2\rightarrow S_1$ for the morphism in $\widetilde{\mathcal{C}}$ given by the commutative diagram
  \[\begin{tikzcd}
    (S_2\arrow[d,"f"]\arrow[r,"{\rm id}"]&S_2)\arrow[d,"f"]\\
    (S_1\arrow[r,"{\rm id}"]&S_1)
  \end{tikzcd}\]
  in $\mathcal{C}$ when no confusion seems likely to arise.

  Let $U:\mathcal{E}\rightarrow \widetilde{\mathcal{C}}$ denote the functor mapping object $S_1$ to $S_1$ and mapping morphism $S_2\rightarrow S_1$ to $S_2\rightarrow S_1$. Then let $F:\widetilde{\mathcal{C}}\rightarrow \mathcal{C}$ denote the functor mapping object $(Z_1\rightarrow S_1)$ to $Z_1$ and mapping morphism given by the diagram
  \[\begin{tikzcd}
    (Z_2\arrow[d,"v_1"]\arrow[r,"i_2"]&S_2)\arrow[d,"u_1"]\\
    (Z_1\arrow[r,"i_1"]&S_1)
  \end{tikzcd}\]
  in $\mathcal{C}$ to $v_1:Z_2\rightarrow Z_1$.

  Let $\mathscr{P}$ be a class of morphisms of $\mathcal{C}$ containing all isomorphisms and stable by compositions and pullbacks. Then we denote by $\widetilde{\mathscr{P}}$ the class of morphisms in $\widetilde{\mathcal{C}}$ given by a Cartesian diagram
  \[\begin{tikzcd}
    (Z_2\arrow[d,"v_1"]\arrow[r,"i_2"]&S_2)\arrow[d,"u_1"]\\
    (Z_1\arrow[r,"i_1"]&S_1)
  \end{tikzcd}\]
  such that $u_1$ is in $\mathscr{P}$. We also put $\mathscr{P}'=W^{-1}(\mathscr{P})$.
\end{eg}
\begin{rmk}\label{2.9}
  In (\ref{2.5}), we are particularly interested in the case is when $\mathcal{C}=\mathscr{S}$, $\mathcal{E}=\mathscr{S}^{sm}$, and $\mathscr{P}=Sm$.
\end{rmk}
\begin{none}\label{2.2}
  Under the notations and hypothesis of (\ref{2.5}), we will show the conditions from (A--1) to (A--3) so that we can use (\ref{1.7}) and (\ref{1.8}). Consider the category $\mathcal{C}'$ where objects are morphisms $(Z_1\rightarrow S_1)$ in $\mathcal{C}$ and morphisms are commutative diagrams
  \[\begin{tikzcd}
    (Z_2\arrow[d]\arrow[r]&S_2)\\
    (Z_1\arrow[r]&S_1)
  \end{tikzcd}\]
  in $\mathcal{C}$. Then $\mathcal{C}'$ is equivalent to $\mathcal{C}$, and we have the obvious functor
  \[F':\widetilde{\mathcal{C}}\rightarrow \mathcal{C}'.\]
  We will verify the conditions from (A--1) to (A--3) for $F'$.

  Let $X_4\stackrel{f_3}\rightarrow X_3\stackrel{f_2}\rightarrow X_2\stackrel{f_1}\rightarrow X_1$ be morphisms in $\mathcal{C}$ given by the diagram
  \begin{equation}\label{2.2.1}\begin{tikzcd}
    (Z_4\arrow[d]\arrow[r]&S_4)\\
    (Z_3\arrow[r]\arrow[d]&S_3)\\
    (Z_2\arrow[d]\arrow[r]&S_2)\\
    (Z_1\arrow[r]&S_1)
  \end{tikzcd}\end{equation}
  in $\mathcal{C}$. Then put
  \[\widetilde{X_1}=X_1,\quad \widetilde{X_2}=X_2,\quad \widetilde{X_3}=X_3,\quad \widetilde{X_4}=X_4,\]
  \[\widetilde{X_{12}}=(Z_2\rightarrow S_1\times S_2),\quad \widetilde{X_{13}}=(Z_3\rightarrow S_1\times S_3),\quad \widetilde{X_{14}}=(Z_4\rightarrow X_1\times X_4),\]
  \[\widetilde{X_{23}}=(Z_3\rightarrow S_2\times S_3),\quad \widetilde{X_{24}}=(Z_4\rightarrow S_2\times S_4),\quad \widetilde{X_{34}}=(Z_4\rightarrow X_3\times X_4),\]
  \[\widetilde{X_{123}}=(Z_3\rightarrow S_1\times S_2\times S_3),\quad \widetilde{X_{124}}=(Z_4\rightarrow S_1\times S_2\times S_4),\]
  \[\widetilde{X_{134}}=(Z_4\rightarrow S_1\times S_3\times S_4),\quad \widetilde{X_{234}}=(Z_4\rightarrow S_2\times S_3\times S_4),\]
  \[\widetilde{X_{1234}}=(Z_4\rightarrow S_1\times S_2\times S_3\times S_4)\]
  where the morphisms are induced by compositions of morphisms in (\ref{2.2.1}). By (\ref{2.4}), these are in $\widetilde{\mathcal{C}}$. We also have the morphisms
  \[p_1^{12}:\widetilde{X_{12}}\rightarrow \widetilde{X_1},\quad \ldots,\quad p_{234}^{1234}:\widetilde{X_{1234}}\rightarrow \widetilde{X_{234}},\]
  \[d_{12}^2:\widetilde{X_{2}}\rightarrow \widetilde{X_{12}},\quad \ldots,\quad d_{123}^{23}:\widetilde{X_{23}}\rightarrow \widetilde{X_{123}}\]
  induced by the projections
  \[S_1\times S_2\rightarrow S_1,\quad \ldots,\quad S_1\times S_2\times S_3\times S_4\rightarrow S_2\times S_3\times S_4\]
  and the graph morphisms
  \[S_2\times S_1\times S_2,\quad \ldots,\quad S_2\times S_3\rightarrow S_1\times S_2\times S_3\]
  respectively. Our construction of $\widetilde{X_1},\ldots,\widetilde{X_{1234}}$, $p_1^{12},\ldots,p_{234}^{1234}$, and $d_{12}^2,\ldots,d_{123}^{23}$ satisfy the conditions from (A--1) to (A--3) for $F'$.
\end{none}
\begin{df}\label{2.6}
  Let $\mathcal{C}$ be a category with immersions, let $\mathscr{P}$ be a class of morphisms in $\mathcal{C}$ containing all isomorphisms and stable by compositions and pullbacks, let $H:C\rightarrow {\rm Tri}^{\otimes}$ be a $\mathscr{P}$-premotivic pseudofunctor, let $i$ be a closed immersion in $\mathcal{C}$, and let $j$ denote the complement of $i$. We say that $H$ satisfies ${\rm (Loc}_i)$ if $i_*$ is fully faithful and the pair of functors $(i^*,j^*)$ is conservative.

  We also say that $H$ satisfies the {\it localization property}, denoted by (Loc), if
  \begin{enumerate}[(1)]
    \item For any Cartesian diagram
    \[\begin{tikzcd}
      W\arrow[d,"j'"]\arrow[r,"i'"]&U\arrow[d,"j"]\\
      Z\arrow[r,"i"]&S
    \end{tikzcd}\]
    where $i$ is a closed immersion with complement $j$, we have $H(W)=0$,
    \item $H$ satisfies $({\rm Loc}_i)$ for any closed immersion $i$ in $\mathcal{C}$.
  \end{enumerate}
\end{df}
\begin{none}\label{2.7}
  Under the notations and hypotheses of (\ref{2.6}), let $i:Z\rightarrow S$ be a closed immersion in $\mathcal{C}$, and let $j$ denote its complement. We list several consequences of (Loc) and the axioms from (B--1) to (B--4) whose proofs are the same as those in \cite[\S 2.3]{CD12}
  \begin{enumerate}[(1)]
    \item The functor $i_*$ admits a right adjoint $i^!$.
    \item There exists a unique natural transformation $\partial_i:i_*i^*\rightarrow j_\sharp j^*[1]$
        such that the triangle
        \[j_\sharp j^*\stackrel{ad'}\longrightarrow {\rm id}\stackrel{ad}\longrightarrow i_*i^*\stackrel{\partial_i}\longrightarrow j_\sharp j^*[1]\]
        is distinguished.
    \item There exists a unique natural transformation $\partial_i:j_*j^*\rightarrow i_* i^![1]$ such
        that the triangle
        \[i_* i^!\stackrel{ad'}\longrightarrow {\rm id}\stackrel{ad}\longrightarrow j_*j^*\stackrel{\partial_i}\longrightarrow i_* i^![1]\]
        is distinguished.
    \item We have $i^*j_\sharp=0$.
    \item The unit ${\rm id}\stackrel{ad}\longrightarrow j^*j_\sharp$ is an isomorphism.
  \end{enumerate}
\end{none}
\begin{none}\label{2.15}
  Under the notations and hypotheses of (\ref{2.6}), assume $\mathcal{C}=\mathscr{S}$. A consequence of (Loc) in \cite[3.3.4]{CD12} is as follows. Consider a Cartesian diagram
  \[\begin{tikzcd}
    U_{12}\arrow[d,"u_1"]\arrow[r,"u_2"]&U_2\arrow[d,"j_2"]\\
    U_1\arrow[r,"j_1"]&U
  \end{tikzcd}\]
  in $\mathscr{S}$ where $j_1$ and $j_2$ are open immersions, and put $j_{12}=j_1u_1$. Then there is a distinguished triangle
  \[j_{12\sharp}j_{12}^*\longrightarrow j_{1\sharp}j_1^*\oplus j_{2\sharp}j_2^*\stackrel{ad'+ad'}\longrightarrow {\rm id}\longrightarrow j_{12\sharp}j_{12}^*[1]\]
  where the first arrow is induced by the natural transformations
  \[j_{12\sharp}j_{12}^*\stackrel{\sim}\longrightarrow j_{1\sharp}u_{1\sharp}u_1^*j_1^*\stackrel{ad'}\longrightarrow j_{1\sharp}j_1^*,\]
  \[j_{12\sharp}j_{12}^*\stackrel{\sim}\longrightarrow j_{2\sharp}u_{2\sharp}u_2^*j_2^*\stackrel{-ad'}\longrightarrow j_{2\sharp}j_2^*.\]
  The distinguished triangled is called the Mayer-Vietoris distinguished triangle. Its right adjoint
  \[{\rm id}\longrightarrow j_{1*}j_1^*\oplus j_{2*}j_2^*\longrightarrow j_{12*}j_{12}^*\longrightarrow {\rm id}[1]\]
  is also a distinguished triangle.
\end{none}
\begin{prop}\label{2.8}
  Under the notations and hypotheses of {\rm (\ref{2.6})}, consider a Cartesian diagram
    \[\begin{tikzcd}
      Z_2\arrow[d,"v_1"]\arrow[r,"i_2"]&S_2\arrow[d,"u_1"]\\
      Z_1\arrow[r,"i_1"]&S_1
    \end{tikzcd}\]
    in $\mathcal{C}$ where $i_1$ is a closed immersion. Assume that $H$ satisfies {\rm (Loc)}. Then the exchange transformation
    \[u_1^*i_{1*}\stackrel{Ex}\longrightarrow i_{2*}v_1^*\]
    given by the composition
    \[u_1^*i_{1*}\stackrel{ad}\longrightarrow u_1^*i_{1*}v_{1*}v_1^*\stackrel{\sim}\longrightarrow u_1^*u_{1*}i_{2*}v_1^*\stackrel{ad'}\longrightarrow i_{2*}v_1^*\]
    is an isomorphism.
\end{prop}
\begin{proof}
  Consider the diagram
  \[\begin{tikzcd}
      Z_2\arrow[d,"v_1"]\arrow[r,"i_2"]&S_2\arrow[d,"u_1"]\arrow[r,leftarrow,"j_2"]&U_2\arrow[d,"w_1"]\\
      Z_1\arrow[r,"i_1"]&S_1\arrow[r,leftarrow,"j_1"]&U_1
    \end{tikzcd}\]
  in $\mathcal{C}$ where each square is Cartesian and $j_1$ denotes the complement of $i_1$. By (D--4), $j_2$ is the complement of $i_2$. To show the statement, by (Loc), it suffices to show that the natural transformations
  \begin{equation}\label{2.8.1}
    j_2^*u_1^*i_{1*}\stackrel{Ex}\longrightarrow j_2^*i_{2*}v_1^*,\quad     i_2^*u_1^*i_{1*}\stackrel{Ex}\longrightarrow i_2^*i_{2*}v_1^*
  \end{equation}
  are isomorphisms. The first arrow is an isomorphism since $j_2^*i_{2*}=0$ and $j_2^*u_1^*i_{1*}\cong w_1^*j_1^*u_{1*}=0$. Hence the remaining is to show that the second arrow is an isomorphism.

  Consider the commutative diagram
  \[\begin{tikzcd}
    v_1^*i_1^*i_{1*}\arrow[r,"\sim"]\arrow[rrd,"ad'"']&i_2^*u_1^*i_{1*}\arrow[r,"Ex"]&i_2^*i_{2*}v_1^*\arrow[d,"ad'"]\\
    &&v_1^*
  \end{tikzcd}\]
  of functors. The vertical arrow and diagonal arrow are isomorphisms by (Loc), so the right horizontal arrow is an isomorphism. Thus the second arrow of (\ref{2.8.1}) is an isomorphism.
\end{proof}
\begin{prop}\label{2.16}
  Under the notations and hypotheses of {\rm (\ref{2.8})}, assume that $u_1$ is in $\mathscr{P}$. Then the exchange transformation
    \[u_{1\sharp}i_{2*}\stackrel{Ex}\longrightarrow i_{1*}v_{1\sharp}\]
    given by the composition
    \[u_{1\sharp}i_{2*}\stackrel{ad}\longrightarrow i_{1*}i_1^*u_{1\sharp}i_{2*}\stackrel{Ex^{-1}}\longrightarrow i_{1*}v_{1\sharp}i_2^*i_{2*}\stackrel{ad'}\longrightarrow i_{1*}v_{1\sharp}\]
    is an isomorphism.
\end{prop}
\begin{proof}
  Note first that the second arrow in the composition is defined and an isomorphism by (\ref{2.8}). The third arrow is an isomorphism by (Loc). Thus the remaining is to show that the first arrow is an isomorphism.

  Consider the Cartesian diagram
  \[\begin{tikzcd}
    S_2\arrow[d,"u_1"]\arrow[r,"j_2",leftarrow]&U_2\arrow[d,"w_1"]\\
    S_1\arrow[r,"j_1",leftarrow]&U_1
  \end{tikzcd}\]
  in $\mathcal{C}$ where $j_1$ denotes the complement of $i_1$. Then by (D--4), $j_2$ is the complement of $i_2$. By (Loc), it suffices to show that the natural transformations
  \begin{equation}\label{2.16.1}
    j_1^*u_{1\sharp}i_{2*}\stackrel{Ex}\longrightarrow j_1^*i_{1*}v_{1\sharp},
  \end{equation}
  \begin{equation}\label{2.16.2}
    i_1^*u_{1\sharp}i_{2*}\stackrel{Ex}\longrightarrow i_1^*i_{1*}v_{1\sharp}
  \end{equation}
  are isomorphisms. By (B--3) and (\ref{2.7}(4)), we have
  \[j_1^*u_{1\sharp}i_{2*}\cong w_{1\sharp}j_2^*i_{2*}=0,\quad j_1^*i_{1*}v_{1\sharp}=0,\]
  so (\ref{2.16.1}) is an isomorphism. Thus the remaining is to show that (\ref{2.16.2}) is an isomorphism. Consider the commutative diagram
  \[\begin{tikzcd}
    v_{1\sharp}i_2^*i_{2*}\arrow[rrd,"ad'"']\arrow[r,"Ex"]&i_1^*u_{1\sharp}i_{2*}\arrow[r,"Ex"]&i_1^*i_{1*}v_{1\sharp}\arrow[d,"ad'"]\\
    &&v_{1\sharp}
  \end{tikzcd}\]
  of functors. The vertical arrow and diagonal arrow are isomorphisms by (Loc), and the left horizontal arrow is an isomorphism by (\ref{2.8}). Thus the right horizontal arrow is an isomorphism.
\end{proof}
\begin{df}\label{2.11}
  We will consider the following definitions in (\ref{2.10}).
  \begin{enumerate}[(1)]
    \item Let $\mathcal{T}$ be a triangulated category with small sums, and let $K$ be an object. We say that $K$ is {\it compact} if the functor
    \[{\rm Hom}_\mathcal{T}(K,-)\]
    commutes with small sums.
    \item Let $\mathcal{T}$ be a triangulated category, and let $\mathcal{F}$ be a family of objects of $\mathcal{T}$. We say that $\mathcal{F}$ {\it generates} $\mathcal{T}$ if the family of functors
        \[{\rm Hom}_{\mathcal{T}}(K,-)\]
        for $K\in \mathcal{F}$ is conservative.
  \end{enumerate}
\end{df}
\begin{df}\label{2.12}
  Let $\mathcal{D}$ be one of $\mathscr{S}$ and $\mathscr{S}^{sm}$, and let
  \[H:\mathcal{D}\rightarrow {\rm Tri}^{\otimes}\]
  be a $Sm$-premotivic pseudofunctor. For an object $S$ of $\mathcal{D}$, we denote by $1_S(1)$ the cone of the object $p_\sharp a_* 1_S[-2]$ where $p:{\bf A}_S^1\rightarrow S$ denotes the projection and $a:S\rightarrow {\bf A}_S^1$ denotes the zero section. Then for any object $K$ of $H(S)$ and positive integer $n$, put
  \[K(n)=K\otimes_S 1_S(1)\otimes_S \cdots \otimes_S 1_S(1)\]
  where the number of $1_S(1)$ is $n$. When there is an object $1_S(-1)$ of $H(S)$ such that $1_S(1)\otimes 1_S(-1)\cong 1_S$, put
  \[K(-n)=K\otimes_S 1_S(-1)\otimes_S \cdots \otimes_S 1_S(-1)\]
  where the number of $1_S(-1)$ is $n$.
\end{df}
\begin{df}\label{2.10}
  Let $\mathcal{D}$ be one of $\mathscr{S}$ and $\mathscr{S}^{sm}$, and let
  \[H:\mathcal{D}\rightarrow {\rm Tri}^{\otimes}\]
  be a $Sm$-premotivic pseudofunctor. Consider the following axioms.
  \begin{enumerate}
    \item[(B--5)] For any object $S$ of $\mathcal{D}$, the counit
    \[p_\sharp p^*\stackrel{ad'}\longrightarrow {\rm id}\] is an isomorphism where $p$ denotes the projection $\mathbb{A}^1\times S\rightarrow S$.
    \item[(B--6)] For any object $S$ of $\mathcal{D}$, $1_S(1)$ is $\otimes$-invertible, i.e., there is an object $K$ of $H(S)$ such that $1_S(1)\otimes_S K\cong 1_S$.
    \item[(B--7)] Assume (B--6). For any object $S$ of $\mathcal{D}$, the family of objets $\mathcal{F}_S$ generates $H(D)$. Here, $\mathcal{F}_S$ denotes the family of objects $f_\sharp 1_X(d)[n]$ for smooth morphism $f:X\rightarrow S$ and $(d,n)\in \mathbb{Z}\times \mathbb{Z}$.
    \item[(B--8)] Assume (B--6). For any morphism $f:X\rightarrow S$ of $\mathcal{D}$ and $(d,n)\in \mathbb{Z}\times \mathbb{Z}$, $f_\sharp 1_X(d)[n]$ is compact.
    \item[(B--9)] For any proper morphism $f$ in $\mathcal{D}$, $f_*$ has a right adjoint, denoted by $f^!$.
  \end{enumerate}
  Following \cite[2.4.45]{CD12}, we say that $H$ is {\it motivic} if it satisfies (Loc) and the axioms (B--5), (B--6), and (B--9).
\end{df}
\begin{rmk}\label{2.13}
  In (\ref{2.10}), if $H$ satisfies (Loc), (B--5), (B--6), (B--7), and (B--8), then $H$ satisfies (B--9) by \cite[2.4.47]{CD12}, which means that $H$ is motivic.
\end{rmk}
\begin{none}\label{2.14}
  Under the notations and hypotheses of (\ref{2.10}), if $H$ is motivic, then by \cite[2.4.50]{CD12}, $H$ satisfies the Grothendieck six operations formalism in [loc.\ cit].
\end{none}
\section{Proof of (\ref{0.4}), part II}
\begin{none}\label{4.1}
  Under the notations and hypotheses of (\ref{2.5}), consider the diagram
  \[\begin{tikzcd}
    \mathcal{E}\arrow[rd,"V"']\arrow[r,"U"]&\widetilde{\mathcal{C}}\\
    &{\rm Tri}^\otimes
  \end{tikzcd}\]
  of $2$-categories where $V$ is a $\mathscr{P}'$-premotivic pseudofunctor satisfying (Loc). One purpose of this section is to show that the diagram can be extended to a commutative diagram
  \begin{equation}\label{4.1.1}\begin{tikzcd}
    \mathcal{E}\arrow[rd,"V"']\arrow[r,"U"]&\widetilde{\mathcal{C}}\arrow[d,"G"]\arrow[ld,phantom,"{\rotatebox[origin=c]{0}{$\Leftrightarrow$}}_{\delta}",very near start]\\
    \,&{\rm Tri}^\otimes
  \end{tikzcd}\end{equation}
  of $2$-categories where $G$ is a $\widetilde{\mathscr{P}}$-premotivic pseudofunctor satisfying (Loc) and $\delta:V\rightarrow G\circ U$ is a $\mathscr{P}'$-premotivic pseudonatural transformation.
\end{none}
\begin{df}\label{4.2}
  Under the notations and hypotheses of (\ref{2.5}), for any closed immersion $i_1:Z_1\rightarrow S_1$, we denote by $\rho_{i_1}$ the closed immersion in $\widetilde{\mathcal{C}}$ that is isomorphic to a commutative diagram
  \[\begin{tikzcd}
    (Z_1\arrow[d,"i_1"]\arrow[r,"i_1"]&S_1)\arrow[d,"{\rm id}"]\\
    (S_1\arrow[r,"{\rm id}"]&S_1)
  \end{tikzcd}\]
  in $\mathcal{C}$.
\end{df}
\begin{none}\label{4.3}
  Under the notations and hypotheses of (\ref{4.1}), for any object $(Z_1\stackrel{i_1}\rightarrow S_1)$ of $\widetilde{\mathcal{C}}$, put $G(Z_1\rightarrow S_1)$ as the full subcategory of $V(S_1)$ consisting of objects $K$ such that $j_1^*K=0$ where $j_1:U_1\rightarrow S_1$ denotes the complement of $i_1$. Since $j_1^*$ is triangulated, $G(Z_1\rightarrow S_1)$ is also triangulated, and since $j_1^*$ is monoidal, $G(Z_1\rightarrow S_1)$ has the symmetric closed monoidal structure induced from that of $V(S_1)$.

  Then we denote by $\rho_{i_1*}$ the embedding
  \[G(Z_1\rightarrow S_1)\rightarrow V(S_1).\]
  Now, we will prove some results about $\rho_{i_1*}$ as follows.
\end{none}
\begin{prop}\label{4.5}
  Under the notations and hypotheses of {\rm (\ref{4.3})},
  \begin{enumerate}[{\rm (1)}]
    \item the functor $\rho_{i_1*}$ has a left adjoint, denoted by $\rho_{i_1}^*$, and the functor $\rho_{i_1*}$ has a right adjoint, denoted by $\rho_{i_1}^!$,
    \item if we denote by $j_1$ the complement of $i_1$, then we have a distinguished triangle
    \[j_{1\sharp}j_1^*\stackrel{ad}\longrightarrow {\rm id}\stackrel{ad'}\longrightarrow \rho_{i_1*}\rho_{i_1}^*\longrightarrow j_{1\sharp}j_1^*[1],\]
    \item the counit
    \[\rho_{i_1}^*\rho_{i_1*}\stackrel{ad'}\longrightarrow {\rm id}\]
    is an isomorphism,
    \item $\rho_{i_1}^*j_{1\sharp}=0$.
  \end{enumerate}
\end{prop}
\begin{proof}
  Let us first show that $\rho_{i_1*}$ has a right adjoint denoted by $\rho_{i_1}^*$. For any object $K$ of $V(S_1)$, {\it choose} a cone of the counit
  \[j_{1\sharp}j_1^*K\stackrel{ad'}\longrightarrow K\]
  in $V(S_1)$, and it is denoted by $\rho_{i_1}^*K$. Since $j_1^*\rho_{i_1}^*K=0$ by (\ref{2.7}(4)), $\rho_{i_1}^*K$ is in $G(Z_1\rightarrow S_1)$. We also fix a distinguished triangle
  \begin{equation}\label{4.4.1}
    j_{1\sharp}j_1^*K\stackrel{ad'}\longrightarrow K\longrightarrow \rho_{i_1*}\rho_{i_1}^*K\longrightarrow j_{1\sharp}j_1^*K[1]
  \end{equation}
  in $V(S_1)$, and we denote by $ad$ (resp. $\delta_K$) the second arrow (resp.\ third arrow).

  Then let us define $\rho_{i_1}^*a$ for any morphism $a:K\rightarrow L$ in $V(S_1)$. Consider the commutative diagram
  \[\begin{tikzcd}
    j_{1\sharp}j_1^*K\arrow[r,"ad'"]\arrow[d,"a"]&K\arrow[d,"ad"]\arrow[r,"ad"]&\rho_{i_1*}\rho_{i_1^*}K\arrow[r,"\delta_K"]&j_{1\sharp}j_1^*K[1]\arrow[d,"a"]\\
    j_{1\sharp}j_1^*L\arrow[r,"ad'"]&L\arrow[r,"ad"]&\rho_{i_1*}\rho_{i_1^*}L\arrow[r,"\delta_L"]&j_{1\sharp}j_1^*K[1]
  \end{tikzcd}\]
  in $V(S_1)$ where the rows are the distinguished triangles. We will show that there is a unique morphism
  \[b:\rho_{i_1*}\rho_{i_1^*}K\longrightarrow \rho_{i_1*}\rho_{i_1^*}L\]
  in $V(S_1)$ making the above diagram commutative. This will be the definition of $\rho_{i_1}^*a$, and the functoriality of $\rho_{i_1}^*$ follows from the uniqueness.

  The existence of $b$ is obtained by an axiom of triangulated categories, so the remaining is the uniqueness of $b$. If we have two morphisms
  \[b,b':\rho_{i_1*}\rho_{i_1^*}K\longrightarrow \rho_{i_1*}\rho_{i_1^*}L\]
  making the above diagram commutative, put $c=b-b'$. Since $c\circ {ad}=0$, we have $c=d\delta_K$ for some morphism
  \[d:j_{1\sharp}j_1^*K[1]\rightarrow \rho_{i_1*}\rho_{i_1}^*L\]
  in $V(S_1)$. Hence to show $b=b'$, it suffices to show
  \[{\rm Hom}_{V(S_1)}(j_{1\sharp}j_1^*K[1],\rho_{i_1*}\rho_{i_1}^*L)=0,\]
  which is true since $j_1^*\rho_{i_1*}=0$.

  Thus $b$ is unique, so $\rho_{i_1}^*$ has a structure of functor. We also obtain the fact that $ad:{\rm id}\longrightarrow \rho_{i*}\rho_i^*$ is functorial. Now, we will show that $\rho_{i_1}^*$ is left adjoint to $\rho_{i_1*}$. For any object $K$ of $G(Z_1\rightarrow S_1)$, the morphism
  \[\rho_{i_1*}K\stackrel{ad}\longrightarrow \rho_{i_1*}\rho_{i_1}^*\rho_{i_1*}K\]
  is an isomorphism since its cone is $j_{1\sharp}j_1^*\rho_{i_1*}K[1]$, which is zero. Since $\rho_{i_1*}$ is fully faithful, the inverse of the above morphism induces the isomorphism
  \[\rho_{i_1}^*\rho_{i_1*}K\longrightarrow K,\]
  in $G(Z_1\rightarrow S_1)$. It is denoted by $ad'$. Since $ad$ is functorial, $ad'$ is also functorial. From construction, the compositions
  \[\rho_{i_1}^*\stackrel{ad}\longrightarrow \rho_{i_1}^*\rho_{i_1*}\rho_{i_1}^*\stackrel{ad'}\longrightarrow \rho_{i_1}^*\]
  \[\rho_{i_1*}\stackrel{ad}\longrightarrow \rho_{i_1*}\rho_{i_1}^*\rho_{i_1*}\stackrel{ad'}\longrightarrow \rho_{i_1*}\]
  are isomorphisms. Thus $\rho_{i_1}^*$ is right adjoint to $\rho_{i_1*}$. By a similar method, $\rho_{i_1*}$ has a left adjoint, denoted by $\rho_{i_1}^!$.

  Applying $\rho_{i_1}^*$ to the distinguished triangle (\ref{4.4.1}) and using the fact that the natural transformation
  \[\rho_{i_1}^*\rho_{i_1*}\stackrel{ad'}\longrightarrow {\rm id}\]
  is an isomorphism, we have $\rho_{i_1}^*j_{1\sharp}=0$.
\end{proof}
\begin{none}\label{4.7}
  Under the notations and hypotheses of (\ref{4.3}), consider a morphism $\widetilde{X_2}\stackrel{f_1}\rightarrow \widetilde{X_1}$ given by a commutative diagram
  \begin{equation}\label{4.7.1}\begin{tikzcd}
    (Z_2\arrow[r,"i_2"]\arrow[d,"v_1"]&S_2\arrow[d,"u_1"])\\
    (Z_1\arrow[r,"i_1"]&S_1)
  \end{tikzcd}\end{equation}
  in $\mathcal{C}$. We denote by $j_1$ (resp.\ $j_2$) the complement of $i_1$ (resp.\ $i_2$). Then we have the commutative diagram
  \[\begin{tikzcd}
    S_2\arrow[d,"u_1"]\arrow[r,"j_2",leftarrow]&U_2\arrow[r,"j_2'",leftarrow]&S_2\times_{S_1}U_1\arrow[d,"w_1'"]\\
    S_1\arrow[rr,"j_1",leftarrow]&&U_1
  \end{tikzcd}\]
  in $\mathcal{C}$ where $w_1'$ denotes the projections and $j_2'$ is obtained by (D--4). By (B--3) and (\ref{4.5}), we have
  \[\rho_{i_2}^*u_1^*j_{1\sharp}\cong \rho_{i_2}^*j_{2\sharp}j_{2\sharp}'w_1'^*=0.\]
  Thus the natural transformation
  \[\rho_{i_2}^*u_1^*\stackrel{ad}\longrightarrow \rho_{i_2}^*u_1^*\rho_{i_1*}\rho_{i_1}^*\]
  is an isomorphism since its cone is $0$ by (\ref{4.5}).

  When (\ref{4.7.1}) is Cartesian, we have the Cartesian diagram
  \[\begin{tikzcd}
    S_2\arrow[d,"u_1"]\arrow[r,"j_2",leftarrow]&U_2\arrow[d,"w_1"]\\
    S_1\arrow[r,"j_1"]&U_1
  \end{tikzcd}\]
  in $\mathcal{C}$. Then we have
  \[j_2^*u_1^*\rho_{i_1*}\cong w_1^*j_1^*\rho_{i_1*}=0.\]
  Thus by (\ref{4.5}), the natural transformation
  \[u_1^*\rho_{i_1*}\stackrel{ad}\longrightarrow \rho_{i_2*}\rho_{i_2}^*u_1^*\rho_{i_1*}\]
  is an isomorphism. Thus we have proven the following result.
\end{none}
\begin{prop}\label{4.8}
  Under the notations and hypotheses of {\rm (\ref{4.7})},
  \begin{enumerate}[{\rm (1)}]
    \item the natural transformation
    \[\rho_{i_2}^*u_1^*\stackrel{ad}\longrightarrow \rho_{i_2}^*u_1^*\rho_{i_1*}\rho_{i_1}^*\]
    is an isomorphism,
    \item when {\rm (\ref{4.7.1})} is Cartesian, the natural transformation
  \[u_1^*\rho_{i_1*}\stackrel{ad}\longrightarrow \rho_{i_2*}\rho_{i_2}^*u_1^*\rho_{i_1*}\]
  is an isomorphism.
  \end{enumerate}
\end{prop}
\begin{none}\label{4.6}
  Now, we will construct $G$ introduced in (\ref{4.1}). Under the notations and hypotheses of (\ref{4.3}), for any morphism $\widetilde{f_1}$ in $\mathscr{C}$ given by a commutative diagram
  \[\begin{tikzcd}
    (Z_2\arrow[r,"i_2"]\arrow[d,"v_1"]&S_2\arrow[d,"u_1"])\\
    (Z_1\arrow[r,"i_1"]&S_1)
  \end{tikzcd}\]
  in $\mathcal{C}$, put
  \[G(\widetilde{f_1}):=\rho_{i_2}^*u_1^*\rho_{i_1*}.\]
  It is also denoted by $\widetilde{f_1}^*$ following the convention in (\ref{3.7}). We will first construct a monoidal structure for $G(\widetilde{f_1})$. Since $u_1^*$ and $\rho_{i_1*}$ are monoidal, it suffices to construct a monoidal structure for $\rho_{i_2}^*$. Let $j_2:U_2\rightarrow S_2$ denotes the complement of $i_2$, and let $K$ and $L$ be objects of $V(S_2)$. Consider the distinguished triangle
  \[\rho_{i_2*}\rho_{i_2}^*K\otimes_{S_2} j_{2\sharp}j_2^*L\stackrel{ad'}\longrightarrow \rho_{i_2*}\rho_{i_2}^*K\otimes_{S_2} L\stackrel{ad}\longrightarrow \rho_{i_2*}\rho_{i_2}^*K\otimes_{S_2} \rho_{i_2*}\rho_{i_2}^*L\longrightarrow \rho_{i_2*}\rho_{i_2}^*K\otimes_{S_2} j_{2\sharp}j_2^*L[1].\]
  in $V(S_2)$. The first object is zero by (B--4) for $V$ since $j_2^*\rho_{i_2*}=0$. Thus the natural transformation
  \begin{equation}\label{4.6.2}
    \rho_{i_2*}\rho_{i_2}^*K\otimes_{S_2} L\stackrel{ad}\longrightarrow \rho_{i_2*}\rho_{i_2}^*K\otimes_{S_2} \rho_{i_2*}\rho_{i_2}^*L
  \end{equation}
  is an isomorphism. Then consider the commutative diagram
  \[\begin{tikzcd}
    j_{2\sharp}j_2^*(K\otimes_{S_2} L)\arrow[d]\arrow[r,"ad'"]&K\otimes_{S_2} L\arrow[d,"{\rm id}"]\arrow[r,"ad"]&\rho_{i_2*}\rho_{i_2}^*(K\otimes_{S_2}L)\arrow[r]&j_{2\sharp}j_2^*(K\otimes_{S_2} L)\arrow[d][1]\\
    j_{2\sharp}j_2^*K\otimes_{S_2} L\arrow[r,"ad'"]&K\otimes_{S_2}L\arrow[r,"ad"]&\rho_{i_2*}\rho_{i_2}^*K\otimes_{S_2}L\arrow[r]&j_{2\sharp}j_2^*K\otimes_{S_2} L[1]
  \end{tikzcd}\]
  in $V(S_2)$ where the left side and the right side vertical arrows are given by
  \[j_{2\sharp}j_2^*(K\otimes_{S_2} L)\stackrel{\sim}\longrightarrow j_{2\sharp}(j_2^*K\otimes_{U_2}j_2^*L)\stackrel{Ex}\longrightarrow j_{2\sharp}j_2^*K\otimes_{S_2}L.\]
  By the proof of (\ref{4.5}) (more precisely, the proof of uniqueness of $b$ in (loc.\ cit)), there is a unique isomorphism
  \[\rho_{i_2*}\rho_{i_2}^*(K\otimes_{S_2}L)\longrightarrow \rho_{i_2*}\rho_{i_2}^*K\otimes_{S_2}L\]
  in $V(S_1)$ making the above diagram commutes. With (\ref{4.6.2}), we obtain the isomorphism
  \[\rho_{i_2*}\rho_{i_2}^*(K\otimes_{S_2}L)\longrightarrow \rho_{i_2*}\rho_{i_2}^*K\otimes_{S_2}\rho_{i_2*}\rho_{i_2}^*L\]
  in $V(S_1)$. Since $\rho_{i_2*}$ is monoidal and fully faithful, we obtain the isomorphism
  \[\rho_{i_2}^*(K\otimes_{S_2}L)\longrightarrow \rho_{i_2}^*K\otimes_{Z_2}\rho_{i_2}^*L\]
  in $G(\widetilde{X_2})$. This satisfies the coherence condition by the proof of (\ref{4.5}) again.
\end{none}
\begin{none}\label{4.12}
   Let us continue the argument in (\ref{4.6}). We will show that $G$ has a contravariant pseudofunctor structure. For any object $\widetilde{X_1}=(Z_1\stackrel{i_1}\rightarrow S_1)$ in $\widetilde{\mathcal{C}}$, consider the identity morphism ${\rm id}_{\widetilde{X_1}}:\widetilde{X_1}\rightarrow \widetilde{X_1}$ in $\widetilde{\mathcal{C}}$. The natural transformation
  \[G_{{\rm id}_{\widetilde{X_1}}}:{\rm id}_{\widetilde{X_1}}^*\longrightarrow {\rm id}\]
  is given by
  \[\rho_{i_1}^*\rho_{i_1*}\stackrel{ad'}\longrightarrow {\rm id},\]
  which is an isomorphism by (\ref{4.5}). For any morphisms $\widetilde{X_3}\stackrel{f_2}\longrightarrow \widetilde{X_2}\stackrel{f_1}\longrightarrow \widetilde{X_1}$ in $\widetilde{\mathcal{C}}$ given by a diagram
  \[\begin{tikzcd}
    (Z_3\arrow[d,"v_2"]\arrow[r,"i_3"]&S_3\arrow[d,"u_2"])\\
    (Z_2\arrow[r,"i_2"]\arrow[d,"v_1"]&S_2\arrow[d,"u_1"])\\
    (Z_1\arrow[r,"i_1"]&S_1)
  \end{tikzcd}\]
  in $\mathcal{C}$, the natural isomorphism
  \[G_{f_1,f_2}:\widetilde{f_2}^*\widetilde{f_1}^*\longrightarrow (\widetilde{f_1}\widetilde{f_2})^*\]
  is given by
  \[\rho_{i_3}^*u_2^*\rho_{i_2*}\rho_{i_2}^*u_1^*\rho_{i_1*}\stackrel{ad^{-1}}\longrightarrow \rho_{i_3}^*u_2^*u_1^*\rho_{i_1*}\stackrel{\sim}\longrightarrow \rho_{i_3}^*(u_1u_2)^*\rho_{i_1*}.\]
  Here, the first arrow is defined and an isomorphism by (\ref{4.8}). Now, we will verify the axioms of psedofunctors for $G$ as follows.
  \begin{enumerate}[(1)]
    \item Consider morphisms
    \[\widetilde{X_2}\stackrel{f_1}\longrightarrow \widetilde{X_1}\stackrel{{\rm id}_{\widetilde{X_1}}}\longrightarrow \widetilde{X_1}\]
    in $\mathscr{C}$ given by a commutative diagram
      \[\begin{tikzcd}
    (Z_2\arrow[d,"v_1"]\arrow[r,"i_2"]&S_2\arrow[d,"u_1"])\\
    (Z_1\arrow[r,"i_1"]\arrow[d,"{\rm id}"]&S_1\arrow[d,"{\rm id}"])\\
    (Z_1\arrow[r,"i_1"]&S_1)
    \end{tikzcd}\]
    in $\mathcal{C}$. We have to show that the diagram
    \[\begin{tikzcd}
      \widetilde{f_1}^*{\rm id}_{\widetilde{X_1}}^*\arrow[r,"G_{{\rm id}_{\widetilde{X_1}}}"]\arrow[d,"G_{{\rm id}_{\widetilde{X_1}},f_1}"]&\widetilde{f_1}^*\arrow[d,equal]\\
      ({\rm id}_{\widetilde{X_1}}\widetilde{f_1})^*\arrow[r,equal]&\widetilde{f_1}^*
    \end{tikzcd}\]
    of functors commutes. It is the diagram
    \[\begin{tikzcd}
      \rho_{i_2}^*u_1^*\rho_{i_1*}\rho_{i_1}^*\rho_{i_1*}\arrow[r,"ad'"]\arrow[d,"ad^{-1}"]&\rho_{i_2}^*u_1^*\rho_{i_1*}\arrow[d,equal]\\
      \rho_{i_2}^*u_1^*\rho_{i_1*}\arrow[r,equal]&\rho_{i_2}^*u_1^*\rho_{i_1*}
    \end{tikzcd}\]
    of functors, which commutes.
    \item Consider morphisms $\widetilde{X_2}\stackrel{{\rm id}_{\widetilde{X_2}}}\longrightarrow \widetilde{X_2}\stackrel{\widetilde{f_1}}\longrightarrow \widetilde{X_1}$ in $\widetilde{C}$. As in the above argument, the diagram
    \[\begin{tikzcd}
      {\rm id}_{\widetilde{X_2}}^*f_1^*\arrow[d,"G_{f_1,{\rm id}_{\widetilde{X_2}}}"]\arrow[r,"G_{{\rm id}_{\widetilde{X_2}}}"]&\widetilde{f_1}^*\arrow[d,equal]\\
      (f_1{\rm id}_{\widetilde{X_2}})^*\arrow[r,equal]&f_1^*
    \end{tikzcd}\]
    of functors commutes.
    \item Consider morphisms
    \[\widetilde{X_4}\stackrel{\widetilde{f_3}}\rightarrow \widetilde{X_3}\stackrel{\widetilde{f_2}}\rightarrow \widetilde{X_2} \stackrel{\widetilde{f_1}}\rightarrow \widetilde{X_1}\]
    in $\mathcal{C}$ given by a commutative diagram
    \[\begin{tikzcd}
    (Z_4\arrow[d,"v_3"]\arrow[r,"i_4"]&S_4\arrow[d,"u_3"])\\
    (Z_3\arrow[d,"v_2"]\arrow[r,"i_3"]&S_3\arrow[d,"u_2"])\\
    (Z_2\arrow[r,"i_2"]\arrow[d,"v_1"]&S_2\arrow[d,"u_1"])\\
    (Z_1\arrow[r,"i_1"]&S_1)
    \end{tikzcd}\]
    in $\mathcal{C}$. We have to show that the diagram
    \begin{equation}\label{4.12.1}\begin{tikzcd}
      \widetilde{f_3}^*\widetilde{f_2}^*\widetilde{f_1}^*\arrow[r,"G_{\widetilde{f_1},\widetilde{f_2}}"]\arrow[d,"G_{\widetilde{f_2},\widetilde{f_3}}"]&\widetilde{f_3}^* (\widetilde{f_1}\widetilde{f_2})^*\arrow[d,"G_{\widetilde{f_1}\widetilde{f_2},\widetilde{f_3}}"]\\
      (\widetilde{f_2}\widetilde{f_3})^*\widetilde{f_1}^*\arrow[r,"G_{\widetilde{f_1},\widetilde{f_2}\widetilde{f_3}}"]&(\widetilde{f_1}\widetilde{f_2}\widetilde{f_3})^*
    \end{tikzcd}\end{equation}
    of functors commutes. It is the big outside diagram of the diagram
    \[\begin{tikzcd}
      \rho_{i_4}^*u_3^*\rho_{i_3*}\rho_{i_3}^*u_2^*\rho_{i_2*}\rho_{i_2}^*u_1^*\rho_{i_1*}\arrow[d,"ad^{-1}"]\arrow[r,"ad^{-1}"]&
      \rho_{i_4}^*u_3^*\rho_{i_3*}\rho_{i_3}^*u_2^*u_1^*\rho_{i_1*}\arrow[r,"\sim"]\arrow[d,"ad^{-1}"]&
      \rho_{i_4}^*u_3^*\rho_{i_3*}\rho_{i_3}^*(u_1u_2)^*\rho_{i_1*}\arrow[d,"ad^{-1}"]\\
      \rho_{i_4}^*u_3^*u_2^*\rho_{i_2*}\rho_{i_2}^*u_1^*\rho_{i_1*}\arrow[d,"\sim"]\arrow[r,"ad^{-1}"]&
      \rho_{i_4}^*u_3^*u_2^*u_1^*\rho_{i_1*}\arrow[d,"\sim"]\arrow[r,"\sim"]&
      \rho_{i_4}^*u_3^*(u_1u_2)^*\rho_{i_1*}\arrow[d,"\sim"]\\
      \rho_{i_4}^*(u_2u_3)^*\rho_{i_2*}\rho_{i_2}^*u_1^*\rho_{i_1*}\arrow[r,"ad^{-1}"]&
      \rho_{i_4}^*(u_2u_3)^*u_1^*\rho_{i_1*}\arrow[r,"\sim"]&
      \rho_{i_4}^*(u_1u_2u_3)^*\rho_{i_1*}
    \end{tikzcd}\]
    of functors, whose small diagrams commute. Here, the arrows denoted by $ad^{-1}$ are defined and isomorphisms by (\ref{4.8}). Thus (\ref{4.12.1}) commutes.
  \end{enumerate}
  Thus we have verified all axioms of pseudofunctors for $G$, so $G$ is a pseudofunctor.
\end{none}
\begin{none}\label{4.10}
  Let us interpret (\ref{4.8}(2)) as follows. Consider a morphism $\widetilde{f_1}:\widetilde{X_2}\rightarrow \widetilde{X_1}$ given by a Cartesian diagram
  \[\begin{tikzcd}
    (Z_2\arrow[r,"i_2"]\arrow[d,"v_1"]&S_2\arrow[d,"u_1"])\\
    (Z_1\arrow[r,"i_1"]&S_1)
  \end{tikzcd}\]
  in $\mathcal{C}$. Then we have the Cartesian diagram
  \[\begin{tikzcd}
    \widetilde{X_2}\arrow[r,"\rho_{i_2}"]\arrow[d,"\widetilde{f_1}"]&S_2\arrow[d,"u_1"]\\
    \widetilde{X_1}\arrow[r,"\rho_{i_1}"]&S_1
  \end{tikzcd}\]
  in $\widetilde{\mathcal{C}}$. Using $G$ constructed in (\ref{4.12}), we have the exchange transformation
  \[u_1^*\rho_{i_1*}\stackrel{Ex}\longrightarrow \rho_{i_2*}\widetilde{f_1}^*\]
  is given by
  \[u_1^*\rho_{i_1*}\stackrel{ad}\longrightarrow \rho_{i_2*}\rho_{i_2}^*u_1^*\rho_{i_1*}\stackrel{G_{u_1,\rho_{i_2}}}\longrightarrow \rho_{i_2*}\widetilde{f_1}^*\rho_{i_1}^*\rho_{i_1*} \stackrel{ad'}\longrightarrow \rho_{i_2}^*\widetilde{f_1}^*.\]
  The first arrow is an isomorphism by (\ref{4.8}(2)), and the second arrow is an isomorphism since $G$ is a pseudofunctor. The third arrow is also an isomorphism by (\ref{4.5}(3)). Thus the exchange transformation is an isomorphism.
\end{none}
\begin{none}\label{4.9}
  Let us continue the argument in (\ref{4.12}). So far, we have constructed the contravariant pseudofunctor $G$. Now, we will show that $G$ is $\widetilde{\mathscr{P}}$-premotivic by verifying the axioms from (B--1) to (B--4) as follows.
  \begin{enumerate}[(1)]
    \item For any morphism $\widetilde{f_1}$ given by a commutative diagram
      \[\begin{tikzcd}
        (Z_2\arrow[r,"i_2"]\arrow[d,"v_1"]&S_2\arrow[d,"u_1"])\\
        (Z_1\arrow[r,"i_1"]&S_1)
      \end{tikzcd}\]
      in $\mathcal{C}$, the right adjoint of $\widetilde{f_1}^*$ is $\rho_{i_1}^!u_{1*}\rho_{i_2*}$. Thus (B--1) holds for $G$.
    \item For any $\widetilde{\mathscr{P}}$-morphism $\widetilde{f_1}$ in $\widetilde{C}$ given by a Cartesian diagram
    \[\begin{tikzcd}
        (Z_2\arrow[r,"i_2"]\arrow[d,"v_1"]&S_2\arrow[d,"u_1"])\\
        (Z_1\arrow[r,"i_1"]&S_1)
      \end{tikzcd}\]
    in $\mathcal{C}$, we will show that the left adjoint of $\widetilde{f_1}^*$ is $\rho_{i_1}^*u_{1\sharp}\rho_{i_2*}$. For this, we will show that $\widetilde{f_1}^*$ is isomorphic to
    \[\rho_{i_2}^!u_1^*\rho_{i_1*}.\]
    Consider the Cartesian diagram
    \[\begin{tikzcd}
      (Z_2\arrow[r,"i_2"]\arrow[d,"v_1"]&S_2\arrow[d,"u_1"])\\
      (Z_1\arrow[r,"i_1"]&S_1)
    \end{tikzcd}\]
    where $j_1$ (resp.\ $j_2$) denotes the complement of $i_1$ (resp.\ $i_2$). Since the natural transformations
    \[\rho_{i_2}^*\rho_{i_2*}\stackrel{ad'}\longrightarrow {\rm id},\quad {\rm id}\stackrel{ad}\longrightarrow \rho_{i_2}^!\rho_{i_2*}\]
    are isomorphisms by (\ref{4.5}(3)), we only need to show that the essential image of $u_1^*\rho_{i_1*}$ is in $G(Z_2\rightarrow S_2)$, which is true by (\ref{4.10}). Thus (B--2) holds for $G$.
    \item Consider a Cartesian diagram
    \[\begin{tikzcd}
      \widetilde{X_4}\arrow[d,"\widetilde{f_3}"]\arrow[r,"\widetilde{f_2}"]&\widetilde{X_2}\arrow[d,"\widetilde{f_1}"]\\
      \widetilde{X_3}\arrow[r,"\widetilde{f_2}"]&\widetilde{X_1}
    \end{tikzcd}\]
    in $\widetilde{C}$ given by a commutative diagram
    \[\begin{tikzcd}
      (Z_4\arrow[dd,"v_3"]\arrow[rr,"i_4"]\arrow[rd,"v_4"]&&S_4)\arrow[dd,"u_3",near start]\arrow[rd,"u_4"]\\
      &(Z_2\arrow[rr,"i_2",crossing over,near start]&&S_2)\arrow[dd,"u_1"]\\
      (Z_3\arrow[rr,"i_3",near start]\arrow[rd,"v_2"]&&S_3)\arrow[rd,"u_2"]\\
      &(Z_1\arrow[rr,"i_1"]\arrow[uu,leftarrow,crossing over,"v_1",near end]&&S_1)
    \end{tikzcd}\]
    in $\mathcal{C}$. Assume that $\widetilde{f_1}$ is a $\widetilde{\mathscr{P}}$-morphism. Consider the natural transformation
    \[\widetilde{f_1}^*\widetilde{f_2}_*\stackrel{Ex}\longrightarrow \widetilde{f_4}_*\widetilde{f_3}^*\]
    that is the left adjoint of the exchange transformation in (B--3). We will show that this is an isomorphism. By (\ref{4.7}(3)), $\rho_{i_2*}$ is fully faithful, so it suffices to show that the natural transformation
    \[\rho_{i_2*}\widetilde{f_1}^*\widetilde{f_2}_*\stackrel{Ex}\longrightarrow \rho_{i_2*}\widetilde{f_4}_*\widetilde{f_3}^*\]
    is an isomorphism.

    From the commutative diagram
    \[\begin{tikzcd}
      \widetilde{X_4}\arrow[dd,"\widetilde{f_3}"]\arrow[rr,"\rho_{i_4}"]\arrow[rd,"\widetilde{f_4}"]&&S_4\arrow[dd,"u_3",near start]\arrow[rd,"u_4"]\\
      &\widetilde{X_2}\arrow[rr,"\rho_{i_2}",crossing over,near start]&&S_2\arrow[dd,"u_1"]\\
      \widetilde{X_3}\arrow[rr,"\rho_{i_3}",near start]\arrow[rd,"\widetilde{f_2}"]&&S_3\arrow[rd,"u_2"]\\
      &\widetilde{X_1}\arrow[rr,"\rho_{i_1}"]\arrow[uu,leftarrow,crossing over,"\widetilde{f_1}",near end]&&S_1
    \end{tikzcd}\]
    in $\widetilde{\mathcal{C}}$, we have the commutative diagram
    \[\begin{tikzcd}
      u_1^*u_{2*}\rho_{i_3*}\arrow[r,"\sim"]\arrow[d,"Ex"]&u_1^*\rho_{i_1*}\widetilde{f_2}_*\arrow[r,"Ex"]&\rho_{i_2*}\widetilde{f_1}^*\widetilde{f_2}_*\arrow[d,"Ex"]\\
      u_{4*}u_3^*\rho_{i_3*}\arrow[r,"Ex"]&u_{4*}\rho_{i_4*}\widetilde{f_3}^*\arrow[r,"\sim"]&\rho_{i_2*}\widetilde{f_4}_*\widetilde{f_3}^*
    \end{tikzcd}\]
    of functors. The left vertical arrow is an isomorphism by (B--3) for $V$. Hence to show that the right vertical arrow is an isomorphism, it suffices to show that the upper right horizontal and lower left horizontal arrows are isomorphisms. This follows from (\ref{4.10}). Thus (B--3) holds for $G$.
    \item Let $\widetilde{f_1}:\widetilde{X_2}\rightarrow \widetilde{X_1}$ be a $\widetilde{\mathscr{P}}$-morphism in $\widetilde{C}$ given by a Cartesian diagram
    \[\begin{tikzcd}
        (Z_2\arrow[r,"i_2"]\arrow[d,"v_1"]&S_2\arrow[d,"u_1"])\\
        (Z_1\arrow[r,"i_1"]&S_1)
      \end{tikzcd}\]
    in $\mathcal{C}$. For any object $K$ of $G(\widetilde{X_2})$ and $L$ of $G(\widetilde{X_1})$, we will show that the exchange transformation
    \[\widetilde{f_1}_\sharp(K\otimes_{\widetilde{X_2}}\widetilde{f_1}^*L)\stackrel{Ex}\longrightarrow \widetilde{f_1}_\sharp K\otimes_{\widetilde{X_1}}L\]
    is an isomorphism. Since $\rho_{i_1}^*$ and $\rho_{i_2}^*$ are essentially surjective by (\ref{4.5}(3)), we can put
    \[K=\rho_{i_2}^*K',\quad L=\rho_{i_1}^*L'.\]
    Then consider the commutative diagram
    \[\begin{tikzcd}
      \widetilde{f_1}_\sharp(\rho_{i_2}^*K'\otimes_{\widetilde{X_2}}\widetilde{f_1}^*\rho_{i_1}^*L')\arrow[d,"\sim"]\arrow[r,"Ex"]&\widetilde{f_1}_\sharp\rho_{i_2}^*K' \otimes_{\widetilde{X_1}}\rho_{i_1}^*L'\arrow[d,"Ex"]\\
      \widetilde{f_1}_\sharp(\rho_{i_2}^*K'\otimes_{\widetilde{X_2}}\rho_{i_2}^*u_1^*L')\arrow[d,"\sim"]&\rho_{i_1}^*u_{1\sharp}K'\otimes_{\widetilde{X_1}}\rho_{i_1}^*L'\arrow[dd,"\sim"]\\
      \widetilde{f_1}_\sharp \rho_{i_2}^*(K'\otimes_{S_2}u_1^*L')\arrow[d,"Ex"]\\
      \rho_{i_1}^*u_{1\sharp}(K'\otimes_{S_2}u_1^*L')\arrow[r,"Ex"]&\rho_{i_1}^*(u_{1\sharp}K'\otimes_{S_1}L')
    \end{tikzcd}\]
    in $G(\widetilde{X_1})$. Here, the middle left vertical and the lower right vertical arrows are obtained by the fact that $\rho_{i_1}^*$ and $\rho_{i_2}^*$ are monoidal, which is proved in (\ref{4.6}). Since the bottom left vertical arrow and the upper right vertical arrows are isomorphisms by (B--3), to show that the upper horizontal arrow is an isomorphism, it suffices to show that the lower horizontal arrow is an isomorphism. This follows from (B--4) for $G$.
  \end{enumerate}
  Thus $G$ is a $\widetilde{\mathscr{P}}$-premotivic pseudofunctor.
\end{none}
\begin{none}\label{4.14}
  Let us continue the argument in (\ref{4.9}). We will show that $G$ satisfies (Loc). Recall from (\ref{2.5}) that a closed immersion in $\widetilde{\mathcal{C}}$ is of the form $\rho_{i_1}$ for some closed immersion $i_1:Z_1\rightarrow S_1$ in $\mathcal{C}$. The complement of $\rho_{i_1}$ in $\widetilde{\mathcal{C}}$ is $j_1$ where $j_1:U_1\rightarrow S_1$ denotes the complement of $i_1$ in $\mathcal{C}$.

  If $a:K\rightarrow K'$ is a morphism in $G(S_1)$ such that $\rho_{i_1}^*a$ and $j_1^*a$ are isomorphisms, then consider the commutative diagram
  \[\begin{tikzcd}
    j_{1\sharp}j_1^*K\arrow[d,"j_{1\sharp}j_1^*a"]\arrow[r,"ad'"]&K\arrow[r,"ad'"]\arrow[d,"a"]&\rho_{i_1*}\rho_{i_1}^*K\arrow[r]\arrow[d,"\rho_{i_1*}\rho_{i_1}^*a"] &j_{1\sharp}j_1^*K[1]\arrow[d,"j_{1\sharp}j_1^*a"]\\
    j_{1\sharp}j_1^*K'\arrow[r,"ad'"]&K'\arrow[r,"ad'"]&\rho_{i_1*}\rho_{i_1}^*K'\arrow[r]&j_{1\sharp}j_1^*K'[1]
  \end{tikzcd}\]
  in $G(S_1)$ where the two rows are distinguished triangles obtained by (\ref{4.5}(2)). By assumption, the first and third vertical arrows are isomorphisms, so the second vertical arrow is an isomorphism. Thus the pair of functors $(\rho_{i_1}^*,j_1^*)$ is conservative. By (\ref{4.5}(3)), the counit $\rho_{i_1}^*\rho_{i_1*}\stackrel{ad}\longrightarrow {\rm id}$ is an isomorphism. Thus $G$ satisfies $({\rm Loc}_{\rho_{i_1}})$.

  Consider the Cartesian diagram
  \[\begin{tikzcd}
    (\emptyset\rightarrow U_1)\arrow[d]\arrow[r]&U_1\arrow[d,"j_1"]\\
    (Z_1\stackrel{i_1}\rightarrow S_1)\arrow[r,"\rho_{i_1}"]&S_1
  \end{tikzcd}\]
  in $\widetilde{\mathcal{C}}$. By the construction of $G$, we have $G(\emptyset\rightarrow U_1)=0$. Thus we have verified all axioms of (Loc) for $G$, so $G$ satisfies (Loc).
\end{none}
\begin{none}\label{4.15}
  Let us continue the argument in (\ref{4.14}). We will construct a pseudonatural equivalence $\delta:V\rightarrow G\circ U$ in (\ref{4.1.1}). Recall from (\ref{4.3}) that $G(U(S_1))=V(S_1)$, and for any object $S_1$ of $\mathcal{E}$, put $\delta(S_1)$ as the identity functor
  \[V(S_1)\longrightarrow G(U(S_1)).\]
  For any morphism $f_1:S_2\rightarrow S_1$ in $\mathcal{E}$, put $\delta(S_1)$ as the natural transformation
  \[f_1^*\stackrel{\sim}\longrightarrow {\rm id}_{S_2}^*f_1^*{\rm id}_{S_1*}.\]
  We will verify the axioms of pseudonatural transformations for $\delta$ as follows.
  \begin{enumerate}[(1)]
    \item For any object $S_1$ of $\mathcal{E}$, we have to show that the diagram
    \[\begin{tikzcd}
      \delta(S_1)V({\rm id}_{S_1})\arrow[rd,"V_{{\rm id}_{S_1}}"']\arrow[rr,"\delta_{{\rm id}_{S_1}}"]&&G(U(S_1))\delta(S_1)\arrow[ld,"G_{U(S_1)}"]\\
      &\delta(S_1)
    \end{tikzcd}\]
    of functors commutes. It is the diagram
    \[\begin{tikzcd}
      {\rm id}_{S_1}^*\arrow[rdd,"{\rm id}"]\arrow[rr,"\sim"]&&{\rm id}_{S_1}^*{\rm id}_{S_1}^*{\rm id}_{S_1*}\arrow[d,"\sim"]\\
      &&{\rm id}_{S_1}^*{\rm id}_{S_1*}\arrow[ld,"ad'"]\\
      &{\rm id}
    \end{tikzcd}\]
    of functors, which commutes.
    \item For any morphisms $S_3\stackrel{f_2}\rightarrow S_2\stackrel{f_1}\rightarrow S_1$ in $\mathcal{E}$, we have to show that the diagram
    \[\begin{tikzcd}
      \delta(S_3)V(f_2)V(f_1)\arrow[d,"V_{f_1,f_2}"]\arrow[r,"\delta_{f_2}"]&G(U(f_2))\delta(S_2)V(f_1)\arrow[r,"\delta_{f_1}"]&G(U(f_2))G(U(f_1))\delta(S_1)\arrow[d, "G_{U(f_1),U(f_2)}"]\\
      \delta(S_3)V(f_1f_2)\arrow[rr,"\delta_{f_1f_2}"]&&G(U(f_2)U(f_1))\delta(S_1)
    \end{tikzcd}\]
    of functors commutes. It is the diagram
    \[\begin{tikzcd}
      f_2^*f_1^*\arrow[dd,"\sim"]\arrow[r,"\sim"]&{\rm id}_{S_2}^*f_2^*{\rm id}_{S_2*}f_1^*\arrow[r,"\sim"]&{\rm id}_{S_3}^*f_2^*{\rm id}_{S_2*}{\rm id}_{S_2}^*f_1^*{\rm id}_{S_1*} \arrow[d,"ad^{-1}"]\\
      &&{\rm id}_{S_3}^*f_2^*f_1^*{\rm id}_{S_1*}\arrow[d,"\sim"]\\
      {\rm id}_{S_3}^*f_2^*f_1^*{\rm id}_{S_1*}\arrow[rr,"\sim"]&&{\rm id}_{S_3}^*(f_1f_2)^*{\rm id}_{S_1*}
    \end{tikzcd}\]
    of functors, which commutes.
  \end{enumerate}
  We have verified all axioms of pseudonatural transformations for $\delta$, and $\delta(S_1)$ is an equivalence for any object $S_1$ of $\mathcal{C}$, so $\delta$ is a pseudonatural equivalence.

  So far, we have proved the following.
\end{none}
\begin{thm}\label{4.16}
  Under the notations and hypotheses of (\ref{4.1}), there is a $\widetilde{\mathscr{P}}$-premotivic pseudofunctor $G:\mathcal{E}\rightarrow {\rm Tri}^{\otimes}$ satisfying {\rm (Loc)} and a pseudonatural equivalence $\delta:V\rightarrow G\circ U$ such that the diagram
  \[\begin{tikzcd}
    \mathcal{E}\arrow[rd,"V"']\arrow[r,"U"]&\widetilde{\mathcal{C}}\arrow[d,"G"]\arrow[ld,phantom,"{\rotatebox[origin=c]{0}{$\Leftrightarrow$}}_{\delta}",very near start]\\
    \,&{\rm Tri}^\otimes
  \end{tikzcd}\]
  of $2$-categories commutes.
\end{thm}
\begin{none}
  Now, we will prove that the above construction is functorial in the following sense.
\end{none}
\begin{thm}\label{4.13}
  Under the notations and hypotheses of {\rm (\ref{4.1})}, consider a diagram
  \[\begin{tikzcd}
    \mathcal{E}\arrow[ddrr,"V'"]\arrow[rr,"U"]\arrow[ddrr,bend right,"V"']&&\widetilde{\mathcal{C}}\arrow[dd,"G"']\arrow[dd,bend left,"G'"]\\
    \\
    &&{\rm Tri}^\otimes
  \end{tikzcd}\]
  of $2$-categories where $V$ and $V'$ are $\mathscr{P}'$-premotivic pseudofunctors satisfying {\rm (Loc)} and $G$ and $G'$ are $\widetilde{\mathscr{P}}$-premotivic pseudofunctors satisfying {\rm (Loc)}. Consider also a diagram
  \[\begin{tikzcd}
    V\arrow[r,"\delta"]\arrow[d,"\epsilon"]&G\circ U\\
    V'\arrow[r,"\delta'"]&G'\circ U
  \end{tikzcd}\]
  where $\epsilon$ is a $\mathscr{P}'$-premotivic pseudonatural transformations and $\delta$ and $\delta'$ are pseudonatural equivalences. Then there is a $\widetilde{\mathscr{P}}$-premotivic pseudonatural transformation
  \[\beta:G\rightarrow G'\]
  unique up to isomorphisms such that the induced $\mathscr{P}'$-premotivic pseudonatural transformation $G\circ U\rightarrow G'\circ U$ makes the above diagram commutative.
\end{thm}
\begin{proof}
  We may assume that $\delta$ and $\delta'$ are the identities. We will first construct $\beta:G\rightarrow G'$. For any object $\widetilde{X_1}=(Z_1\stackrel{i_1}\rightarrow S_1)$ of $\widetilde{\mathcal{C}}$, put $\beta(\widetilde{X_1})$ as the composition
  \[G(\widetilde{X_1})\stackrel{\rho_{i_1*}}\longrightarrow V(S_1)\stackrel{\epsilon(S_1)}\longrightarrow V'(S_1)\stackrel{\rho_{i_1}^*}\longrightarrow G'(\widetilde{X_1})\]
  of functors. For any morphism $\widetilde{f_1}:\widetilde{X_2}\rightarrow \widetilde{X_1}$ in $\mathcal{C}$ given by a commutative diagram
  \[\begin{tikzcd}
        (Z_2\arrow[r,"i_2"]\arrow[d,"v_1"]&S_2\arrow[d,"u_1"])\\
        (Z_1\arrow[r,"i_1"]&S_1)
  \end{tikzcd}\]
  in $\mathcal{C}$, put $\beta(\widetilde{f_1})$ as the natural transformation
  \[\beta(\widetilde{X_2})G(\widetilde{f_1})\longrightarrow G'(\widetilde{f_1})\beta(\widetilde{X_1})\]
  given by the composition
  \[\rho_{i_2}^*\epsilon(S_2)\rho_{i_2*}\widetilde{f_1}^*\stackrel{Ex^{-1}}\longrightarrow \rho_{i_2}^*\epsilon(S_2)u_1^*\rho_{i_1*}\stackrel{\epsilon_{u_1}}\longrightarrow \rho_{i_2}^*u_1^*\epsilon(S_1)\rho_{i_1*}\stackrel{\sim}\longrightarrow \widetilde{f_1}^*\rho_{i_1}^*\epsilon(S_1)\rho_{i_1*}.\]
  Here, the first arrow is defined and an isomorphism by (\ref{2.8}).

  Now, we will verify the axioms of pseudonatural transformations for $\beta$.
  \begin{enumerate}[(1)]
    \item For any object $\widetilde{X_1}=(Z_1\stackrel{i_1}\rightarrow S_1)$ in $\widetilde{\mathcal{C}}$, we have to show that the diagram
    \[\begin{tikzcd}
      \beta(\widetilde{X_1})G({\rm id}_{\widetilde{X_1}})\arrow[rr,"\beta({\rm id}_{\widetilde{X_1}})"]\arrow[rd,"G_{{\rm id}_{\widetilde{X_1}}}"']&&G'({\rm id}_{\widetilde{X_1}})\beta(\widetilde{X_1})\arrow[ld,"G_{{\rm id}_{\widetilde{X_1}}}'"]\\
      &\beta(\widetilde{X_1})
    \end{tikzcd}\]
    of functors commutes. It is true since it is the big outside diagram of the diagram
    \[\begin{tikzcd}
      \rho_{i_1}^*\epsilon(S_1)\rho_{i_1*}{\rm id}_{\widetilde{X_1}}^*\arrow[r,"Ex^{-1}"]\arrow[rrd,"G_{{\rm id}_{\widetilde{X_1}}}"']&\rho_{i_1}^*\epsilon(S_1){\rm id}_{S_1}^*\rho_{i_1*}\arrow[rd,"V_{{\rm id}_{S_1}}"]\arrow[r,"\epsilon_{{\rm id}_{S_1}}"]&\rho_{i_1}^*{\rm id}_{S_1}^*\epsilon(S_1)\rho_{i_1*}\arrow[d,"V_{{\rm id}_{S_1}}'"] \arrow[r,"\sim"]& {\rm id}_{\widetilde{X_1}}^*\rho_{i_1}^*\epsilon(S_1)\rho_{i_1*}\arrow[ld,"G_{{\rm id}_{\widetilde{X_1}}}'"]\\
      &&\rho_{i_1}^*\epsilon(S_1)\rho_{i_1*}
    \end{tikzcd}\]
    of functors, whose small diagrams commute.
    \item For any morphisms $\widetilde{X_3}\stackrel{\widetilde{f_2}}\longrightarrow \widetilde{X_2}\stackrel{\widetilde{f_1}}\longrightarrow \widetilde{X_1}$ in $\widetilde{\mathcal{C}}$ given by a commutative diagram
    \[\begin{tikzcd}
      (Z_3\arrow[d,"v_2"]\arrow[r,"i_3"]&S_3\arrow[d,"u_2"])\\
      (Z_2\arrow[r,"i_2"]\arrow[d,"v_1"]&S_2\arrow[d,"u_1"])\\
      (Z_1\arrow[r,"i_1"]&S_1)
    \end{tikzcd}\]
    in $\mathcal{C}$, we have to show that the diagram
    \[\begin{tikzcd}
      \beta(\widetilde{X_3})G(\widetilde{f_2})G(\widetilde{f_1})\arrow[d,"\beta_{\widetilde{f_2}}"]\arrow[r,"G_{\widetilde{f_1},\widetilde{f_2}}"]& \beta(\widetilde{X_3})G(\widetilde{f_1}\widetilde{f_2})\arrow[dd,"\beta_{\widetilde{f_1},\widetilde{f_2}}"]\\
      G'(\widetilde{f_2})\beta(\widetilde{X_2})G(\widetilde{f_1})\arrow[d,"\beta_{\widetilde{f_1}}"]\\
      G'(\widetilde{f_2})G'(\widetilde{f_1})\beta(\widetilde{X_1})\arrow[r,"G_{\widetilde{f_1},\widetilde{f_2}}'"]&G'(\widetilde{f_1}\widetilde{f_2})\beta(\widetilde{X_1})
    \end{tikzcd}\]
    of functors commutes. It is true since it is the big outside diagram of the diagram
    \[\begin{tikzcd}
      \rho_{i_3}^*\epsilon(S_3)\rho_{i_3*}\widetilde{f_2}^*\widetilde{f_1}^*\arrow[rr,"\sim"]\arrow[d,"Ex^{-1}"]&&\rho_{i_3}^*\epsilon(S_3)\rho_{i_3*}(\widetilde{f_1}\widetilde{f_2})^* \arrow[d,"Ex^{-1}"]\\
      \rho_{i_3}^*\epsilon(S_3)u_2^*\rho_{i_2*}\widetilde{f_1}^*\arrow[d,"\epsilon_{u_2}"]\arrow[r,"Ex^{-1}"]&\rho_{i_3}^*\epsilon(S_3)u_2^*u_1^*\rho_{i_1*}\arrow[d,"\epsilon_{u_2}"] \arrow[r,"\sim"]&\rho_{i_3}^*\epsilon(S_3)(u_1u_2)^*\rho_{i_1*}\arrow[dddd,"\epsilon_{u_1u_2}"]\\
      \rho_{i_3}^*u_2^*\epsilon(S_2)\rho_{i_2*}\widetilde{f_1}^*\arrow[d,"\sim"]\arrow[r,"Ex^{-1}"]&\rho_{i_3}^*u_2^*\epsilon(S_2)u_1^*\rho_{i_1*}\arrow[ddd,"\epsilon_{u_1}"]\\
      \widetilde{f_2}^*\rho_{i_2}^*\epsilon(S_2)\rho_{i_2*}\widetilde{f_1}^*\arrow[d,"Ex^{-1}"]\\
      \widetilde{f_2}^*\rho_{i_2}^*\epsilon(S_2)u_1^*\rho_{i_1*}\arrow[d,"\epsilon_{u_1}"]\\
      \widetilde{f_2}^*\rho_{i_2}^*u_1^*\epsilon(S_1)\rho_{i_1*}\arrow[d,"\sim"]\arrow[r,"\sim"]&\rho_{i_3}^*u_2^*u_1^*\epsilon(S_1)\rho_{i_1*}\arrow[r,"\sim"] & \rho_{i_3}^*(u_1u_2)^*\epsilon(S_1)\rho_{i_1*}\arrow[d,"\sim"]\\
      \widetilde{f_2}^*\widetilde{f_1}^*\rho_{i_1}^*\epsilon(S_1)\rho_{i_1*}\arrow[rr,"\sim"]&&(\widetilde{f_1}\widetilde{f_2})^*\rho_{i_1}^*\epsilon(S_1)\rho_{i_1*}
    \end{tikzcd}\]
    of functors, whose small diagrams commute.
  \end{enumerate}
  We have verified all axioms of pseudonatural transformations for $\beta$, so $\beta$ is a pseudonatural transformation. The commutativity of the diagram
  \[\begin{tikzcd}
    V\arrow[d,"\epsilon"]\arrow[r,"\delta"]&G\circ U\arrow[d,"\beta"]\\
    V'\arrow[r,"\delta'"]&G'\circ U
  \end{tikzcd}\]
  of pseudofunctors follows from construction. Now, we will verify the axioms (C--1) and (C--2) for $\beta$. Let $\widetilde{X_1}=(Z_1\stackrel{i_1}\rightarrow S_1)$ be an object of $\widetilde{\mathcal{C}}$. Since the functor $\epsilon(S_1)$ has a right adjoint by (C--1) for $\epsilon$, the functor
  \[\beta(\widetilde{X_1})=\rho_{i_1}^*\epsilon(S_1)\rho_{i_1*}\]
  has a right adjoint. Thus $\beta$ satisfies (C--1). Let $f_1:\widetilde{X_2}\rightarrow \widetilde{X_1}$ be a morphism in $\widetilde{\mathscr{P}}$ given by a Cartesian diagram
  \[\begin{tikzcd}
    Z_2\arrow[d,"v_1"]\arrow[r,"i_2"]&S_2\arrow[d,"u_1"]\\
    Z_1\arrow[r,"i_1"]&S_1
  \end{tikzcd}\]
  in $\mathcal{C}$. Then we have the Cartesian diagram
  \[\begin{tikzcd}
    \widetilde{X_2}\arrow[r,"\rho_{i_2}"]\arrow[d,"f_1"]&S_2\arrow[d,"u_1"]\\
    \widetilde{X_1}\arrow[r,"\rho_{i_1}"]&S_1
  \end{tikzcd}\]
  in $\widetilde{\mathcal{C}}$, and we have the commutative diagram
  \[\begin{tikzpicture}[baseline= (a).base]
    \node[scale=.92] (a) at (0,0)
    {\begin{tikzcd}
    f_{1\sharp}\rho_{i_2}^*\epsilon(S_2)\rho_{i_2*}\arrow[d,"Ex"]\arrow[r,"ad"]&f_{1\sharp}\rho_{i_2}^*\epsilon(S_2)\rho_{i_2*}f_1^*f_{1\sharp}\arrow[d,"Ex"] \arrow[rr,"Ex^{-1}"]&& f_{1\sharp}\rho_{i_2}^*\epsilon(S_2)u_1^*\rho_{i_2*}f_{1\sharp}\arrow[d,"\epsilon_{u_1}"]\arrow[lldd,"Ex"]\\
    \rho_{i_1}^*u_{1\sharp}\epsilon(S_2)\rho_{i_2*}\arrow[d,"ad"]\arrow[r,"ad"]& \rho_{i_1}^*u_{1\sharp}\epsilon(S_2)\rho_{i_2*}f_1^*f_{1\sharp}\arrow[d,"Ex^{-1}"]&&f_{1\sharp}\rho_{i_2}^* u_1^*\epsilon(S_1)\rho_{i_2*}f_{1\sharp}\arrow[d,"\sim"]\arrow[ld,"Ex"]\\
    \rho_{i_1}^*u_{1\sharp}\epsilon(S_2)u_1^*u_{1\sharp}\rho_{i_2*}\arrow[d,"\epsilon_{u_1}"]\arrow[r,"Ex_{\sharp *}"]&\rho_{i_1}^*u_{1\sharp}\epsilon(S_2)u_1^*f_{1\sharp}\rho_{i_1*} \arrow[r,"\epsilon_{u_1}"]&\rho_{i_1}^*u_{1\sharp}u_1^*\epsilon(S_1)f_{1\sharp}\rho_{i_1*}\arrow[rd,"ad'"]&f_{1\sharp}f_1^*\rho_{i_2}^*\epsilon(S_1)\rho_{i_2*}f_{1\sharp} \arrow[d,"ad'"]\\
    \rho_{i_1}^*u_{1\sharp}u_1^*\epsilon(S_1)u_{1\sharp}\rho_{i_2*}\arrow[r,"ad'"]&\rho_{i_1}^*\epsilon(S_1)u_{1\sharp}\rho_{i_2*}\arrow[rr,"Ex_{\sharp *}"]&&\rho_{i_2}^*\epsilon(S_1) \rho_{i_2*}f_{1\sharp}
  \end{tikzcd}};
  \end{tikzpicture}\]
  of functors. Here, the arrows denoted by $Ex_{\sharp*}$ are induced from (\ref{2.16}). The composition of the five arrows in the upper right route is the exchange transformation
  \begin{equation}\label{4.13.1}
    f_{1\sharp}\beta(S_2)\stackrel{Ex}\longrightarrow \beta(S_1)f_{1\sharp}.
  \end{equation}
  To show that this is an isomorphism, it suffices to show that the composition
  \[\begin{split}
    f_{1\sharp}\rho_{i_2}^*\epsilon(S_2)\rho_{i_2*}&\stackrel{Ex}\longrightarrow \rho_{i_1}^*u_{1\sharp}\epsilon(S_2)\rho_{i_2*}\stackrel{ad}\longrightarrow \rho_{i_1}^*u_{1\sharp}\epsilon(S_2)u_1^*u_{1\sharp}\rho_{i_2*}\stackrel{\epsilon_{u_1}}\longrightarrow \rho_{i_1}^*u_{1\sharp}u_1^*\epsilon(S_1)u_{1\sharp}\rho_{i_2*} \\
  &\stackrel{ad'}\longrightarrow \rho_{i_1}^*\epsilon(S_1)u_{1\sharp}\rho_{i_2*} \stackrel{Ex_{\sharp *}}\longrightarrow \rho_{i_2}^*\epsilon(S_1) \rho_{i_2*}f_{1\sharp}
  \end{split}\]
  of the natural transformations in the lower left route is an isomorphism. The composition of the second, third, and fourth arrows is an isomorphism by (C--2) for $\epsilon$, and the first arrow is an isomorphism by (B--3) for $G'$. Since the fifth arrow is an isomorphism by (\ref{2.16}), (\ref{4.13.1}) is an isomorphism. Thus $\beta$ satisfies the axiom (C--2). We have verified the axioms (C--1) and (C--2) for $\beta$, so $\beta$ is a $\widetilde{\mathscr{P}}$-premotivic pseudonatural transformation.

  Now, we will show the uniqueness of $\beta$. If $\beta':G\rightarrow G'$ be another $\widetilde{\mathscr{P}}$-premotivic pseudonatural transformation such that the diagram
  \[\begin{tikzcd}
    V\arrow[r,"\delta"]\arrow[d,"\epsilon"]&G\circ U\arrow[d,"\beta'"]\\
    V'\arrow[r,"\delta'"]&G'\circ U
  \end{tikzcd}\]
  commutes, we have to show that there is an isomorphism $\Phi:\beta\rightarrow \beta'$. We will construct it as follows. For any object $\widetilde{X_1}=(Z_1\stackrel{i_1}\rightarrow S_1)$ of $\widetilde{\mathcal{C}}$, put $\Phi_{\widetilde{X_1}}$ as the composition
  \[\begin{split}
    \beta(\widetilde{X_1})&\stackrel{ad'^{-1}}\longrightarrow \beta(\widetilde{X_1})\rho_{i_1}^*\rho_{i_1*}\stackrel{\beta_{\rho_{i_1}}}\longrightarrow \rho_{i_1}^*\beta(S_1)\rho_{i_1*} \stackrel{\sim}\longrightarrow \rho_{i_1}^*\epsilon(S_1)\rho_{i_1*}\\
    &\stackrel{\sim}\longrightarrow \rho_{i_1}^*\beta'(S_1)\rho_{i_1*}\stackrel{\beta_{\rho_{i_1}}^{-1}}\longrightarrow \beta'(\widetilde{X_1})\rho_{i_1}^*\rho_{i_1*}\stackrel{ad'}\longrightarrow \beta'(\widetilde{X_1}).
  \end{split}\]
  Here, the first arrow is defined and an isomorphism by (Loc), and the sixth arrow is also an isomorphism by (Loc). Thus $\Phi_{\widetilde{X_1}}$ is an isomorphism. The axiom of modifications for $\Phi$ is that for any morphism $\widetilde{f_1}:\widetilde{X_2}\rightarrow \widetilde{X_1}$ in $\widetilde{\mathcal{C}}$ given by a commutative diagram
  \[\begin{tikzcd}
    (Z_2\arrow[r,"i_2"]\arrow[d,"v_1"]&S_2\arrow[d,"u_1"])\\
    (Z_1\arrow[r,"i_1"]&S_1)
  \end{tikzcd}\]
  in $\mathcal{C}$, the diagram
  \[\begin{tikzcd}
    \beta(\widetilde{X_2})\widetilde{f_1}^*\arrow[r,"\beta_{f_1}"]\arrow[d,"\Phi_{\widetilde{X_2}}"]&\widetilde{f_1}^*\beta(\widetilde{X_1})\arrow[d,"\Phi_{\widetilde{X_1}}"]\\
    \beta(\widetilde{X_2})'\widetilde{f_1}^*\arrow[r,"\beta_{f_1}'"]&\widetilde{f_1}^*\beta'(\widetilde{X_1})
  \end{tikzcd}\]
  of functors commutes. This is true since the diagram is the big outside diagram of the diagram
  \[\begin{tikzcd}
    \beta(\widetilde{X_2})\widetilde{f_1}^*\arrow[d,"ad'^{-1}"]\arrow[rrr,"\beta_{\widetilde{f_1}}"]\arrow[rrd,"ad'^{-1}"]&&&\widetilde{f_1}^*\beta(\widetilde{X_1})\arrow[d,"ad'^{-1}"]\\
    \beta(\widetilde{X_2})\rho_{i_2}^*\rho_{i_2*}\widetilde{f_1}^*\arrow[d,"\beta_{\rho_{i_2}}"]\arrow[r,"Ex^{-1}"]&\beta(\widetilde{X_2})\rho_{i_2}^*u_1^*\rho_{i_1*} \arrow[d,"\beta_{\rho_{i_2}}"]\arrow[r,"\sim"]&\beta(\widetilde{X_2})\widetilde{f_1}^*\rho_{i_1}^*\rho_{i_1*}\arrow[r,"\beta_{\widetilde{f_1}}"]&\widetilde{f_1}^* \beta(\widetilde{X_1})\rho_{i_1}^*\rho_{i_1*}\arrow[d,"\beta_{\rho_{i_1}}"]\\
    \rho_{i_2}^*\beta(S_2)\rho_{i_2*}\widetilde{f_1}\arrow[d,"\sim"]\arrow[r,"Ex^{-1}"]&\rho_{i_2}^*\beta(S_2)u_1^*\rho_{i_1*}\arrow[d,"\sim"]\arrow[rr,"\beta_{u_1}"]&& \widetilde{f_1}^*\rho_{i_1}^*\beta(S_1)\rho_{i_1*}\arrow[d,"\sim"]\\
    \rho_{i_2}^*\epsilon(S_2)\rho_{i_2*}\widetilde{f_1}\arrow[d,"\sim"]\arrow[r,"Ex^{-1}"]&\rho_{i_2}^*\epsilon(S_2)u_1^*\rho_{i_1*}\arrow[d,"\sim"]\arrow[rr,"\epsilon_{u_1}"]&& \widetilde{f_1}^*\rho_{i_1}^*\epsilon(S_1)\rho_{i_1*}\arrow[d,"\sim"]\\
    \rho_{i_2}^*\beta'(S_2)\rho_{i_2*}\widetilde{f_1}\arrow[d,"\beta_{\rho_{i_2}}'"]\arrow[r,"Ex^{-1}"]&\rho_{i_2}^*\beta(S_2)u_1^*\rho_{i_1*}\arrow[d,"\beta_{\rho_{i_2}}'"] \arrow[rr,"\beta_{u_1}'"]&& \widetilde{f_1}^*\rho_{i_1}^*\beta(S_1)\rho_{i_1*}\arrow[d,"\beta_{\rho_{i_1}}'"]\\
    \beta'(\widetilde{X_2})\rho_{i_2}^*\rho_{i_2*}\widetilde{f_1}^*\arrow[d,"ad'"]\arrow[r,"Ex^{-1}"]&\beta'(\widetilde{X_2})\rho_{i_2}^*u_1^*\rho_{i_1*}\arrow[r,"\sim"]& \beta'(\widetilde{X_2})\widetilde{f_1}^*\rho_{i_1}^*\rho_{i_1*}\arrow[r,"\beta_{\widetilde{f_1}}'"]\arrow[lld,"ad'"]&\widetilde{f_1}^*\beta'(\widetilde{X_1})\rho_{i_1}^* \rho_{i_1*}\arrow[d,"ad'"]\\
    \beta'(\widetilde{X_2})\widetilde{f_1}^*\arrow[rrr,"\beta_{\widetilde{f_1}}'"]&&&\widetilde{f_1}^*\beta'(\widetilde{X_1})
  \end{tikzcd}\]
  of functors, whose small diagrams commute. Here, the arrows denoted by $ad'^{-1}$ are defined and isomorphisms by (Loc), and the arrows denoted by $Ex^{-1}$ are defined and an isomorphism by (\ref{2.8}). Thus $\beta$ is isomorphic to $\beta'$, so we have proved the uniqueness of $\beta$.
\end{proof}
\section{Proof of (\ref{0.4}), part III}
\begin{none}\label{5.1}
  Under the notations and hypotheses of (\ref{2.5}), assume $\mathcal{C}=\mathscr{S}$ and $\mathcal{E}=\mathscr{S}^{sm}$, and consider a commutative diagram
  \[\begin{tikzcd}
    \mathcal{E}\arrow[rd,"V"']\arrow[r,"U"]&\widetilde{\mathcal{C}}\arrow[d,"G"]\arrow[ld,phantom,"{\rotatebox[origin=c]{0}{$\Leftrightarrow$}}_{\delta}",very near start]\arrow[r,"F"]&\mathcal{C}\\
    \,&{\rm Tri}^\otimes
  \end{tikzcd}\]
  of $2$-categories where
  \begin{enumerate}[(i)]
    \item $V$ is a $\mathscr{P}'$-premotivic pseudofunctor satisfying (Loc).
    \item $G$ is a $\widetilde{\mathscr{P}}$-premotivic pseudofunctor satisfying (Loc),
    \item $\delta$ is a pseudonatural equivalence.
  \end{enumerate}
  In this section, we will show that for any morphism $\widetilde{f_1}:\widetilde{X_2}\rightarrow \widetilde{X_1}$ given by a commutative diagram
  \[\begin{tikzcd}
    (Z_1\arrow[d,"{\rm id}"]\arrow[r,"i_2"]&S_2)\arrow[d,"u_1"]\\
    (Z_1\arrow[r,"i_1"]&S_1)
  \end{tikzcd}\]
  in $\mathcal{C}$, $\widetilde{f_1}^*$ is an equivalence.
\end{none}
\begin{prop}\label{5.2}
  Under the notations and hypotheses of (\ref{5.1}), if $u_1$ is a closed immersion, then $\widetilde{f_1}^*$ is an equivalence.
\end{prop}
\begin{proof}
  It suffices to show that the counit and unit
  \[\widetilde{f_1}^*\widetilde{f_1}_*\stackrel{ad'}\longrightarrow {\rm id},\quad {\rm id}\stackrel{ad}\longrightarrow \widetilde{f_1}_*\widetilde{f_1}^*\]
  are isomorphisms. The counit is the composition
  \[\rho_{i_2}^*u_1^*\rho_{i_1*}\rho_{i_1}^!u_{1*}\rho_{i_2*}\stackrel{ad'}\longrightarrow \rho_{i_2}^*u_1^*u_{1*}\rho_{i_2*}\stackrel{ad'}\longrightarrow \rho_{i_2}^*\rho_{i_2*} \stackrel{ad'}\longrightarrow {\rm id}.\]
  The first arrow is an isomorphism by (\ref{4.8}(1)), and the second and third arrows are isomorphisms by (Loc). Thus the composition is also an isomorphism.

  The unit ${\rm id}\stackrel{ad}\longrightarrow \widetilde{f_1}_*\widetilde{f_1}^*$ is the composition
  \begin{equation}\label{5.2.1}
    {\rm id}\stackrel{ad}\longrightarrow \rho_{i_1}^!\rho_{i_1*}\stackrel{ad}\longrightarrow \rho_{i_1}^!u_{1*}u_1^*\rho_{i_1*}\stackrel{ad}\longrightarrow \rho_{i_1}^!u_{1*}\rho_{i_2*}\rho_{i_2}^*u_1^*\rho_{i_1*}.
  \end{equation}
  The first arrow is an isomorphism by (Loc), and the third arrow is an isomorphism by (\ref{4.8}(2)). Thus the remaining is to show that the second arrow is an isomorphism.

  Consider the commutative diagram
  \[\begin{tikzcd}
    \rho_{i_1*}\arrow[rd,"\sim"']\arrow[r,"ad"]&u_{1*}u_1^*\rho_{i_1*}\arrow[d,"Ex"]\\
    &u_{1*}\rho_{i_2*}{\rm id}^*
  \end{tikzcd}\]
  of functors. The vertical arrow is an isomorphism by (\ref{2.8}), so the horizontal arrow is an isomorphism. Thus the second arrow of (\ref{5.2.1}) is an isomorphism.
\end{proof}
\begin{prop}\label{5.3}
  Under the notations and hypotheses of (\ref{5.1}), if $u_1$ is an open immersion, then $\widetilde{f_1}^*$ is an equivalence.
\end{prop}
\begin{proof}
    It suffices to show that the counit and unit
  \[\widetilde{f_1}^*\widetilde{f_1}_*\stackrel{ad'}\longrightarrow {\rm id},\quad {\rm id}\stackrel{ad}\longrightarrow \widetilde{f_1}_*\widetilde{f_1}^*\]
  are isomorphisms. The counit is the composition
  \[\rho_{i_2}^*u_1^*\rho_{i_1*}\rho_{i_1}^!u_{1*}\rho_{i_2*}\stackrel{ad'}\longrightarrow \rho_{i_2}^*u_1^*u_{1*}\rho_{i_2*}\stackrel{ad'}\longrightarrow \rho_{i_2}^*\rho_{i_2*}\stackrel{ad'}\longrightarrow {\rm id}.\]
  The first arrow is an isomorphism by (\ref{4.8}(1)), and the second arrow is an isomorphism by (\ref{2.7}(5)). The third arrow is also an isomorphism by (Loc). Thus the composition is an isomorphism.

  The unit ${\rm id}\stackrel{ad}\longrightarrow \widetilde{f_1}_*\widetilde{f_1}^*$ is the composition
  \[{\rm id}\stackrel{ad}\longrightarrow \rho_{i_1}^!\rho_{i_1*}\stackrel{ad}\longrightarrow \rho_{i_1}^!u_{1*}u_1^*\rho_{i_1*}\stackrel{ad}\longrightarrow \rho_{i_1}^!u_{1*}\rho_{i_2*}\rho_{i_2}^*u_1^*\rho_{i_1*}.\]
  The first arrow is an isomorphism by (Loc), and the third arrow is an isomorphism by (\ref{4.8}(1)). Thus the remaining is to show that the second arrow is an isomorphism. For this, by definition, it suffices to show that the morphism
  \[K\stackrel{ad}\longrightarrow u_{1*}u_1^*K\]
  in $V(S_1)$ is an isomorphism for any object $K$ of $V(S_1)$ such that $j_1^*K=0$. Here, $j_1$ denotes the complement of $i_1$.

  Consider the Cartesian diagram
  \[\begin{tikzcd}
    S_2\arrow[d,"u_1"]\arrow[r,"j_2",leftarrow]&U_2\arrow[d,"w_1"]\\
    S_1\arrow[r,"j_1",leftarrow]&U_1
  \end{tikzcd}\]
  in $\mathcal{C}$. Then we have the Mayer-Vietoris distinguish triangle
  \[K\longrightarrow u_{1*}u_1^*K\oplus j_{1*}j_1^*K\longrightarrow (j_1w_1)_*(j_1w_1)^*K\longrightarrow K[1]\]
  in $V(S_1)$. Since $j_1^*K=0$, we have $(j_1w_1)^*K=0$. Thus the morphism $K\stackrel{ad}\longrightarrow u_{1*}u_1^*K$ is an isomorphism.
\end{proof}
\begin{prop}\label{5.4}
  Under the notations and hypotheses of {\rm (\ref{5.1})}, the question that $\widetilde{f_1}^*$ is an equivalence is Zariski local on $S_1$.
\end{prop}
\begin{proof}
  Let $t_1:S_1'\rightarrow S_1$ be a Zariski cover induced by open immersions
  \[j_1:U_1\rightarrow S_1,\ldots,j_r:U_r\rightarrow S_1,\]
  and consider the diagram
  \[\begin{tikzcd}
      \widetilde{X_2}'\arrow[dd,"\widetilde{f_1}'"]\arrow[rr,"\rho_{i_2'}"]\arrow[rd,"\widetilde{g_1}'"]&&S_2'\arrow[dd,"u_1'",near start]\arrow[rd,"t_2"]\\
      &\widetilde{X_2}\arrow[rr,"\rho_{i_2}",crossing over,near start]&&S_2\arrow[dd,"u_1"]\\
      \widetilde{X_1}'\arrow[rr,"\rho_{i_1'}",near start]\arrow[rd,"\widetilde{g_1}"]&&S_1'\arrow[rd,"t_1"]\\
      &\widetilde{X_1}\arrow[rr,"\rho_{i_1}"]\arrow[uu,leftarrow,crossing over,"\widetilde{f_1}",near end]&&S_1
  \end{tikzcd}\]
  in $\widetilde{\mathcal{C}}$ where each square is Cartesian. The statement is that the equivalence of $\widetilde{f_1'}$ implies the equivalence of $\widetilde{f_1}$. Hence assume that $\widetilde{f_1'}$ is an equivalence.

  In the commutative diagrams
  \[\begin{tikzcd}
    \widetilde{g_2}^*\widetilde{f_1}^*\widetilde{f_1}_*\arrow[r,"ad'"]\arrow[d,"\sim"]&\widetilde{g_2}^*\\
    \widetilde{f_1'}^*\widetilde{g_1}^*\widetilde{f_1}_*\arrow[d,"Ex"]\\
    \widetilde{f_1'}^*\widetilde{f_1'}_*\widetilde{g_2}^*\arrow[ruu,"ad'"']
  \end{tikzcd}\quad
  \begin{tikzcd}
    \widetilde{g_1}^*\arrow[r,"ad"]\arrow[rdd,"ad"']&\widetilde{g_1}^*\widetilde{f_1*}\widetilde{f_1}^*\arrow[d,"Ex"]\\
    &\widetilde{f_1'}_*\widetilde{g_2}^*\widetilde{f_1}^*\arrow[d,"\sim"]\\
    &\widetilde{f_1'}_*\widetilde{f_1'}^*\widetilde{g_1}
  \end{tikzcd}\]
  of functors, the diagonal arrows are isomorphisms since $\widetilde{f_1'}$ is an equivalence. The arrows denoted by $Ex$ are also isomorphisms by (B--3). Thus the horizontal arrows are isomorphisms.

  Then the natural transformations
  \[\rho_{i_2'*}\widetilde{g_2}^*\widetilde{f_1}^*\widetilde{f_1}_*\stackrel{ad'}\longrightarrow \rho_{i_2'*}\widetilde{g_2}^*,\quad  \rho_{i_1*}'\widetilde{g_1}^*\stackrel{ad}\longrightarrow \rho_{i_1*}'\widetilde{g_1}^*\widetilde{f_1*}\widetilde{f_1}^*\]
  are isomorphisms, so by (\ref{2.8}), the natural transformations
  \[t_2^*\rho_{i_2*}\widetilde{f_1}^*\widetilde{f_1}_*\stackrel{ad'}\longrightarrow t_2^*\rho_{i_2*},\quad  t_1^*\rho_{i_1*}\stackrel{ad}\longrightarrow t_1^*\rho_{i_1*}\widetilde{f_1*}\widetilde{f_1}^*\]
  are isomorphisms. Then for any open immersion $j:U\rightarrow S_1$ that is an intersection of $j_1,\ldots,j_n$, the natural transformations
  \[j_\sharp j^*\rho_{i_2*}\widetilde{f_1}^*\widetilde{f_1}_*\stackrel{ad'}\longrightarrow j_\sharp j^*\rho_{i_2*},\quad  j_\sharp j^*\rho_{i_1*}\stackrel{ad}\longrightarrow j_\sharp j^*\rho_{i_1*}\widetilde{f_1*}\widetilde{f_1}^*\]
  are isomorphisms. Applying the Mayer-Vietoris distinguished triangle repeatedly, we see that the natural transformations
  \[\rho_{i_2*}\widetilde{f_1}^*\widetilde{f_1}_*\stackrel{ad'}\longrightarrow \rho_{i_2*},\quad  \rho_{i_1*}\stackrel{ad}\longrightarrow \rho_{i_1*}\widetilde{f_1*}\widetilde{f_1}^*\]
  are isomorphisms. Since $\rho_{i_1*}$ and $\rho_{i_2*}$ are fully faithful by (Loc), the counit and unit
  \[\widetilde{f_1}^*\widetilde{f_1}_*\stackrel{ad'}\longrightarrow {\rm id},\quad {\rm id}\stackrel{ad}\longrightarrow\widetilde{f_1*}\widetilde{f_1}^*\]
  are isomorphism. Thus $\widetilde{f_1}$ is an equivalence.
\end{proof}
\begin{prop}\label{5.5}
  Under the notations and hypotheses of (\ref{5.1}), consider morphisms
  \[\begin{tikzcd}
    \widetilde{X_2}\arrow[r,shift left,"\widetilde{f_1}"]\arrow[r,shift right,"\widetilde{f_2}"']&\widetilde{X_1}
  \end{tikzcd}\]
  in $\widetilde{\mathcal{C}}$ given by a diagram
  \[\begin{tikzcd}
    Z_1\arrow[d,shift left,"{\rm id}"]\arrow[d,shift right,"{\rm id}"']\arrow[r,"i_2"]&S_2\arrow[d,shift left,"u_1"]\arrow[d,shift right,"u_2"']\\
    Z_1\arrow[r,"i_1"]&S_1
  \end{tikzcd}\]
  in $\mathcal{C}$. Then there is a natural isomorphism
  \[\widetilde{f_1}^*\cong \widetilde{f_2}^*.\]
\end{prop}
\begin{proof}
  The morphisms $u_1$ and $u_2$ have the factorizations
  \[\begin{tikzcd}
    S_2\arrow[r,shift left,"{(u_1,{\rm id})}"]\arrow[r,shift right,"{(u_2,{\rm id})}"']&S_1\times S_2\arrow[r]&S_1
  \end{tikzcd}\]
  where the right arrow is the projection. The left arrows are closed immersions since $S_1$ is separated. Hence replacing $S_1$ by $S_1\times S_2$, we may assume that $u_1$ and $u_2$ are closed immersions.

  Then consider the diagram
  \[\begin{tikzcd}
    &\widetilde{X_2}\arrow[r,shift left,"\widetilde{g_1}"]\arrow[r,shift right,"\widetilde{g_2}"']\arrow[d,"\widetilde{g_3}"]&\widetilde{Y_1}\arrow[d,"\widetilde{h_3}"]\arrow[r,"\widetilde{p_1}"]&\widetilde{X_1}\\
    \widetilde{X_2}\arrow[r,"\widetilde{g_4}"]&\widetilde{Y_1}\arrow[r,shift left,"\widetilde{h_1}"]\arrow[r,shift right,"\widetilde{h_2}"']&\widetilde{Y_2}
  \end{tikzcd}\]
  in $\widetilde{\mathcal{C}}$ given by the diagram
  \[\begin{tikzcd}
    &&Z_1\arrow[dd,"{\rm id}"]\arrow[rr,"{\rm id}"]\arrow[rd,"i_2"]&&Z_1\arrow[rd,"i_3"]\arrow[dd,"{\rm id}",near start]\arrow[rr,"{\rm id}"]&&Z_1\arrow[rd,"i_1"]\\
    &&&S_2\arrow[rr,shift left,"g_1",crossing over]\arrow[rr,shift right,"g_2"',crossing over]&&S_1\times S_1\times S_2\arrow[dd,"h_3",near start]\arrow[rr,"p_1"]&&S_1\\
    Z_1\arrow[rr,"{\rm id}"]\arrow[rd,"i_2"]&&Z_1\arrow[rd,"i_3"]\arrow[rr,"{\rm id}",near start]&&Z_1\arrow[rd,"i_4"]\\
    &S_2\arrow[rr,"g_4"]&&S_1\times S_1\times S_2\arrow[uu,leftarrow,crossing over,"g_3",near start] \arrow[rr,shift left,"h_1"]\arrow[rr,shift right,"h_2"']&&S_1\times S_1\times S_1
  \end{tikzcd}\]
  in $\mathcal{C}$ where
  \begin{enumerate}[(i)]
    \item $i_3=(i_1,i_1,i_2)$, $i_4=(i_1,i_1,i_1)$,
    \item $g_1=(u_1,u_2,{\rm id})$, $g_2=(u_2,u_1,{\rm id})$, $g_3=(u_1,u_1,{\rm id})$, $g_4=(u_2,u_2,{\rm id})$,
    \item $p_1$ denotes the first projection,
    \item $h_1=\tau_{23}\circ ({\rm id}\times {\rm id}\times u_2)$, $h_2=\tau_{13}\circ ({\rm id}\times {\rm id}\times u_2)$, $h_3={\rm id}\times {\rm id}\times u_1$ where $\tau_{ij}$ denotes the transposition of $i$-th and $j$-th $S_1$.
  \end{enumerate}
  Note that $g_4$ is a closed immersion since $S_1$ is separated. We have the natural transformations
  \[\begin{split}
    \widetilde{f_1}_*&\stackrel{\sim}\longrightarrow \widetilde{p_1}_*\widetilde{g_1}_*\stackrel{ad'^{-1}}\longrightarrow  \widetilde{p_1}_*\widetilde{h_3}^*\widetilde{h_3}_* \widetilde{g_1}_*\stackrel{\sim}\longrightarrow  \widetilde{p_1}_*\widetilde{h_3}^*\widetilde{h_1}_*\widetilde{g_3}_*\stackrel{ad}\longrightarrow  \widetilde{p_1}_*\widetilde{h_3}^*\widetilde{h_1}_*\widetilde{g_4}_*\widetilde{g_4}^*\widetilde{g_3}_*\\
    &\stackrel{\sim}\longrightarrow \widetilde{p_1}_*\widetilde{h_3}^*(\widetilde{h_1}\widetilde{g_4})_*\widetilde{g_4}^*\widetilde{g_3}_*\stackrel{\sim}\longrightarrow
    \widetilde{p_1}_*\widetilde{h_3}^*(\widetilde{h_2}\widetilde{g_4})_*\widetilde{g_4}^*\widetilde{g_3}_*\stackrel{\sim}\longrightarrow \widetilde{p_1}_*\widetilde{h_3}^*\widetilde{h_2}_*\widetilde{g_4}_*\widetilde{g_4}^*\widetilde{g_3}_*\\
    &\stackrel{ad^{-1}}\longrightarrow
    \widetilde{p_1}_*\widetilde{h_3}^*\widetilde{h_2}_*\widetilde{g_3}_*\stackrel{\sim}\longrightarrow \widetilde{p_1}_*\widetilde{h_3}^*\widetilde{h_3}_*\widetilde{g_2}_*\stackrel{ad}\longrightarrow \widetilde{p_1}_*\widetilde{g_2}_*\stackrel{\sim}\longrightarrow \widetilde{f_2}_*.
  \end{split}\]
  Here, the second arrow is defined and an isomorphism by (Loc), and the tenth arrow is an isomorphism by the same reason. The fourth arrow is an isomorphism by (\ref{5.2}) since $g_4$ is a closed immersion, and the eighth arrow is defined and an isomorphism by the same reason. Thus the composition is an isomorphism. Its left adjoint is $\widetilde{f_1}^*\cong \widetilde{f_2}^*$.
\end{proof}
\begin{thm}\label{5.6}
  Under the notations and hypotheses of {\rm (\ref{5.1})}, $\widetilde{f_1}^*$ is an equivalence.
\end{thm}
\begin{proof}
  Consider the diagram
  \[\begin{tikzcd}
    Z_1\arrow[r,"i_3"]\arrow[d,"i_1"]&S_1\times S_2\\
    S_1
  \end{tikzcd}\]
  in $\mathcal{C}$ where $i_3$ denotes the morphism induced by $i_1$ and $i_2$. By \cite[3.9]{Sch05}, this can be extended to a Cartesian diagram
  \[\begin{tikzcd}
    Z_1\arrow[r,"i_3"]\arrow[d,"i_1"]&S_1\times S_2\arrow[d,"b"]\\
    S_1\arrow[r,"a"]&Y
  \end{tikzcd}\]
  of schemes such that $a$ and $b$ are closed immersions. Choose a Zariski cover $\{Y_i\}_{i\in I}$ of $Y$ such that $I$ is finite and each $Y_i$ is affine, and put
  \[U_i=a^{-1}(Y_i),\quad V_i=q_1^{-1}b^{-1}(Y_i),\quad W_i=i_1^{-1}(U_i)\]
  where $q_1:S_2\rightarrow S_1\times S_2$ denotes the graph morphism of $u_1$. Then we have the commutative diagram
  \[\begin{tikzcd}
      W_i\arrow[d,"c_i"]\arrow[r,"d_i"]&U_i\times V_i\arrow[d,"b_i"]\\
      U_i\arrow[r,"a_i"]&Y_i
    \end{tikzcd}\]
  in $\mathcal{C}$ where $W_i=i_1^{-1}(U_i)$, $c_i$ and $d_i$ are closed immersions induced by $i_1$ and $i_2$, and $a_i$ and $b_i$ are morphisms induced by $a$ and $b$. Note that $a_i$ is a closed immersion and that $b_i$ is an immersion.

  The morphism $u_1:S_2\rightarrow S_1$ has the factorization
  \[S_2\stackrel{q_1}\rightarrow S_1\times S_2\stackrel{p_1}\rightarrow S_1\]
  where $p_1$ denotes the projection. Since $S_1$ is separated, $q_1$ is a closed immersion. Hence by (\ref{5.2}), we reduce to the case when $\widetilde{f_1}$ is given by the commutative diagram
  \[\begin{tikzcd}
    (Z_1\arrow[d,"{\rm id}"]\arrow[r,"i_3"]&S_1\times S_2)\arrow[d,"p_1"]\\
    (Z_1\arrow[r,"i_1"]&S_1)
  \end{tikzcd}\]
  in $\mathcal{C}$. Then by (\ref{5.4}), we reduce to the case when $\widetilde{f_1}$ is given by the commutative diagram
  \[\begin{tikzcd}
    (W_i\arrow[d,"{\rm id}"]\arrow[r,"d_i'"]&U_i\times S_2)\arrow[d,"r_i"]\\
    (W_i\arrow[r,"c_i"]&U_i)
  \end{tikzcd}\]
  in $\mathcal{C}$ where $d_i'$ denotes the morphism induced by $i_3$ and $r_i$ denotes the projection. Since $d_i'$ has the factorization
  \[W_i\stackrel{d_i}\rightarrow U_i\times V_i\rightarrow U_i\times S_2\]
  where the second arrow is the open immersion induced by the open immersion $V_i\rightarrow S_2$, by (\ref{5.3}), we reduce to the case when $\widetilde{f_1}$ is given by the commutative diagram
  \[\begin{tikzcd}
    (W_i\arrow[d,"{\rm id}"]\arrow[r,"d_i"]&U_i\times V_i)\arrow[d,"r_i'"]\\
    (W_i\arrow[r,"c_i"]&U_i)
  \end{tikzcd}\]
  in $\mathcal{C}$ where $r_i'$ denotes the projection. By (\ref{5.2}), since $a_i$ is a closed immersion, we reduce to the case when $\widetilde{f_1}$ is given by the commutative diagram
  \[\begin{tikzcd}
    (W_i\arrow[d,"{\rm id}"]\arrow[r,"d_i"]&U_i\times V_i)\arrow[d,"a_ir_i'"]\\
    (W_i\arrow[r,"a_ic_i"]&Y_i)
  \end{tikzcd}\]
  in $\mathcal{C}$. Consider the morphism $\widetilde{f_2}$ given by the commutative diagram
  \[\begin{tikzcd}
    (W_i\arrow[d,"{\rm id}"]\arrow[r,"d_i"]&U_i\times V_i)\arrow[d,"b_i"]\\
    (W_i\arrow[r,"a_ic_i"]&Y_i)
  \end{tikzcd}\]
  in $\mathcal{C}$. By (\ref{5.2}) and (\ref{5.3}), $\widetilde{f_2}^*$ is an equivalence since $b_i$ is an immersion, and by (\ref{5.5}), $\widetilde{f_1}^*\cong \widetilde{f_2}^*$. Thus $\widetilde{f_1}^*$ is an equivalence.
\end{proof}
\section{Proof of (\ref{0.4}), part IV}
\begin{thm}\label{6.1}
  Consider a commutative diagram
  \[\begin{tikzcd}
    &&U_{12}\arrow[lld,"u_1"']\arrow[rd,"u_2"]\arrow[lddd,"f_{12}"',near start]\\
    U_1\arrow[rdd,"f_1"']\arrow[rd,"j_1"]&&&U_2\arrow[lld,"j_2"',crossing over]\arrow[lldd,"f_2"]\\
    &X\arrow[d,"f"']\\
    &S
  \end{tikzcd}\]
  in $\mathscr{S}$ where the small inner square is Cartesian and $j_1$ and $j_2$ are open immersions. Let $\mathscr{P}$ be a class of morphisms in $\mathscr{S}$ containing all open immersions and stable by compositions and pullbacks, and let $H:\mathscr{S}\rightarrow {\rm Tri}^\otimes$ be a $\mathscr{P}$-motivic pseudofunctor satisfying {\rm (Loc)}. If $f_1^*$, $f_2^*$, and $f_{12}^*$ have left adjoints, then $f^*$ also has a left adjoint.
\end{thm}
\begin{proof}
  Put $j_{12}=j_1u_1$, and we denote by $f_{1\sharp}$ (resp.\ $f_{2\sharp}$, resp.\ $f_{12\sharp}$) the left adjoint of $f_1^*$ (resp.\ $f_2^*$, resp.\ $f_{12}^*$).\\[4pt]
  (I) {\it Construction of $f_\sharp K$.} Consider the natural transformation
  \[\eta:f_{12\sharp}j_{12}^*\longrightarrow f_{1\sharp}j_1^*\oplus f_{2\sharp}j_2^*\]
  given by the left adjoint of the composition
  \[\begin{split}
    j_{1*}f_1^*\oplus j_{2*}f_2^*&\stackrel{\sim}\longrightarrow j_{1*}j_1^*f^*\oplus j_{2*}j_2^*f^*\\
    &\stackrel{ad\oplus ad}\longrightarrow j_{1*}u_{1*}u_1^*j_1^*f^*\oplus j_{2*}u_{2*}u_2^* j_2^*f^*\stackrel{\sim}\longrightarrow j_{12*}f_{12}^*\oplus j_{12*}f_{12}^*\longrightarrow j_{12*}f_{12}^*
  \end{split}\]
  where the fourth arrow is the summation. For any object $K$ of $H(X)$, using an axiom of triangulated categories, {\it choose} $f_\sharp K$ as a cone of $f_{12\sharp}j_{12}^*\stackrel{\eta}\longrightarrow f_{1\sharp}j_1^*\oplus f_{2\sharp}j_2^*$. Then we have a distinguished triangle
  \begin{equation}\label{6.1.1}
    f_{12\sharp}j_{12}^*K\stackrel{\eta}\longrightarrow f_{1\sharp}j_1^*K\oplus f_{2\sharp}j_2^*K\longrightarrow f_\sharp K\longrightarrow f_{12\sharp}j_{12}^*K\longrightarrow f_{1\sharp}j_1^*[1].
  \end{equation}
  (II) {\it Construction of $K\stackrel{ad}\longrightarrow f^*f_\sharp K$.}
  Let $K$ be an object of $H(X)$. Consider the natural transformation
  \[\mu_1:j_{1\sharp}\longrightarrow f^*f_{1\sharp}\]
  given by the left adjoint of the composition
  \[f_1^*f_*\stackrel{\sim}\longrightarrow j_1^*f^*f_*\stackrel{ad'}\longrightarrow j_1^*.\]
  We similarly have the natural transformations
  \[\mu_2:j_{2\sharp}\longrightarrow f^*f_{2\sharp},\]
  \[\mu_{12}:j_{12\sharp}\longrightarrow f^*f_{12\sharp}.\]
  Then we have the commutative diagram
  \begin{equation}\label{6.1.2}\begin{tikzcd}
    j_{12\sharp}j_{12}^*K\arrow[d,"\mu_{12}"]\arrow[r]&j_{1\sharp}j_1^*K\oplus j_{2\sharp}j_2^*K\arrow[d,"\mu_1\oplus \mu_2"]\arrow[r]&K\arrow[r]&j_{12\sharp}j_{12}^*K[1] \arrow[d,"\mu_{12}"]\\
    f^*f_{12\sharp}j_{12}^*K\arrow[r,"\eta"]&f^*f_{1\sharp}j_1^*K\oplus f^*f_{2\sharp}j_2^*K\arrow[r]&f^*f_\sharp K\arrow[r]&f^*f_{12\sharp}j_{12}^*K[1]
  \end{tikzcd}\end{equation}
  in $H(X)$ where the first row is the Mayer-Vietoris distinguished triangle and the second row is induced by the distinguished triangle (\ref{6.1.1}). Using an axiom of triangulated categories, {\it choose} a morphism $K\stackrel{ad}\longrightarrow f^*f_\sharp K$ in $H(X)$ making the above diagram commutative.\\[4pt]
  (III) {\it Construction of ${\rm Hom}_{H(S)}(f_\sharp K,L)\stackrel{\sim}\longrightarrow {\rm Hom}_{H(X)}(K,f^*L)$.} Consider the homomorphism
  \[\varphi:{\rm Hom}_{H(S)}(f_\sharp K,L)\longrightarrow {\rm Hom}_{H(X)}(K,f^*L)\]
  of abelian groups given by
  \[(f_\sharp K\stackrel{\alpha}\longrightarrow L)\mapsto (K\stackrel{ad}\longrightarrow f^*f_\sharp K\stackrel{\alpha}\longrightarrow f^*L).\]
  We will show that $\varphi$ is bijective. Using the Mayer-Vietoris distinguished triangle
  \[j_{12\sharp}j_{12}^*K\longrightarrow j_{1\sharp}j_1^*K\oplus j_{2\sharp}j_2^*K\longrightarrow K\longrightarrow j_{12\sharp}j_{12}^*K[1],\]
  we reduce to the case when $K$ is in the essential image of $j_{1\sharp}$, $j_{2\sharp}$, or $j_{12\sharp}$. We will only consider the case when $K$ is in the essential image of $j_{1\sharp}$ since the proofs for the other two cases are the same.

  If $K=j_{1\sharp} K'$ for some object $K'$ of $H(X_1)$, using (\ref{2.7}(5)) to (\ref{6.1.1}), we have the isomorphism
  \[\psi:f_{1\sharp}K'\longrightarrow f_{\sharp}j_{1\sharp}K'.\]
  Then using (\ref{2.7}(5)) to (\ref{6.1.2}), we have the commutative diagram
  \begin{equation}\label{6.1.3}\begin{tikzcd}
    &j_{1\sharp}K'\arrow[ld,"\mu_1"']\arrow[d,"ad"]\\
    f^*f_{1\sharp}K'\arrow[r,"\psi"]&f^*f_{\sharp}j_{1\sharp}K'
  \end{tikzcd}\end{equation}
  in $H(X)$.

  Let $L$ be an object of $H(S)$, and consider the diagram
  \begin{equation}\label{6.1.6}\begin{tikzcd}
    {\rm Hom}_{H(S)}(f_\sharp j_{1\sharp}K',L)\arrow[r,"\varphi"]\arrow[d,"\psi"]&{\rm Hom}_{H(X)}( j_{1\sharp}K',f^*L)\arrow[r,"\nu'"]&{\rm Hom}_{H(X_1)}(K',j_1^*f^*L)\arrow[d,"\sim"]\\
    {\rm Hom}_{H(S)}(f_{1\sharp}K',L)\arrow[rr,"\nu"]&&{\rm Hom}_{H(X_1)}(K',f_1^*L)
  \end{tikzcd}\end{equation}
  of abelian groups where $\nu$ (resp.\ $\nu'$) denotes the isomorphisms induced by the adjunction of $f_{1\sharp}$ and $f_1^*$ (resp.\ $j_{1\sharp}$ and $j_1^*$). Then consider an element $\alpha\in {\rm Hom}_{H(S)}(f_{\sharp}j_{1\sharp}K',L)$. Its image in ${\rm Hom}_{H(X_1)}(K',f_1^*L)$ via the upper right route and the lower left route are the compositions
  \[K'\stackrel{ad}\longrightarrow j_1^*j_{1\sharp}K'\stackrel{ad}\longrightarrow j_1^*f^*f_\sharp j_{1\sharp}K'\stackrel{\alpha}\longrightarrow j_1^*f^*L\stackrel{\sim}\longrightarrow f_1^*L,\]
  \[K'\stackrel{ad}\longrightarrow f_1^*f_{1\sharp}K'\stackrel{\psi^{-1}}\longrightarrow f_1^*f_\sharp j_{1\sharp}K'\stackrel{\alpha}\longrightarrow f_1^*L\]
  respectively. They are equal since the diagram
  \[\begin{tikzcd}
    K'\arrow[r,"ad"]\arrow[rdd,"ad"']&j_1^*j_{1\sharp}K'\arrow[d,"\mu"]\arrow[r,"ad"]&j_1^*f^*f_\sharp j_{1\sharp}K'\arrow[dd,"\sim"]\arrow[r,"\alpha"]&j_1^*f^*L\arrow[dd,"\sim"]\\
    &j_1^*f^*f_{1\sharp}K'\arrow[ru,"\psi"]\arrow[d,"\sim"]\\
    &f_1^*f_{1\sharp}K'\arrow[r,"\psi"]&f_1^*f_\sharp j_{1\sharp}K'\arrow[r,"\alpha"]&f_1^*L
  \end{tikzcd}\]
  in $H(S)$ commutes by the commutativity of (\ref{6.1.3}). Thus (\ref{6.1.6}) commutes. This means that the homomorphism
  \[\varphi:{\rm Hom}_{H(S)}(f_\sharp K,L)\longrightarrow {\rm Hom}_{H(X)}(K,f^*L)\]
  is an isomorphism.\\[4pt]
  (IV) {\it Construction of $f_\sharp K\longrightarrow f_\sharp L$.} Let $K$ and $L$ be objects of $H(X)$. By (III), the homomorphism
  \[\varphi:{\rm Hom}_{H(S)}(f_\sharp K,f_\sharp L)\longrightarrow {\rm Hom}_{H(X)}(K,f^*f_\sharp L)\]
  given by
  \[(f_\sharp K\stackrel{\alpha'}\longrightarrow f_\sharp L)\mapsto (K\stackrel{ad}\longrightarrow f^*f_\sharp K\stackrel{\alpha'}\longrightarrow f^*f_\sharp L)\]
  is an isomorphism. This means that for any morphism $\alpha:K\rightarrow L$ in $H(X)$, in the diagram
  \[\begin{tikzcd}
    K\arrow[d,"ad"]\arrow[r,"\alpha"]&L\arrow[d,"ad"]\\
    f^*f_\sharp K&f^*f_\sharp L
  \end{tikzcd}\]
  in $H(X)$, there is a unique morphism $f_\sharp K\longrightarrow f_\sharp L$ in $H(S)$ such that the induced morphism $f^*f_\sharp K\longrightarrow f^*f_\sharp L$ makes the above diagram commutative. This morphism is denoted by $f_\sharp \alpha:f_\sharp K\rightarrow f_\sharp L$.\\[4pt]
  (V) {\it Functoriality of $f_\sharp$.} Let $\alpha:K\rightarrow L$ and $\beta:L\rightarrow M$ be morphisms in $H(S)$. Consider the diagram
  \[\begin{tikzcd}
    K\arrow[d,"ad"]\arrow[r,"\alpha"]&L\arrow[d,"ad"]\arrow[r,"\beta"]&M\arrow[d,"ad"]\\
    f^*f_\sharp K&f^*f_\sharp L&f^*f_\sharp M
  \end{tikzcd}\]
  in $H(S)$. By definition, $f_\sharp \alpha$ (resp.\ $f_\sharp \beta$, resp.\ $f_\sharp (\beta\alpha)$) is a unique morphism such that the induced morphism $f^*f_\sharp K\longrightarrow f^*f_\sharp L$ (resp.\ $f^*f_\sharp K\longrightarrow f^*f_\sharp L$, resp.\ $f^*f_\sharp K\longrightarrow f^*f_\sharp M$) makes the above diagram commutative. Thus
  \[f_\sharp \beta\circ f_\sharp \alpha=f_\sharp (\beta\circ \alpha),\]
  so $f_\sharp$ is a functor.\\[4pt]
  (VI) {\it Final step of the proof.} We will show that $f_\sharp$ is left adjoint to $f^*$ to complete the proof. For this, it suffices to show that the isomorphism
  \[\varphi:{\rm Hom}_{H(S)}(f_\sharp K,L)\longrightarrow {\rm Hom}_{H(X)}(K,f^*L)\]
  is functorial on $K$ and $L$. Let $\beta:K\rightarrow K'$ be a morphism in $H(X)$, and consider the diagram
  \begin{equation}\label{6.1.4}\begin{tikzcd}
    {\rm Hom}_{H(S)}(f_\sharp K',L)\arrow[r,"\varphi"]\arrow[d,"\beta"]&{\rm Hom}_{H(X)}(K',f^*L)\arrow[d,"\beta"]\\
    {\rm Hom}_{H(S)}(f_\sharp K,L)\arrow[r,"\varphi"]&{\rm Hom}_{H(X)}(K,f^*L)
  \end{tikzcd}\end{equation}
  of abelian groups. For an element $\alpha\in {\rm Hom}_{H(S)}(f_\sharp K',L)$, its image in ${\rm Hom}_{H(X)}(K,f^*L)$ via the upper right route and lower left route are the compositions
  \[K\stackrel{\beta}\longrightarrow K'\stackrel{ad}\longrightarrow f^*f_\sharp K'\stackrel{\alpha}\longrightarrow f^*L,\]
  \[K\stackrel{ad}\longrightarrow f^*f_\sharp K\stackrel{\beta}\longrightarrow f^*f_\sharp K'\stackrel{\alpha}\longrightarrow f^*L\]
  respectively. They are equal since the diagram
  \[\begin{tikzcd}
    K\arrow[d,"ad"]\arrow[r,"\beta"]&K'\arrow[d,"ad"]\arrow[r,"ad"]&f^*f_\sharp K'\arrow[d,"\alpha"]\\
    f^*f_\sharp K\arrow[r,"\beta"]&f^*f_\sharp K'\arrow[r,"\alpha"]&f^*L
  \end{tikzcd}\]
  in $H(X)$ commutes. Thus (\ref{6.1.4}) commutes, so $\varphi$ is functorial on $K$.

  Let $\gamma:L\rightarrow L'$ be a morphism in $H(X)$, and consider the diagram
  \begin{equation}\label{6.1.5}\begin{tikzcd}
    {\rm Hom}_{H(S)}(f_\sharp K,L)\arrow[r,"\varphi"]\arrow[d,"\gamma"]&{\rm Hom}_{H(X)}(K,f^*L)\arrow[d,"\gamma"]\\
    {\rm Hom}_{H(S)}(f_\sharp K,L')\arrow[r,"\varphi"]&{\rm Hom}_{H(X)}(K,f^*L')
  \end{tikzcd}\end{equation}
  of abelian groups. For any element $\alpha\in {\rm Hom}_{H(S)}(f_\sharp K,L)$, its image in ${\rm Hom}_{H(X)}(K,f^*L')$ via the upper right route and lower left route are both equal to the composition
  \[K\stackrel{ad}\longrightarrow f^*f_\sharp K\stackrel{\alpha}\longrightarrow f^*L\stackrel{\beta}\longrightarrow f^*L'.\]
  Thus (\ref{6.1.5}) commutes, so $\varphi$ is functorial on $L$. This completes the proof that $f_\sharp$ is left adjoint to $f^*$.
\end{proof}
\begin{prop}\label{6.4}
  Under the notations and hypotheses of {\rm (\ref{6.1})}, let $H':\mathscr{S}\rightarrow {\rm Tri}^{\otimes}$ be another pseudofunctor satisfying {\rm (Loc)}, and let $\gamma:H\rightarrow H'$ be a pseudonatural transformation satisfying {\rm (C--1)}. Assume that $f^*$, $f_1^*$, $f_2^*$, and $f_{12}^*$ have left adjoints denoted by $f_\sharp$, $f_{1\sharp}$, $f_{2\sharp}$, and $f_{3\sharp}$ respectively. If the exchange transformations
  \[f_{1\sharp}\gamma(U_1)\stackrel{Ex}\longrightarrow \gamma(S)f_{1\sharp},\quad f_{2\sharp}\gamma(U_2)\stackrel{Ex}\longrightarrow \gamma(S)f_{2\sharp},\quad f_{12\sharp}\gamma(U_{12})\stackrel{Ex}\longrightarrow \gamma(S)f_{12\sharp}\]
  are isomorphisms, then the exchange transformation
  \[f_{\sharp}\gamma(X)\stackrel{Ex}\longrightarrow \gamma(S)f_{\sharp}\]
  is an isomorphism.
\end{prop}
\begin{proof}
  For each object $Y$ of $\mathcal{C}$, let $\gamma'(Y)$ denote the right adjoint of $\gamma$. Consider the right adjoints of the above natural transformations
  \[f_1^*\gamma'(S)\stackrel{Ex}\longrightarrow \gamma'(U_1)f_1^*,\quad f_2^*\gamma'(S)\stackrel{Ex}\longrightarrow \gamma'(U_2)f_2^*,\quad f_{12}^*\gamma'(S)\stackrel{Ex}\longrightarrow \gamma'(U_1)f_{12}^*,\quad f^*\gamma'(S)\stackrel{Ex}\longrightarrow \gamma'(X)f^*.\]
  The first three natural transformations are isomorphisms by assumption. Consider the commutative diagram
  \[\begin{tikzcd}
    j_1^*f^*\gamma'(S)\arrow[d,"\sim"]\arrow[r,"Ex"]&j_1^*\gamma'(X)f^*\arrow[r,"Ex"]&\gamma'(U_1)j_1^*f^*\arrow[d,"\sim"]\\
    f_1^*\gamma'(S)\arrow[rr,"Ex"]&&\gamma'(U_1)f_1^*
  \end{tikzcd}\]
  of functors. The lower horizontal arrow is an isomorphism by assumption, and the upper right horizontal arrow is an isomorphism since $\mathscr{P}$ contains all open immersions. Thus the upper left horizontal arrow is an isomorphism. Similarly we have isomorphisms
  \[j_{2}^*f^*\gamma'(S)\stackrel{Ex}\longrightarrow j_{2}^*\gamma'(X)f^*,\quad j_{12}^*f^*\gamma'(S)\stackrel{Ex}\longrightarrow j_{12}^*\gamma'(X)f^*\]
  are isomorphisms. Consider the commutative diagram
  \[\begin{tikzcd}
    j_{12\sharp}j_{12}^*f^*\gamma'(S)\arrow[r]\arrow[d,"Ex"]&j_{1\sharp}j_1^*f^*\gamma'(S)\oplus j_{2\sharp}j_2^*f^*\gamma'(S)\arrow[r]\arrow[d,"Ex\oplus Ex"]&f^*\gamma'(S)\arrow[r]\arrow[d,"Ex"]& j_{12\sharp}j_{12}^*f^*\gamma'(S)[1]\arrow[d,"Ex"]\\
    j_{12\sharp}j_{12}^*\gamma'(X)f^*\arrow[r]&j_{1\sharp}j_1^*\gamma'(X)f^*\oplus j_{2\sharp}j_2^*\gamma'(X)f^*\arrow[r]&f^*\gamma'(S)\arrow[r]& j_{12\sharp}j_{12}^*\gamma'(X)f^*[1]\\
  \end{tikzcd}\]
  of functors where each row is induced by the Mayer-Vietoris distinguished triangle. The first and second columns are isomorphisms, so the third column is an isomorphism.
\end{proof}
\begin{none}\label{6.2}
  Now, we prove the main theorem (\ref{0.4}).
\end{none}
\begin{thm}\label{6.3}
  Consider a diagram
  \[\begin{tikzcd}
    \mathscr{S}^{sm}\arrow[rd,"V"']\arrow[rr,"W"]&&\mathscr{S}\\
    &{\rm Tri}^{\otimes}
  \end{tikzcd}\]
  of $2$-categories where
  \begin{enumerate}[{\rm (i)}]
    \item $V$ is a $Sm$-premotivic pseudofunctor satisfying {\rm (Loc)},
    \item $W$ denotes the inclusion functor.
  \end{enumerate}
  Then there is a contravariant pseudofunctor $H:\mathscr{S}\rightarrow Tri^{\otimes}$ satisfying {\rm (Loc)} and a pseudonatural equivalence $\kappa:V\rightarrow H\circ W$ making the above diagram commutes.

  The above construction is functorial in the following sense. Suppose that we have a commutative diagram
  \[\begin{tikzcd}
    \mathscr{S}^{sm}\arrow[rd,"V'"]\arrow[rd,bend right,"V"']\arrow[rr,"W"]&&\mathscr{S}\arrow[ld,"H"']\arrow[ld,"H'",bend left]\\
    &{\rm Tri}^{\otimes}&\;
  \end{tikzcd}\]
  of $2$-categories and a diagram
  \[\begin{tikzcd}
    V\arrow[d,"\epsilon"]\arrow[r,"\kappa"]&H\circ W\\
    V'\arrow[r,"\kappa'"]&H'\circ W
  \end{tikzcd}\]
  of pseudofunctors where $V'$ and $H'$ are $Sm$-premotivic pseudofunctors satisfying {\rm (Loc)}, $\kappa'$ is a pseudonatural equivalence, and $\epsilon$ is a $Sm$-premotivic pseudonatural transformation. Then there exists a $Sm$-premotivic pseudonatural transformation $H\rightarrow H'$ unique up to isomorphism such that the induced $Sm$-pseudonatural transformation $H\circ W\rightarrow H'\circ W$ makes the diagram {\rm (\ref{0.4.1})} commutative.
\end{thm}
\begin{proof}
  Consider the diagram
  \[\begin{tikzcd}
    \mathcal{E}\arrow[rd,"V"']\arrow[r,"U"]&\widetilde{\mathcal{C}}\arrow[r,"F"]&\mathscr{S}\\
    &{\rm Tri}^{\otimes}
  \end{tikzcd}\]
  of $2$-categories where $\mathcal{C}=\mathscr{S}$, $\mathcal{E}=\mathscr{S}^{sm}$, and $\widetilde{\mathcal{C}}$, $F$, and $U$ are given by (\ref{2.5}). By (\ref{4.16}), there are a $\widetilde{Sm}$-premotivic pseudofunctor $G:\widetilde{\mathcal{C}}\rightarrow {\rm Tri}^{\otimes}$ satisfying (Loc) and a pseudonatural equivalence $\delta:V\rightarrow G\circ U$ such that the diagram
  \[\begin{tikzcd}
    \mathcal{E}\arrow[rd,"V"']\arrow[r,"U"]&\widetilde{\mathcal{C}}\arrow[d,"G"]\arrow[ld,phantom,"{\rotatebox[origin=c]{0}{$\Leftrightarrow$}}_{\delta}",very near start]\arrow[r,"F"]&\mathcal{C}\\
    \,&{\rm Tri}^\otimes
  \end{tikzcd}\]
  of $2$-categories commutes. Then by (\ref{5.6}), for any morphism $\widetilde{f_1}:\widetilde{X_2}\rightarrow \widetilde{X_1}$ given by a commutative diagram
  \[\begin{tikzcd}
    (Z_1\arrow[d,"{\rm id}"]\arrow[r,"i_2"]&S_2)\arrow[d,"u_1"]\\
    (Z_1\arrow[r,"i_1"]&S_1)
  \end{tikzcd}\]
  in $\mathcal{C}$, $\widetilde{f_1}^*$ is an equivalence. Thus by (\ref{1.7}), there are a contravariant pseudofunctor $H:\mathcal{C}\rightarrow {\rm Tri}^{\otimes}$ and a pseudonatural equivalence $\alpha:G\rightarrow H\circ F$ such that the diagram
  \[\begin{tikzcd}
    \mathcal{E}\arrow[rd,"V"']\arrow[r,"U"]&\widetilde{\mathcal{C}}\arrow[d,"G"]\arrow[ld,phantom,"{\rotatebox[origin=c]{0}{$\Leftrightarrow$}}_{\delta}",very near start]\arrow[r,"F"]\arrow[rd,phantom,"{\rotatebox[origin=c]{0}{$\Leftrightarrow$}}_\alpha",very near start]&\mathcal{C}\arrow[ld,"H"]\\
    \,&{\rm Tri}^\otimes&\,
  \end{tikzcd}\]
  of $2$-categories commutes, and by (\ref{3.3}), $H$ is $F(\widetilde{Sm})$-premotivic where $\widetilde{Sm}$ is defined in (\ref{2.5}). Thus we have a commutative diagram
  \[\begin{tikzcd}
    \mathcal{E}\arrow[rd,"V"']\arrow[rr,"W"]\arrow[rrd,phantom,"{\rotatebox[origin=c]{0}{$\Leftrightarrow$}}_{\kappa}"]&&\mathcal{C}\arrow[ld,"H"]\\
    \,&{\rm Tri}^\otimes&\,
  \end{tikzcd}\]
  of $2$-categories where $W=F\circ U$ and $\kappa=(\alpha\circ U)\circ \delta$. The remaining of the first part of the statement is showing that $H$ is $Sm$-premotivic. For this, consider a smooth morphism $v_1:Z_2\rightarrow Z_1$ in $\mathcal{C}$, and choose a closed immersion $i_1:Z_1\rightarrow S_1$ where $S_1$ is in $\mathcal{E}$. Then by \cite[IV.18.1.1]{EGA}, Zariski locally on $Z_2$, there is a Cartesian diagram
  \[\begin{tikzcd}
    Z_2\arrow[d,"v_1"]\arrow[r,"i_2"]&S_2\arrow[d,"u_1"]\\
    Z_1\arrow[r,"i_1"]&S_1
  \end{tikzcd}\]
  in $\mathcal{C}$ such that $u_1$ is smooth. Thus Zariski locally on $Z_2$, $v_1$ is a $F(\widetilde{Sm})$-morphism, so Zariski locally on $Z_1$, $v_1^*$ has a left adjoint. Then (\ref{6.1}) implies that $v_1$ has a left adjoint since $Z_2$ is quasi-compact. Thus we have proven the first part of the statement.

  Consider a diagram
  \[\begin{tikzcd}
    V\arrow[d,"\epsilon"]\arrow[r,"\delta"]&G\circ U\arrow[r,"\alpha"]&H\circ W\\
    V'\arrow[r,"\delta'"]&G'\circ U\arrow[r,"\alpha'"]&H'\circ W
  \end{tikzcd}\]
  of pseudofunctors. By (\ref{4.13}), there is a $\widetilde{Sm}$-premotivic pseudonatural transformation $\beta:G\rightarrow G'$ such that the induced pseudonatural transformation $G\circ U\rightarrow G'\circ U$ makes the above diagram commutative, and then by (\ref{1.8}) and (\ref{3.13}), there is a $F(\widetilde{Sm})$-premotivic pseudonatural transformation $\gamma:H\rightarrow H'$ such that the induced pseudonatural transformation $H\circ W\rightarrow H'\circ W$ makes the above diagram commutes. By (\ref{6.4}), $\gamma$ is $Sm$-premotivic pseudonatural transformation as in the above paragraph. Hence the remaining to show that $\gamma$ is unique up to isomorphism. If $\gamma':H\rightarrow H'$ is another $Sm$-premotivic pseudonatural transformation such that the induced pseudonatural transformation $H\circ W\rightarrow H'\circ W$ makes the above diagram commutative, then by (\ref{4.13}), the two pseudonatural transformations
  \[\alpha'^{-1}\circ (\gamma\circ F)\circ \alpha:G\rightarrow G',\quad \alpha'^{-1}\circ (\gamma'\circ F)\circ \alpha:G\rightarrow G'\]
  are isomorphic. Thus by (\ref{1.8}), $\gamma$ and $\gamma'$ are isomorphic since we have the commutative diagrams
  \[\begin{tikzcd}
    G\arrow[d,"\alpha'^{-1}\circ (\gamma\circ F)\circ \alpha"']\arrow[r,"\alpha"]&H\circ F\arrow[d,"\gamma"]\\
    G'\arrow[r,"\alpha'"]&H'\circ F
  \end{tikzcd}
  \quad
  \begin{tikzcd}
    G\arrow[d,"\alpha'^{-1}\circ (\gamma'\circ F)\circ \alpha"']\arrow[r,"\alpha"]&H\circ F\arrow[d,"\gamma'"]\\
    G'\arrow[r,"\alpha'"]&H'\circ F
  \end{tikzcd}
  \]
  of pseudonatural transformations.
\end{proof}
\section{Proof of (\ref{0.7})}
\begin{thm}\label{7.1}
  Consider a commutative diagram
  \[\begin{tikzcd}
    \mathscr{S}^{sm}\arrow[rd,"V"']\arrow[rrd,phantom,"{\rotatebox[origin=c]{0}{$\Leftrightarrow$}}_{\kappa}"]\arrow[rr,"W"]&&\mathscr{S}\arrow[ld,"H"]\\
    &{\rm Tri}^{\otimes}&\;
  \end{tikzcd}\]
  where
  \begin{enumerate}[{\rm (i)}]
    \item $W$ denotes the inclusion functor,
    \item $H$ is a $\mathscr{P}$-premotivic pseudofunctor satisfying {\rm (Loc)},
    \item $V$ is a $\mathscr{P}'$-premotivic pseudofunctor satisfying {\rm (Loc)},
    \item $\kappa:V\rightarrow H\circ W$ is a pseudonatural equivalence.
  \end{enumerate}
  Then we have the followings.
  \begin{enumerate}[{\rm (1)}]
    \item If $V$ satisfies {\rm (B--5)}, then $H$ satisfies {\rm (B--5)}.
    \item If $V$ satisfies {\rm (B--6)}, then $H$ satisfies {\rm (B--6)}.
    \item If $V$ satisfies {\rm (B--6)} and {\rm (B--7)}, then $H$ satisfies {\rm (B--7)}.
    \item If $V$ satisfies {\rm (B--6)} and {\rm (B--8)}, then $H$ satisfies {\rm (B--8)}.
  \end{enumerate}
\end{thm}
\begin{proof}
  (1) Let $i:Z\rightarrow S$ be a closed immersion in $\mathscr{S}$ where $S$ is in $\mathscr{S}^{sm}$. Consider the Cartesian diagram
  \[\begin{tikzcd}
    Z\times \mathbb{A}^1\arrow[d,"i'"]\arrow[r,"p'"]&Z\arrow[d,"i"]\\
    S\times \mathbb{A}^1\arrow[r,"p"]&S
  \end{tikzcd}\]
  in $\mathscr{S}$ where $p$ denotes the projection. Then we have the commutative diagram
  \[\begin{tikzcd}
    p_\sharp'i'^*i_*'p'^*\arrow[d,"Ex"]\arrow[r,"ad'"]&p_\sharp'p'^*\arrow[ddd,"ad'"]\\
    i^*p_\sharp i_*'p'^*\arrow[d,"Ex^{-1}"]\\
    i^*p_\sharp p^*i_*\arrow[d,"ad'"]\\
    i^*i_*\arrow[r,"ad'"]&{\rm id}
  \end{tikzcd}\]
  of functors. Here, the left middle vertical arrow is defined and an isomorphism by (\ref{2.8}). The left top vertical arrow is an isomorphism by (B--3), and the lower horizontal and upper horizontal arrows are isomorphisms by (Loc). Thus the right vertical arrow is an isomorphism since the left bottom vertical arrow is an isomorphism by (B--5) for $V$.\\[4pt]
  (2) Let $i:Z\rightarrow S$ be a closed immersion in $\mathcal{C}$ such that $S\in \mathscr{S}^{sm}$. If $V$ satisfies (B--6), then there is an object $\tau$ of $H(S)$ such that $1_S(1)\otimes \tau\cong 1_S$. Then
  \[1_Z(1)\otimes i^*\tau\cong i^*(1_S(1)\otimes \tau)\cong i^*1_S\cong 1_Z,\]
  so $1_Z(1)$ is $\otimes$-invertible in $H(Z)$. Thus $H$ satisfies (B--6).\\[4pt]
  (3) Suppose that $V$ satisfies (B--6) and (B--7). Then by (2), $H$ satisfies (B--6). Let $i:Z\rightarrow S$ be a closed immersion in $\mathcal{C}$ such that $S\in \mathscr{S}^{sm}$. For any object $X$ of $\mathscr{S}$, let $\mathcal{F}_X$ denote the family consisting of objects
  \[f_\sharp 1_{X'}(d)[n]\]
  in $H(S)$ where $(d,n)\in \mathbb{Z}\times \mathbb{Z}$ and $g:X'\rightarrow X$ is a smooth morphism in $\mathscr{S}$. The remaining is to show that $\mathcal{F}_Z$ generates $H(Z)$. Let $L\rightarrow L'$ be a morphism in $H(Z)$ such that the induced homomorphism
  \[{\rm Hom}_{H(Z)}(K',L)\rightarrow {\rm Hom}_{H(Z)}(K',L')\]
  is an isomorphism for any object $K'$ of $\mathcal{F}_Z$. Then the induced homomorphism
  \[{\rm Hom}_{H(Z)}(i^*K,L)\rightarrow {\rm Hom}_{H(Z)}(i^*K,L')\]
  is an isomorphism for any object $K$ of $\mathcal{F}_S$, so the induced homomorphism
  \[{\rm Hom}_{H(S)}(K,i_*L)\rightarrow {\rm Hom}_{H(S)}(K,i_*L')\]
  is an isomorphism for any object $K$ of $\mathcal{F}_S$. Since $V$ satisfies (B--7), this means that the induced morphism $i_*L\rightarrow i_*L'$ in $H(S)$ is an isomorphism. Thus the morphism $L\rightarrow L'$ is an isomorphism since $i_*$ is fully faithful by (Loc). This proves that $H$ satisfies (B--7).

  (4) Suppose that $V$ satisfies (B--6) and (B--8). Then by (2), $H$ satisfies (B--6). Let $i:Z\rightarrow S$ be a closed immersion in $\mathscr{S}$ such that $S$ is in $\mathscr{S}^{sm}$. We will first show that the object $1_Z$ in $H(Z)$ is compact. By (Loc), we have the distinguished triangle
  \[j_\sharp j^*1_S\stackrel{ad'}\longrightarrow 1_S\stackrel{ad}\longrightarrow i_*i^*1_S\longrightarrow j_\sharp j^*1_S[1]\]
  in $H(S)$. Since $V$ satisfies (B--8), $j_\sharp j^*1_S$ and $1_S$ are compact. Thus $i_*i^*1_S$ is compact, so $1_Z\cong i^*1_S$ is compact since $i_*$ is fully faithful by (Loc). Let $g:Z'\rightarrow Z$ be a smooth morphism in $\mathscr{S}$. Then we have
  \[{\rm Hom}_{H(Z)}(g_\sharp 1_{Z'},-)={\rm Hom}_{H(Z')}(1_{Z'},g^*(-)),\]
  and it commutes with small sums since $g^*$ commutes with small sums and we have shown that $1_{Z'}$ is compact.
\end{proof}
\begin{thm}\label{7.2}
  In {\rm (\ref{6.3})}, if $V$ satisfies the axioms from {\rm (B--5)} to {\rm (B--8)}, then $H$ satisfies the Grothendieck six operations formalism in {\rm (\ref{0.7})}.
\end{thm}
\begin{proof}
  It follows from (\ref{7.1}), (\ref{2.13}), and (\ref{2.14}).
\end{proof}
\section{Applications}
\begin{none}\label{8.1}
  Consider the pseudofunctor
  \[{\rm DM}(-,\Lambda):\mathscr{S}^{sm}\rightarrow {\rm Tri}^{\otimes},\]
  which is the restriction of the pseudofunctor ${\rm DM}(-,\Lambda)$ to $\mathscr{S}^{sm}$ defined in \cite[11.1.1]{CD12}. By \cite[11.1.2]{CD12}, it satisfies the axioms from (B--1) to (B--6), and by \cite[11.1.6]{CD12}, it satisfies (B--7) and (B--8). Moreover, by \cite[11.4.2]{CD12}, it satisfies ${\rm (Loc}_i)$ for any closed immersion $i:Z\rightarrow S$ in $\mathscr{S}^{sm}$. Since ${\rm DM}(\emptyset,\Lambda)=0$, the pseudofunctor ${\rm DM}(-,\Lambda)$ satisfies (Loc). Thus by (\ref{6.3}), we can define the following.
\end{none}
\begin{df}\label{8.2}
  We denote by
  \[{\rm DM}^{loc}(-,\Lambda):\mathscr{S}\rightarrow {\rm Tri}^{\otimes}\]
  a $Sm$-premotivic pseudofunctor satisfying (Loc) that is an extension of ${\rm DM}(-,\Lambda):\mathscr{S}^{sm}\rightarrow {\rm Tri}^{\otimes}$ to $\mathscr{S}$, which is unique up to pseudonatural equivalence.
\end{df}
\begin{none}\label{8.6}
  By (\ref{7.2}), it satisfies the Grothendieck six operations formalism in \cite[2.4.50]{CD12}.
\end{none}
\begin{rmk}\label{8.3}
  Note that the definition of ${\rm DM}^{loc}(-,\Lambda)$ {\it depends} on the base scheme $T$. We usually assume that $T$ is the spectrum of a field or a Dedekind domain. When $T$ is changed by a smooth quasi-projective morphism $T'\rightarrow T$ of schemes, by definition, $DM^{loc}(-,\Lambda)$ is not changed.
\end{rmk}
\begin{none}\label{8.4}
  We can also apply (\ref{6.3}) to \'etale realization. Assume that $T$ is a noetherian scheme separated over ${\rm Spec}\,\mathbb{Z}[1/n]$ where $n>1$ is an integer. Consider the pseudofunctor
  \[{\rm DM}_{\et}(-,\Lambda):\mathscr{S}^{sm}\rightarrow {\rm Tri}^{\otimes}\]
  defined in \cite[2.2.4]{CD16}. Then consider the change of coefficients
  \[{\rm DM}(-,\mathbb{Z})\rightarrow {\rm DM}(-,\mathbb{Z}/n\mathbb{Z}),\]
  which is the Nisnevich version of the $Sm$-premotivic pseudonatural transformation in \cite[5.3.1]{CD16}. Consider also the $Sm$-premotivic pseudonatural transformation
  \[{\rm DM}(-,\mathbb{Z}/n\mathbb{Z})\rightarrow {\rm DM}_{\et}(-,\mathbb{Z}/n\mathbb{Z})\]
  defined in \cite[2.2.8]{CD16}. We also have the $Sm$-premotivic pseudonatural equivalence
  \[{\rm DM}_{\et}(-,\mathbb{Z}/n\mathbb{Z})\rightarrow {\rm D}_{\et}(-,\mathbb{Z}/n\mathbb{Z})\]
  by \cite[4.4.5]{CD16}. Here, for $S\in \mathscr{S}$, ${\rm D}_{\et}(S,\mathbb{Z}/n\mathbb{Z})$ denotes the derived category of sheaves of $\mathbb{Z}/n\mathbb{Z}$-modules on the small \'etale site of $S$.

  Then we obtain the \'etale realization $R_{\et,n}$, which is the composition
  \[{\rm DM}(-,\mathbb{Z})\rightarrow {\rm DM}(-,\mathbb{Z}/n\mathbb{Z})\rightarrow {\rm DM}_{\et}(-,\mathbb{Z}/n\mathbb{Z})\rightarrow {\rm D}_{\et}(-,\mathbb{Z}/n\mathbb{Z}).\]
  The pseudofunctor
  \[{\rm D}_{\et}(-,\mathbb{Z}/n\mathbb{Z}):\mathscr{S}\rightarrow {\rm Tri}^{\otimes}\]
  satisfies (Loc) and the axioms from (B--1) to (B--4), whose proofs are in \cite[tome 3]{SGA4}. Hence by (\ref{6.3}), we can define the following.
\end{none}
\begin{df}\label{8.5}
  Assume that $T$ is a noetherian scheme separated over ${\rm Spec}\,{\bf Z}[1/n]$ where $n>1$ is an integer. We denote by
  \[R_{\et,n}^{loc}:{\rm DM}^{loc}(-,\mathbb{Z})\rightarrow {\rm D}_{\et}(-,\mathbb{Z}/n\mathbb{Z})\]
  a $Sm$-premotivic pseudonatural transformation that is an extension of $R_{\et,n}$ to $\mathscr{S}$, which is unique up to isomorphism.
\end{df}
\begin{rmk}\label{8.7}
  Note that the remark (\ref{8.3}) is also applied for $R_{\et,n}^{loc}$.
\end{rmk}
\begin{none}
  Using (\ref{8.5}), we can also obtain the $\ell$-adic realization
  \[{\rm DM}^{loc}(-,\mathbb{Z})\rightarrow {\rm D}_{\et}(-,\mathbb{Z}_\ell)\]
  when $T$ is a noetherian scheme separated over ${\rm Spec}\,\mathbb{Z}[1/\ell]$.
\end{none}
\begin{none}\label{8.8}
  In the case that $\Lambda$ is a ${\bf Q}$-algebra, let us compare ${\rm DM}^{loc}(-,\Lambda)$ with ${\rm DA}_{\et}(-,\Lambda)$ (see \cite[3.1]{Ayo14} for the definition of ${\rm DA}_t(-,\Lambda)$ where $t$ is a topology on $\mathscr{S}$) when $T$ is the spectrum of a field or a Dedekind domain. They are equivalent on $\mathscr{S}^{sm}$ by \cite[16.2.22]{CD12}, and the pseudofunctor
  \[{\rm DA}(-,\Lambda):\mathscr{S}\rightarrow {\rm Tri}^{\otimes}\]
  satisfies (Loc) by \cite[6.2.2]{CD12}. Thus by (\ref{6.3}), there is a pseudonatural equivalence between ${\rm DM}^{loc}(-,\Lambda):\mathscr{S}\rightarrow {\rm Tri}^{\otimes}$ and ${\rm DA}_{\et}(-,\Lambda):\mathscr{S}\rightarrow {\rm Tri}^{\otimes}$ that is an extension of the pseudonatural equivalence on $\mathscr{S}^{sm}$, which is unique up to isomorphism.
\end{none}
\section{Orientation}
\begin{none}\label{9.1}
  Consider the $Sm$-premotivic pseudofunctor
  \[\dmeff(-,\Lambda):\mathscr{S}^{sm}\rightarrow {\rm Tri}^{\otimes},\]
  which is the restriction of the $Sm$-premotivic pseudofunctor $\dmeff(-,\Lambda)$ defined in \cite[11.1.1]{CD12}. By \cite[6.3.15]{CD12}, it satisfies ${\rm (Loc}_i)$ for any closed immersion $i:Z\rightarrow S$ in $\mathscr{S}^{sm}$. Then $\dmeff(-,\Lambda)$ satisfies (Loc) since $\dmeff(\emptyset,\Lambda)=0$. Thus by (\ref{6.3}), we can define the following.
\end{none}
\begin{df}\label{9.2}
  We denote by
  \[\dmeffloc(-,\Lambda):\mathscr{S}\rightarrow {\rm Tri}^{\otimes}\]
  a $Sm$-premotivic pseudofunctor satisfying (Loc) that is an extension of $\dmeff(-,\Lambda):\mathscr{S}^{sm}\rightarrow {\rm Tri}^{\otimes}$ to $\mathscr{S}$, which is unique up to pseudonatural equivalence.
\end{df}
\begin{none}\label{9.3}
  Consider the $Sm$-premotivic pseudofunctor
  \[\mathcal{H}_{\bullet}:\mathscr{S}\rightarrow {\rm Tri}^{\otimes}\]
  in \cite[Following paragraph of 3.2.12]{MV}. In \cite[11.2.16]{CD12}, if we use \cite[5.3.19]{CD12} instead of \cite[5.3.28]{CD12}, we obtain the $Sm$-premotivic pseudonatural transformation
  \[{\rm DA}_{Nis}(-,\Lambda)\rightarrow \dmeff(-,\Lambda)\]
  of pseudofunctors from $\mathscr{S}^{sm}$ to ${\rm Tri}^{\otimes}$. By the argument in \cite[5.3.35]{CD12}, we obtain the $Sm$-premotivic pseudonatural transformation
  \[\mathcal{H}_{\bullet}\rightarrow {\rm DA}_{Nis}(-\Lambda)\]
  of pseudofunctors from $\mathscr{S}^{sm}$ to ${\rm Tri}^{\otimes}$. Thus by composing these two, we obtain the $Sm$-premotivic pseudonatural transformation
  \[\mathcal{H}_{\bullet}\rightarrow \dmeff(-,\Lambda).\]
  of pseudofunctors from $\mathscr{S}^{sm}$ to ${\rm Tri}^{\otimes}$
  By \cite[3.2.21]{MV}, $\mathcal{H}_{\bullet}$ satisfies (Loc). Thus by (\ref{6.3}), the above one can be extended to a $Sm$-premotivic pseudonatural transformation
  \[\mathcal{H}_{\bullet}\rightarrow \dmeffloc(-,\Lambda)\]
  of pseudofunctors from $\mathscr{S}$ to ${\rm Tri}^{\otimes}$, and it is unique up to isomorphism.

  Consider the $Sm$-premotivic pseudonatural transformation
  \[\Sigma^{\infty}:\dmeff(-,\Lambda)\rightarrow {\rm DM}(-,\Lambda)\]
  of pseudofunctors from $\mathscr{S}^{sm}$ to ${\rm Tri}^{\otimes}$ that is the restriction of $\Sigma^{\infty}$ to $\mathscr{S}^{sm}$ defined in \cite[11.1.2.1]{CD12}. Since $\dmeffloc$ satisfies (Loc), by (\ref{6.3}), the above one can be extended to a $Sm$-premotivic pseudonatural transformation
  \[\Sigma^{\infty}:\dmeffloc(-,\Lambda)\rightarrow {\rm DM}^{loc}(-,\Lambda)\]
  of pseudofunctors from $\mathscr{S}$ to ${\rm Tri}^{\otimes}$, which is unique up to isomorphism.

  Now in \cite[11.3.1]{CD12}, we have the morphism $\mathfrak{c}_1:M({\bf P}^{\infty})\rightarrow {\bf 1}(1)[2]$ in $\dmeff({\rm Spec}\,{\bf Z},\Lambda)=\dmeffloc({\rm Spec}\,{\bf Z},\Lambda)$. Then using the $Sm$-premotivic pseudonatural transformations constructed in the above paragraphs, the argument in \cite[\S 11.3]{CD12} holds for $\dmeffloc(-,\Lambda)$ and ${\rm DM}^{loc}(-,\Lambda)$. Thus we can construct a function
  \[c_1:{\rm Pic}(X)\rightarrow {\rm Hom}_{{\rm DM}^{loc}(X,\Lambda)}(M(X),{\bf 1}(1)[2])\]
  for each object $X$ of $\mathscr{S}$ such that $c_1$ is functorial on $X$ and that $c_1(L):M({\bf P}^1)\rightarrow {\bf 1}(1)[2]$ is the canonical projection where $L$ is the canonical bundle of ${\bf P}^1$.

  Then by the argument in \cite[2.4.40]{CD12}, we have the following theorem.
\end{none}
\begin{thm}\label{9.4}
  The $Sm$-premotivic pseudofunctor ${\rm DM}^{loc}(-,\Lambda)$ admits an orientation in the sense of {\rm \cite[2.4.38]{CD12}}.
\end{thm}
\titleformat*{\section}{\center \scshape }

\end{document}